\documentclass[oneside,english]{amsart}
\usepackage[T1]{fontenc}
\usepackage[latin9]{inputenc}
\usepackage{color}
\usepackage{amsbsy}
\usepackage{amsthm}
\usepackage{amssymb}

\makeatletter
\numberwithin{equation}{section}
\numberwithin{figure}{section}

\usepackage{amsfonts}
\usepackage{amsmath}
\usepackage{amsfonts}\usepackage{amsthm}\usepackage{color}\usepackage{enumitem}\usepackage{bm}\usepackage{cancel}\usepackage{thmtools}
\usepackage{hyperref}
\usepackage{tikz-cd} 
\def\theenumi{\arabic{enumi}}

\def\theenumii{\alph{enumii}}
\def\p@enumii{\theenumi.}

\def\theenumiii{\arabic{enumiii}}
\def\p@enumiii{(\theenumi)(\theenumii)}

\def\p@enumiv{\p@enumiii.\theenumiii}
\newtheorem{theorem}{Theorem}[section]

\newtheorem{assumption}[theorem]{Assumption}

\newtheorem{claim}[theorem]{Claim}

\newtheorem{corollary}[theorem]{Corollary}

\newtheorem{fact}[theorem]{Fact}
\newtheorem{lemma}[theorem]{Lemma}

\newtheorem{proposition}[theorem]{Proposition}

\theoremstyle{definition}
\newtheorem{definition}[theorem]{Definition}
\newtheorem{notation}[theorem]{Notation}
\newtheorem{remark}[theorem]{Remark}
\newtheorem{example}[theorem]{Example}

\makeatother

\usepackage{babel}
\begin{document}
\title{On FC-central extensions of groups of intermediate growth}
\author{Tianyi Zheng}
\begin{abstract}
It is shown that FC-central extensions retain sub-exponential volume
growth. A large collection of FC-central extensions of the first Grigorchuk
group is provided by the constructions in the works of Erschler \cite{Erschler06}
and Kassabov-Pak \cite{KassabovPak}. We show that in these examples
subgroup separability is preserved. We introduce two new collections
of extensions of the Grigorchuk group. One collection gives first
examples of intermediate growth groups with centers isomorphic to
$\mathbb{Z}^{\infty}$; and the other provides groups with prescribed
oscillating intermediate growth functions. 
\end{abstract}

\thanks{The author is partially supported by the Alfred. P. Sloan Foundation
and a Hellman Fellowship from UCSD}
\maketitle

\section{Introduction}

The FC-center of a group $G$, denoted by $Z_{{\rm FC}}(G)$, consists
of all elements of $G$ which have finite conjugacy classes. We say
$\Gamma$ is an FC-central extension of $G$ if $G$ is a quotient
of $\Gamma$ and $\ker(\Gamma\to G)$ is contained in $Z_{{\rm FC}}(\Gamma)$.
A special case is that $\Gamma$ is a central extension of $G$, that
is, $\ker(\Gamma\to G)$ is contained in the center of $\Gamma$.
In this paper we study FC-central extensions of groups of sub-exponential
growth, in particular, of groups related to Grigorchuk groups. 

Let $G$ be a group equipped with a finite generating set $S$. Denote
by $d_{S}$ the graph distance on the Cayley graph of $(G,S)$. The
growth function $v_{G,S}(n)$ counts the number of elements in the
ball of radius $n$,
\[
v_{G,S}(n)=\left|\{\gamma\in G:\ d_{S}(\gamma,id)\le n\}\right|.
\]
A finitely generated group $G$ is of polynomial growth if there exists
$0<d<\infty$ and a constant $C>0$ such that $v_{G,S}(n)\le Cn^{d}$.
It is of exponential growth if $\lim_{n\to\infty}v_{G,S}(n)^{1/n}>1$.
If $v_{G,S}(n)$ is sub-exponential but not bounded by a polynomial,
we say $G$ is of intermediate growth. A group of polynomial growth
is virtually nilpotent by Gromov's theorem \cite{gromov}. The first
examples of groups of intermediate growth are constructed by Grigorchuk
in \cite{Grigorchuk84}. 

We show that FC-central extensions retain sub-exponential growth:

\begin{lemma}[= Lemma \ref{FC2}]\label{FC-1}

Let $\Gamma$ be a finitely generated FC-central extension of a group
$G$. If $G$ is of sub-exponential growth, then $\Gamma$ is of sub-exponential
growth as well. 

\end{lemma}

A well-known question of Rosenblatt \cite{Rosenblatt} asks if every
Cayley graph with exponential volume growth function admits a Lipschitz
embedding of the binary tree. Equivalently, the question asks if no
group of exponential growth can be \emph{supramenable}, see the characterizations
of supramenablity in \cite{KMR}. A group of sub-exponential growth
is supramenable, and results in \cite{Monod} imply that central extensions
of a group of sub-exponential growth are supramenable. Lemma \ref{FC-1}
shows that one can not find an example that answers Rosenblatt's question
negatively by taking a central extension of an intermediate growth
group. 

The space $\mathcal{M}_{k}$ of $k$-marked group the space of marked
groups $(G,S)$, where $S=(s_{1},\ldots,s_{k})$ is an ordered generating
$k$-tuple of $G$. Equivalently, $(G,S)$ can be viewed as an epimorphism
$\mathbf{F}_{k}\to G$, where $\mathbf{F}_{k}$ is the free group
on $k$ generators. The space $\mathcal{M}_{k}$ is endowed with the
Cayley topology, also referred to as the Cayley-Grigorchuk topology.
The Cayley topology is induced by the metric $d((G_{1},S_{1}),(G_{2},S_{2}))=2^{-k}$,
where $k$ is the maximal radius such that the ball around $id_{G_{1}}$
in the Cayley graph of $(G_{1},S_{1})$ is identical to the ball of
the same radius around $id_{G_{2}}$ in the Cayley graph of $(G_{2},S_{2})$.
The space $\mathcal{M}_{k}$ is first used by Grigorchuk in \cite{Grigorchuk84}
to study the family of Grigorchuk groups $\left(G_{\omega}\right)$,
$\omega\in\{0,1,2\}^{\mathbb{N}}$. 

Given a collection of $k$-marked groups $\left(\left(G_{i},S_{i}\right)\right)_{i\in I}$,
one can take the diagonal product $(\Delta,S)$ of this collection,
namely the subgroup of $\prod_{i\in I}G_{i}$ generated by diagonally
embedding the generating set of each $G_{i}$. Taking the diagonal
product of a sequence of marked group is a useful tool to construct
groups with \textquotedbl exotic\textquotedbl{} properties. Such a
construction has been performed by Erschler in \cite{Erschler06}
(in the language of piecewise automata), Kassabov and Pak in \cite{KassabovPak},
Brieussel and the author in \cite{BZ1,BZ2}. The goal is often to
produce a family of groups with certain desired properties: sub-exponential
growth groups with arbitrarily large F{\o}lner functions in \cite{Erschler06},
groups with oscillating volume growth in \cite{KassabovPak}, and
groups with prescribed random walk characteristics or metric embedding
distortion into Banach spaces in \cite{BZ1}. 

A connection between FC-central extensions and diagonal products is
the following. Suppose $\left(\left(G_{i},S_{i}\right)\right)_{i=1}^{\infty}$
is a sequence of marked groups in $\mathcal{M}_{k}$ such that each
$G_{i}$ is finite and moreover $(G_{i},S_{i})$ converges to $(G,S)$
in the Cayley topology, then the diagonal product of $\left(\left(G_{i},S_{i}\right)\right)_{i=1}^{\infty}$
is an FC-central extension of $G$. For instance, the groups constructed
in \cite{Erschler06} and \cite{KassabovPak} cited above are FC-central
extensions of the first Grigorchuk group arising this way. 

In general a central extension does not preserve residual finiteness.
By \cite{ErschlerNonResidual}, there exist central extensions of
the first Grigorchuk group $\mathfrak{G}$ that are not commensurable
up to finite kernels with any residually finite groups. 

Subgroup separability (also called locally extended residual finiteness,
LERF in abbreviation) is a strong form of residual finiteness. A group
$G$ is said to be \emph{subgroup separable} if every finitely generated
subgroup is an intersection of subgroups of finite index in $G$.
Subgroup separability is a powerful property and is only known to
hold for a few special classes of groups, notably, free groups \cite{Burns},
polycyclic groups \cite{Malcav}, some $3$-manifold groups \cite{Long-Niblo},
the first Grigorchuk group $\mathfrak{G}$ \cite{Grigorchuk-Wilson},
the Gupta-Sidki 3-group \cite{Garrido}. 

Regarding stability under group constructions, it is known that taking
free products (\cite{Burns}) and carefully controlled amalgamations
preserve subgroup separability, see \cite{Scott,Gitik}. In Section
\ref{sec:Subgroups} we show that the type of FC-central extension
of the first Grigorchuk group $\mathfrak{G}$ as in \cite{Erschler06}
and \cite{KassabovPak} preserves subgroup separability. The definition
of such extensions is explained at the beginning of Section \ref{sec:definition}.
Denote by $\mathbf{F}$ the free product $\left(\mathbb{Z}/2\mathbb{Z}\right)\ast\left(\mathbb{Z}/2\mathbb{Z}\times\mathbb{Z}/2\mathbb{Z}\right)$,
marked with the generating tuple $\mathbf{S}=\left(\mathbf{a},\mathbf{b},\mathbf{c},\mathbf{d}\right)$,
where $\mathbf{a},\mathbf{b},\mathbf{c},\mathbf{d}$ are involutions,
$\mathbf{F}=\left\langle \mathbf{a}\right\rangle \ast\left(\left\langle \mathbf{b}\right\rangle \times\left\langle \mathbf{c}\right\rangle \right)$,
$\mathbf{d}=\mathbf{bc}$. We say that a group $A$ is a marked quotient
of $\mathbf{F}$ if there is an epimorphism $\mathbf{F}\to A$ and
$A$ is marked with the generating tuple $(a,b,c,d)$ which is the
image of $\mathbf{S}$. 

\begin{theorem}[= Theorem \ref{separable}]\label{subgroups}

Let $\mathcal{L}=(A_{n})_{n=1}^{\infty}$ be a sequence of finite
marked quotients of $\mathbf{F}$. Denote by $\mathfrak{G}$ the first
Grigorchuk group and $\Gamma(\mathcal{L})$ the FC-central extension
of $\mathfrak{G}$ with $\mathcal{L}$ as input described in Definition
\ref{def:gamma(L)}. Then $\Gamma(\mathcal{L})$ is subgroup separable.

\end{theorem}

One can also determine the commensurability classes of finitely generated
subgroups of $\Gamma(\mathcal{L})$, see Theorem \ref{comm-1}. The
proof of Theorem \ref{subgroups} relies on the generalized wreath
recursion and branching structure in $\Gamma(\mathcal{L})$ (see Section
\ref{sec:definition}) and properties of subgroups of $\mathfrak{G}$
proven in \cite{Grigorchuk-Wilson} (see Subsection \ref{subsec:Ingredients}).
Groups constructed in \cite{KassabovPak} correspond to taking $\mathcal{L}=\left({\rm PSL}_{2}\left(\mathbb{Z}/n_{i}\mathbb{Z}\right)\right)$.
As a consequence of the main theorem of \cite{KassabovPak}, the collection
of FC-central extensions of $\mathfrak{G}$ which are of the form
$\Gamma(\mathcal{L})$ contains uncountably many groups of intermediate
growth with pairwise non-comparable growth functions. 

The second half of the paper is devoted to constructing more examples
of extensions of intermediate growth groups and studying their properties.
We start with permutation wreath extensions of the first Grigorchuk
group $\mathfrak{G}$, instead of $\mathfrak{G}$ as in previous examples.
The reasons for this choice will be clear in the context: in Section
\ref{sec:central}, we need the Schur multipliers to contain copies
of $\mathbb{Z}$ as direct summands; and in Section \ref{sec:volume}
we rely on the structure of permutation wreath products to provide
sharp volume lower bounds. 

Let $G$ be a group acting on a countable set $X$ and $L$ be a group.
The \emph{permutation wreath product} $W=L\wr_{X}G$ is the semi-direct
product $\oplus_{X}L\rtimes\Gamma$, where $G$ acts on $\oplus_{X}L$
by permuting coordinates. The permutation wreath extension of $\mathfrak{G}$
of the form $W=A\wr_{\mathcal{S}}\mathfrak{G}$, where $A$ is a group
and $\mathcal{S}$ is the orbit of the right-most ray $1^{\infty}$
of the rooted binary tree under the action of $\mathfrak{G}$, is
first considered in \cite{BE1}. 

The goal of Section \ref{sec:central} is to construct groups of intermediate
growth with center isomorphic to $\mathbb{Z}^{\infty}$, where $\mathbb{Z}^{\infty}$
denotes the direct sum of countably many copies of $\mathbb{Z}$.
We take a specific countable $\mathfrak{G}$-set $X$ and construct
a sequence of marked groups $\Gamma_{n}$ which converges to $\mathbb{Z}\wr_{X}\mathfrak{G}$
in the Cayley topology, such that the center of diagonal product of
this sequence is isomorphic to $\mathbb{Z}^{\infty}$. The question
whether there is a group of intermediate growth whose center contains
$\mathbb{Z}^{\infty}$ is asked by Bartholdi and Erschler (personal
communication). 

\begin{theorem}\label{large-center}

There exists a finitely generated torsion-free group $\Gamma$ of
intermediate growth whose center $Z(\Gamma)$ is isomorphic to $\mathbb{Z}^{\infty}.$

\end{theorem}

In Section \ref{sec:construction} we take the diagonal products of
suitable finite marked groups which result in FC-central extensions
of $W=(U\times V)\wr_{\mathcal{S}}\mathfrak{G}$, where $U,V$ are
finite group and $\mathcal{S}$ is the orbit of $1^{\infty}$ under
the action of $\mathfrak{G}$. We show that the volume growth functions
of such diagonal products can be estimated up to good precision, see
Theorem \ref{thm:main}. In particular, such FC-central extensions
of $W$ allow to show the following result on prescribed growth function. 

For the first Grigorchuk group $\mathfrak{G}$, the upper bound in
\cite{Bartholdi98} states that $v_{\mathfrak{G},S}(n)\le\exp\left(Cn^{\alpha_{0}}\right)$,
where $\alpha_{0}=\frac{\log2}{\log\lambda_{0}}\approx0.7674$, and
$\lambda_{0}$ is the positive root of the polynomial $X^{3}-X^{2}-2X-4$.
The volume lower bound in \cite{EZ2} shows that $v_{\mathfrak{G},S}(n)\ge\exp\left(c_{\epsilon}n^{\alpha_{0}-\epsilon}\right)$
for any $\epsilon>0$. The upper and lower bounds together imply that
\[
\lim_{n\to\infty}\frac{\log\log v_{\mathfrak{G},S}(n)}{\log}=\alpha_{0},
\]
that is, the volume exponent of $\mathfrak{G}$ exists and is equal
to $\alpha_{0}$.

Recall that a function $f:\mathbb{N}\to\mathbb{N}$ is sub-additive
if $f(n+m)\le f(n)+f(m)$ and it is sublinear if $\lim_{n\to\infty}f(n)/n=0$. 

\begin{theorem}\label{growth-pres}

Denote by $\alpha_{0}$ the volume exponent of the first Grigorchuk
group $\mathfrak{G}$. There exists an absolute constant $C>0$ such
that the following is true. Let $f:\mathbb{N}\to\mathbb{N}$ be any
sublinear, non-decreasing, subadditive function such that $f(1)=1$
and 
\[
f(n)\ge n^{\alpha_{0}}\mbox{ for all }n\in\mathbb{N}.
\]
Then there exists a group $\Delta$ equipped with a finite generating
set $T$ such that $\Delta$ is an FC-central extension of $W=(\mathbb{Z}/2\mathbb{Z})^{3}\wr_{\mathcal{S}}\mathfrak{G}$
and for all $n\in\mathbb{N}$,
\[
\exp\left(\frac{1}{C}f(n)\right)\le v_{\Delta,T}(n)\le\exp\left(Cf(n)\right).
\]

\end{theorem}

Note that $\log v_{G,S}$ is necessarily a non-decreasing, subadditive
function. The limitation of Theorem \ref{growth-pres} is that it
is based on extensions of $W$, thus the volume functions obtained
are always at least the growth function of $W$ (which is equivalent
to $e^{n^{\alpha_{0}}}$ by \cite{BE1}). It is an important open
problem whether there exists intermediate growth groups with (lower)
volume exponent strictly less than $\alpha_{0}$. This problem is
closely related to the Gap Conjecture of Grigorchuk, see the survey
\cite{GrigorchukGap}. 

The groups in Theorem \ref{growth-pres} are similar to, but not the
same kind as those considered in \cite[Theorem 10.1]{KassabovPak}.
The construction is designed so that volume growth of the diagonal
product is controlled by what we call the traverse fields of words
(defined in Section \ref{sec:tr-metric}). The traverse fields are
compatible with the word recursion, which allows us to apply the norm
contraction inequality to control the volume growth, see Section \ref{sec:tr-metric}
and \ref{sec:volume}. 

\subsection*{Relation to prior works}

Theorem \ref{growth-pres} can be viewed as a further development
of earlier works of Bartholdi and Erschler \cite{BE2}, Brieussel
\cite{BrieusselOsi}, Kassabov and Pak \cite{KassabovPak}, in the
sense that it combines some aspects of the constructions in these
works. The main theorem of \cite{BE2} states that given a function
$\phi:\mathbb{R}_{+}\to\mathbb{R}_{+}$ such that 
\begin{equation}
\phi(2R)\le2\phi(R)\le\phi\left(2^{\frac{1}{\alpha_{0}}}R\right)\mbox{ for all }R\mbox{ large enough},\label{eq:BE-assum}
\end{equation}
then there exists a group $\Gamma$ with finite generating set $S$
and a constant $C>0$ such that $e^{\phi(n/C)}\le v_{\Gamma,S}(n)\le e^{\phi(Cn)}$.
The group $\Gamma$ is taken to be $(\mathbb{Z}/2\mathbb{Z})\wr_{\mathcal{S}}G_{\omega}$,
where $\omega\in\{0,1,2\}^{\mathbb{N}}$ is chosen carefully according
to the given function $\phi$. The groups considered in \cite{BrieusselOsi}
are the same kind as in \cite{BE2}, where estimates are given in
terms of upper and lower growth exponents. The use of permutation
wreath products in Theorem \ref{growth-pres} is motivated by the
sharp volume estimates in \cite{BE1,BE2}. Note that the collection
of admissible functions in Theorem \ref{growth-pres} is larger than
in \cite{BE2} cited above. Indeed, the assumption $2\phi(R)\le\phi\left(2^{\frac{1}{\alpha_{0}}}R\right)$
in (\ref{eq:BE-assum}) implies not only $\phi(R)\ge cR^{\alpha_{0}}$,
but also $\phi(R)/R^{\alpha_{0}}$ is non-decreasing. Consider a function
that oscillates between $n^{\alpha_{0}}$ and some larger sub linear
function $g(n)$, for example, $f$ is a sub-additive function such
that 
\[
f(n)\ge n^{\alpha_{0}},\ \liminf_{n\to\infty}\frac{f(n)}{n^{\alpha_{0}}}<\infty\ \mbox{and }\limsup_{n\to\infty}\frac{f(n)}{n/\log n}=\infty.
\]
This type of oscillating growth functions is not covered by the results
in \cite{BE2}, but allowed in Theorem \ref{growth-pres}. The main
theorem in \cite{KassabovPak} states that given sub-exponential functions
$f_{1},f_{2},g_{1},g_{2}$ satisfying 
\[
f_{1}\succcurlyeq f_{2}\succcurlyeq g_{1}\succcurlyeq g_{2}=v_{\mathfrak{G},S}
\]
and some additional conditions, there exists a group $\Gamma$ (which
is an FC-central extension of $\mathfrak{G}$ as mentioned earlier)
with a finite generating set $S$, such that its growth function satisfies
$g_{2}(n)<v_{\Gamma,S}(n)<f_{1}(n)$ and $v_{\Gamma,S}$ takes values
in the intervals $\left[g_{2}(n),g_{1}(n)\right]$ and $\left[f_{2}(n),f_{1}(n)\right]$
infinitely often. For an illustration such oscillating growth functions,
see \cite[Figure 1]{KassabovPak}. The examples in Theorem \ref{growth-pres}
can exhibit oscillating behaviors similar to \cite{KassabovPak} and
at the same time the log of the volume functions can be estimated
up to an absolute constant. We mention that techniques in Section
\ref{sec:volume} can be used to improve volume upper bounds of groups
considered in \cite{KassabovPak}. 

Another feature of the family of FC-central extensions of $W$ in
Theorem \ref{growth-pres} is that simple random walks on these groups
exhibit rich behavior. The technique in the proof of Theorem \ref{growth-pres}
is relevant for random walks: the random walk is controlled by the
traverse field of random words (instead of deterministic words in
the volume calculations). In particular, one can show that with appropriate
choice of parameters, the entropy of simple random walk on $\Delta$
can be equivalent to any prescribed sublinear function $f$ with $f(n)\ge n^{1/2}$.
Estimates of random walk characteristics on $\Delta$ will be discussed
elsewhere. 

\subsection*{Organization of the paper}

In Section 2 we collect some background on permutation wreath products,
Grigorchuk groups and the space of marked groups. In Section 3 we
show that FC-central extensions preserve sub-exponential volume growth.
Section 4 contains a basic lemma on growth of diagonal products and
discusses direct sum structure in the diagonal product. The rest of
the paper contains three parts, which can be read independently. Part
I consists of Section 5 and 6 on the FC-central extensions of the
Grigorchuk group $\mathfrak{G}$ of the form $\Gamma(\mathcal{L})$.
In Section 5 we present the definition of $\Gamma(\mathcal{L})$ and
explain the wreath recursions on the formal level. Section 6 is devoted
to the study of finitely generated subgroups of $\Gamma(\mathcal{L})$,
where we determine the commensurability classes of finitely generated
subgroups and show subgroup separability. Part II consists of Section
7 where we construct finitely generated groups of intermediate growth
with center isomorphic to $\mathbb{Z}^{\infty}$. Part III consists
of Section 8, 9 and 10 on extensions of $(U\times V)\wr_{\mathcal{S}}\mathfrak{G}$.
In Section 8 we describe the construction and explain the recursive
structure in these groups. In Section 9 we introduce the traverse
fields of words and show that they are compatible with the wreath
recursion. Section 10 contains volume growth estimates for groups
described in Section 8. We first use the traverse fields and the norm
contraction inequality to estimate growth in each factor group, then
combine the results to obtain estimates for the growth of the diagonal
product. Finally we present an elementary lemma for approximation
of prescribed functions and show explicitly that given a function
$f$ as in Theorem \ref{growth-pres}, how to choose appropriate parameters
and find a diagonal product whose log volume function is equivalent
to $f$. 

\subsection*{Acknowledgment}

We thank Nicholas Monod for explaining the results on supramenablity
in \cite{KMR,Monod}. We thank Gideon Amir and Anna Erschler for discussions
at the early stage of this work. 

\section{Preliminaries}

\subsection{Permutation wreath product }

Let $G$ be a group acting on a set $X$ from the right and $L$ be
a group. The \emph{permutation wreath product} $W=L\wr_{X}G$ is defined
as the semi-direct product $\oplus_{X}L\rtimes\Gamma$, where $G$
acts on $\oplus_{X}L$ by permuting coordinates. The \emph{support
${\rm supp}f$ }of $f:X\to L$ consists of points of $x\in X$ such
that $f(x)\neq id_{L}$. Elements of $\oplus_{X}L$ are viewed as
finitely supported functions $X\to L$. The action of $G$ on $\oplus_{X}L$
is It follows that ${\rm supp}(g\cdot f)=({\rm supp}f)\cdot g^{-1}$.
Elements in $W=L\wr_{X}G$ are recorded as pairs $(f,g)$ where $f\in\oplus_{X}L$
and $g\in G$. Multiplication in $W$ is given by 
\[
(f_{1},g_{1})(f_{2},g_{2})=\left(f_{1}\left(g_{1}\cdot f_{2}\right),g_{1}g_{2}\right),\mbox{ where }g\cdot f(x)=f(x\cdot g).
\]
When $G$ and $L$ are finitely generated, $W=L\wr_{X}G$ is finitely
generated as well. For more information on Cayley graphs of permutation
wreath products, see \cite[Section 2]{BE1}. 

We use the notation $\delta_{x}^{\gamma}$ for the function $f:X\to L$
such that $f(x)=\gamma$ and $f(y)=id_{L}$ for all $y\neq x$. The
additive notation $\delta_{x}^{\gamma_{1}}+\delta_{y}^{\gamma_{2}}$,
$x\neq y$, means the function $f:X\to L$ such that $f(x)=\gamma_{1}$,
$f(y)=\gamma_{2}$ and $f(z)=id_{L}$ for $z\notin\{x,y\}$. 

\subsection{Grigorchuk groups\label{subsec:Grigorchuk-groups}}

The groups $\left\{ G_{\omega}\right\} $ were introduced by Grigorchuk
in \cite{Grigorchuk84} as first examples of groups of intermediate
growth. We recall below the definition of these groups. See also the
exposition in the book \cite[Chapter VIII]{delaHarpeBook}.

The spherically symmetric rooted tree $\mathsf{T}_{\mathbf{d}}$ is
the tree with vertices $v=v_{1}\ldots v_{n}$ with each $v_{j}\in\{0,1,\ldots,d_{j}-1\}$.
The root is denoted by the empty sequence $\emptyset$. Edge set of
the tree is $\left\{ \left(v_{1}\ldots v_{n},v_{1}\ldots v_{n}v_{n+1}\right)\right\} $.
The index $n$ is called the depth or level of $v$, denoted $\left|v\right|=n$.
Denote by $\mathsf{T}_{\mathbf{d}}^{n}$ the finite subtree of vertices
up to depth $n$ and $\mathsf{L}_{n}$ the vertices of level $n$.
The boundary $\partial\mathsf{T}_{\mathbf{d}}$ of the tree $\mathsf{T}_{\mathbf{d}}$
is the set of infinite rays $x=v_{1}v_{2}\ldots$ with $v_{j}\in\{0,1,\ldots,d_{j}-1\}$
for each $j\in\mathbb{N}$. Throughout the paper, denote by $\mathsf{T}$
the rooted binary tree, that is, $\mathbf{d}$ is the constant sequence
with $d_{i}=2$. 

Let $\{0,1,2\}$ be the three non-trivial homomorphisms from the $4$-group
$(\mathbb{Z}/2\mathbb{Z})\times(\mathbb{Z}/2\mathbb{Z})=\{id,b,c,d\}$
to the $2$-group $\mathbb{Z}/2\mathbb{Z}=\{id,a\}$. They are ordered
in such a way that $0,1,2$ vanish on $d,c,b$ respectively. For example,
$0$ maps $id,d$ to $id$ and $b,c$ to $a$. Let $\Omega=\{0,1,2\}^{\infty}$
be the space of infinite sequences over letters $\{0,1,2\}$. The
space $\Omega$ is endowed with the shift map $\mathfrak{s}:\Omega\to\Omega$,
$\mathfrak{s}(\omega_{0}\omega_{1}\ldots)=\omega_{1}\omega_{2}\ldots$.

Given an $\omega\in\Omega$, the Grigorchuk group $G_{\omega}$ acting
on the rooted binary tree $\mathsf{T}$ is generated by $\left\{ a,b_{\omega},c_{\omega},d_{\omega}\right\} $,
where $a=(id,id)\varepsilon$, $\varepsilon$ transposes $0$ and
$1$, and the automorphisms $b_{\omega},c_{\omega},d_{\omega}$ are
defined recursively according to $\omega$ as follows. The wreath
recursion sends
\[
\psi_{n}:G_{\mathfrak{s}^{n}\omega}\to G_{\mathfrak{s}^{n+1}\omega}\wr_{\{0,1\}}\mathfrak{S}_{2}
\]
by
\begin{align*}
\psi_{n}\left(b_{\mathfrak{s}^{n}\omega}\right) & =\left(\omega_{n}(b),b_{\mathfrak{s}^{n+1}\omega}\right),\\
\psi_{n}\left(c_{\mathfrak{s}^{n}\omega}\right) & =\left(\omega_{n}(c),c_{\mathfrak{s}^{n+1}\omega}\right),\\
\psi_{n}\left(d_{\mathfrak{s}^{n}\omega}\right) & =\left(\omega_{n}(d),d_{\mathfrak{s}^{n+1}\omega}\right).
\end{align*}
The string $\omega$ determines the portrait of the automorphisms
$b_{\omega},c_{\omega},d_{\omega}$ by the recursive definition above.
The \emph{first Grigorchuk group} corresponds to the periodic sequence
$\omega=(012)^{\infty}$ and is often denoted as $G_{012}$. 

By \cite{Grigorchuk84}, if $\omega$ is eventually constant, then
$G_{\omega}$ has polynomial growth; otherwise $G_{\omega}$ is of
intermediate growth. When $\omega$ contains infinitely many of each
of the three symbols $\{0,1,2\}$, $G_{\omega}$ is an infinite torsion
group. 

\subsection{Marked groups and the Cayley topology}

Let $k\in\mathbb{N}$. We say $(G,S)$ is a\emph{ $k$-marked group}
if $S=\left(s_{1},\ldots,s_{k}\right)$ is an ordered $k$-tuple of
elements in $G$ that generates $G$. We refer to $S$ as a \emph{$k$-marking}
of $G$. Equivalently one can think of a $k$-marking of $G$ as a
quotient map $\pi:{\bf F}_{k}\to G$, where ${\bf F}_{k}$ is the
free group on free generators $\{{\bf s}_{1},\ldots,{\bf s}_{k}\}$.
We say a map $\psi:(G_{1},S_{1})\to(G_{2},S_{2})$ is a \emph{$k$-marked
group homomorphism}, where $S_{i}=\left(s_{1}^{(i)},\ldots,s_{k}^{(i)}\right)$
for $i\in\{1,2\}$, if $\psi$ is a homomorphism $G_{1}\to G_{2}$
such that $\psi\left(s_{j}^{(1)}\right)=s_{j}^{(2)}$ for all $j\in\{1,\ldots,k\}$. 

Given a collection of $k$-marked groups $\left\{ \left(G_{i},S_{i}\right)\right\} _{i\in I}$,
the diagonal product $(\Delta,S)$ of this collection is the subgroup
of $\prod_{i\in I}G_{i}$ generated by diagonally embedding the generating
set of each $G_{i}$. Equivalently, let ${\bf F}_{k}$ be the free
group on free generators ${\bf S}=\{{\bf s}_{1},\ldots,{\bf s}_{k}\}$
and identify the $k$-marking $S_{i}$ of $G_{i}$ as the quotient
map $\pi_{i}:{\bf F}_{k}\to G_{i}$ sending ${\bf S}$ to $S_{i}$,
then the diagonal product $(\Delta,S)$ of the collection $\left\{ \left(G_{i},S_{i}\right)\right\} _{i\in I}$
is given by
\[
\Delta={\bf F}_{k}/\cap_{i\in I}\ker\left(\pi_{i}:{\bf F}_{k}\to G_{i}\right),\ S=\mbox{image of }{\bf S}.
\]
For some basic properties of diagonal products we refer to \cite[Section4]{KassabovPak}.
In particular, let $(\Delta,S)$ be the diagonal product of a sequence
of $k$-marked groups that converge to $(G,S)$ in the Cayley topology,
then $G$ is a $k$-marked quotient of $\Delta$. 

\section{Growth of FC-central extensions\label{sec:dfc}}

In this section we show that sub-exponential growth is preserved under
FC-central extensions. We will need the following classical result
of B.H. Neumann \cite{Neumann} which states that if $N$ is an FC-group
(i.e., $N=Z_{{\rm FC}}(N)$), then 
\begin{itemize}
\item the elements of finite order form a characteristic subgroup $P$ of
$N$ and $P$ contains the derived subgroup $N'$; 
\item every finite set of finite order elements is contained in a finite
normal subgroup of $N$;
\item moreover, if $N$ is finitely generated then $P$ is finite.
\end{itemize}
The key step in the proof of the following lemma is the rearranging
of products in Claim \ref{conj-prod}. 

\begin{lemma}\label{FC2}

Let $G$ be a group of sub-exponential growth. Suppose $\Gamma$ is
a finitely generated FC-central extension of $\Gamma$, that is, 
\[
1\to N\to\Gamma\overset{\pi}{\to}G\to1\ \mbox{where }N\subseteq Z_{{\rm FC}}(\Gamma).
\]
Then $\Gamma$ has sub-exponential growth as well. 

\end{lemma}

\begin{proof}

Given two sets $A,B\subseteq\Gamma$, denote by $AB$ the product
set $\{gh:g\in A,h\in B\}$. 

Take a symmetric finite generating set of $\Gamma$, adjust the generating
set if necessary, we may assume that it is of the form $T=\{t_{1},\ldots,t_{\ell}\}\cup\{n_{1},\ldots,n_{p}\}$,
where the projections $\pi(t_{i})$, $1\le i\le\ell$ are pairwise
distinct non-trivial elements of $G$ and $n_{i}\in N$ for $1\le i\le p$.
Write $T_{1}=\{t_{1},\ldots,t_{\ell}\}$ and denote by $\Gamma_{1}$
the subgroup generated by $T_{1}$. Denote by $Q_{0}$ the union of
the conjugates of $\left\{ n_{1},\ldots,n_{p}\right\} $ in $\Gamma$.
Note that $B_{\Gamma,T}(id,n)\subseteq Q_{0}^{n}B_{\Gamma_{1},T_{1}}(id,n)$.
Since $Q_{0}$ is a finite set in the FC-center of $\Gamma$, by Neumann's
result, $\left\langle Q_{0}\right\rangle $ is virtually abelian.
Thus to show that $\Gamma$ is of sub-exponential growth, it suffices
to show for the subgroup $\Gamma_{1}$. In what follows we replace
$\Gamma$ by $\Gamma_{1}$ and assume the generating set $T=\{t_{1},\ldots,t_{\ell}\}$
is symmetric and the projections $\pi(t_{i})$, $1\le i\le\ell$ are
pairwise distinct non-trivial elements of $G$. 

Equip $G$ by the generating set $\bar{T}=\pi(T)$. In what follows
balls in $\Gamma$ ($G$ resp.) are with respect to word length $\left|\cdot\right|_{T}$
( $\left|\cdot\right|_{\bar{T}}$ resp.). Throughout the proof, fix
a choice of section $s:G\to\Gamma$ such that $s\left(id_{G}\right)=id_{\Gamma}$
and on the generating set $\bar{T}$, $s(\pi(t_{i}))=t_{i}$. For
each integer $k\ge2$, the section $s$ gives rise to a map 
\begin{align*}
\beta_{k} & :G^{k}\to N\\
 & (g_{1},\ldots,g_{k})\mapsto s(g_{1})\ldots s(g_{k})s(g_{1}\ldots g_{k})^{-1}.
\end{align*}

Denote by $A(k,r)$ the image of $k$ copies of $B_{G}(id,r)$ under
the map $\beta_{k}$, that is,
\[
A(k,r):=\beta_{k}\left(B_{G}(id,r)\times\ldots\times B_{G}(id,r)\right).
\]
By its definition, we have that the size of the set $A_{k}^{r}$ is
bounded by
\begin{equation}
|A(k,r)|\le\left|B_{G}(id,r)\right|^{k}.\label{eq:dom-bound}
\end{equation}
Since an element $\gamma\in B_{\Gamma}(id,k)$ can be written as 
\[
\gamma=t_{i_{1}}\ldots t_{i_{|\gamma|}}=s\left(\pi\left(t_{i_{1}}\right)\right)\ldots s\left(\pi\left(t_{i_{|\gamma|}}\right)\right)s(\pi(\gamma))^{-1}\cdot s(\pi(\gamma))\in A(k,1)s\left(\pi(\gamma)\right),
\]
it follows that 
\begin{equation}
B_{\Gamma}(id,k)\subseteq A(k,1)s\left(B_{G}(id,k)\right).\label{eq:k-1}
\end{equation}

Denote by $C(k,r)$ the union of the $\Gamma$-conjugates of $A(k,r)$,
that is, 
\[
C(k,r):=\cup_{g\in\Gamma}g^{-1}A(k,r)g.
\]
Since $A(k,r)$ is a finite subset in $N$ and $N$ is contained in
the FC-center of $\Gamma$, the set $C(k,r)$ is a finite subset in
$N$ as well. We have the following inclusion relation of sets:

\begin{claim}\label{conj-prod}

Let $j$ be an integer such that $j|n$. Then 
\[
A(n,r)\subseteq C(j,r)^{n/j}A(n/j,rj).
\]

\end{claim}

\begin{proof}[Proof of the Claim]

Recall that an element $\gamma$ in $A(n,r)$ can be written in the
form $\gamma=s(g_{1})\ldots s(g_{n})s(g_{1}\ldots g_{n})^{-1}$, where
each $g_{j}$ is an element in $G$ satisfying $|g_{j}|_{\bar{T}}\le r$.
Divide into blocks of size $j$ and write for each $1\le k\le n/j$,
\begin{align*}
a_{k} & =s(g_{j(k-1)+1})\ldots s(g_{jk})s\left(g_{j(k-1)+1}\ldots g_{jk}\right)^{-1},\\
x_{k} & =s\left(g_{j(k-1)+1}\ldots g_{jk}\right),
\end{align*}
and let $w_{0}=id$, 
\[
w_{k}=x_{1}\ldots x_{k}.
\]
Then we have 
\begin{align}
\gamma & =s(g_{1})\ldots s(g_{n})s(g_{1}\ldots g_{n})^{-1}\nonumber \\
 & =\left(\prod_{k=1}^{n/j}a_{k}x_{k}\right)s(g_{1}\ldots g_{n})^{-1}\\
 & =\left(\prod_{k=1}^{n/j}w_{k-1}a_{k}w_{k-1}^{-1}\right)w_{n/j}s(g_{1}\ldots g_{n})^{-1}.\label{eq:prod1}
\end{align}
By definitions of the sets, we have that $a_{k}\in A(j,r)$, its conjugate
$w_{k-1}a_{k}w_{k-1}^{-1}\in C(j,r)$, and 
\[
w_{n/j}s(g_{1}\ldots g_{n})^{-1}=\left(\prod_{k=1}^{n/j}s\left(g_{j(k-1)+1}\ldots g_{jk}\right)\right)s(g_{1}\ldots g_{n})^{-1}\in A(n/j,rj).
\]
The claim then follows from (\ref{eq:prod1}). 

\end{proof}

Now we return to the proof of the lemma. Since $G$ is assumed to
have sub-exponential growth, given any $\epsilon>0$, there exists
a $j_{\epsilon}\in\mathbb{N}$ such that 
\begin{equation}
v_{G,\bar{T}}(n)\le e^{\epsilon n/3}\mbox{ for all }n\ge j_{\epsilon}.\label{eq:G-n rad}
\end{equation}
In Claim \ref{conj-prod}, take $r=1$ and $j=j_{\epsilon}$. By Neumann's
result, the finite set $C(j_{\epsilon},1)$ generates a virtually
abelian group. Since finitely generated virtually abelian groups are
either finite or of polynomial growth, there exists constants $K,d>0$,
depending only on $C(j_{\epsilon},1)$, such that for all $m\ge1$,
we have
\[
\left|C(j_{\epsilon},1)^{m}\right|\le Km^{d}.
\]
Recall that by (\ref{eq:dom-bound}), we have $\left|A(m,j)\right|\le v_{G,\bar{T}}(j)^{m}.$
Combine the two parts, by Claim \ref{conj-prod}, we have that for
all $n\in j_{\epsilon}\mathbb{N}$,
\begin{align*}
\left|A(n,1)\right| & \le\left|C(j_{\epsilon},1)^{n/j_{\epsilon}}\right|\left|A(n/j_{\epsilon},j_{\epsilon})\right|\\
 & \le K\left(\frac{n}{j_{\epsilon}}\right)^{d}v_{G,\bar{T}}(j_{\epsilon})^{n/j_{\epsilon}}\\
 & \le K\left(\frac{n}{j_{\epsilon}}\right)^{d}e^{\epsilon n/3}.
\end{align*}
Combined with (\ref{eq:k-1}), we have that for all $n\in j_{\epsilon}\mathbb{N}$,
\[
\left|B_{\Gamma}(id,n)\right|\le\left|A(n,1)\right|\left|B_{G}(id,n)\right|\le K\left(\frac{n}{j_{\epsilon}}\right)^{d}e^{2\epsilon n/3},
\]
where $K$ and $d$ are constants only depending on $C(j_{\epsilon},1)$.
In particular, it implies that 
\[
\liminf_{n\to\infty}\frac{1}{n}\log\left|B_{\Gamma}(id,n)\right|\le\frac{2}{3}\epsilon.
\]
Since this is true for every $\epsilon>0$, we conclude that $\Gamma$
has sub-exponential growth. 

\end{proof}

\begin{remark}

In the setting of Lemma \ref{FC2}, if $G$ is of polynomial growth,
then $\Gamma$ is of polynomial growth too. Indeed, by Gromov's theorem
\cite{gromov}, a polynomial growth group is virtually nilpotent;
and any finitely generated FC-central extension of a virtually nilpotent
group is virtually nilpotent, see e.g., \cite{Losert}. 

\end{remark}

\begin{example}

Let $G$ be a finitely generated group marked with generating $k$-tuple
$S$, that is we have an epimorphism ${\bf F}_{k}\to G$. Let $R$
be the kernel $R=\ker\left({\bf F}_{k}\to G\right)$. Then the group
$\Gamma={\bf F}_{k}/\left[{\bf F}_{k},R\right]$ is a finitely generated
central extension of $G$. By Lemma \ref{FC2}, $\Gamma={\bf F}_{k}/\left[{\bf F}_{k},R\right]$
has sub-exponential growth if and only if $G=\mathbf{F}_{k}/R$ has
sub-exponential growth. 

The subgroup $R/[\mathbf{F}_{k},R]$, which is in the center of $\Gamma$,
contains the Schur multiplier of $G$: by the Hopf formula (see e.g.,
\cite[Chapter 11]{RobinsonBook}), we have
\[
\left(R\cap[{\bf F}_{k},{\bf F}_{k}]\right)/\left[{\bf F}_{k},R\right]\simeq H^{2}(G,\mathbb{Z}),
\]
where $H^{2}(G,\mathbb{Z})$ is the Schur multiplier of $G$. The
Schur multiplier of the first Grigorchuk group $\mathfrak{G}$ is
determined by Grigorchuk in \cite{GrigorchukSchur}: $H^{2}(\mathfrak{G},\mathbb{Z})\simeq\left(\mathbb{Z}/2\mathbb{Z}\right)^{\infty}$.
It follows that the central extension $\Gamma=\mathbf{F}_{4}/\left[\mathbf{F}_{4},R_{\mathfrak{G}}\right]$
of $\mathfrak{G}=\mathbf{F}_{4}/R_{\mathfrak{G}}$ contains $\left(\mathbb{Z}/2\mathbb{Z}\right)^{\infty}$
in its center. 

\end{example}

\section{Diagonal product of marked groups}

As mentioned in the Introduction, a useful way to produce FC-central
extensions is to take the diagonal product of a suitable sequence
of marked groups. 

\begin{fact}[c.f. {\cite[Remark 9.2]{EZ2}}]

Suppose $\left(\left(\Gamma_{i},S_{i}\right)\right)_{i=1}^{\infty}$
is a sequence of $k$-marked finite groups that converge to $(G,S)$
in the Cayley topology. Then the diagonal product $(\Delta,S)$ of
$\left(\left(\Gamma_{i},S_{i}\right)\right)_{i=1}^{\infty}$ is an
FC-central extension of $G$. Moreover, if $G$ is residually finite,
then $\Delta$ is residually finite as well. 

\end{fact} 

\subsection{Growth }

The following lemma is a direct consequence of the definition of convergence
in the space of marked groups and sub-multiplicativity of the volume
growth function. 

\begin{lemma}\label{FC}

Suppose $\left(\left(\Gamma_{i},S_{i}\right)\right)_{i=1}^{\infty}$
is a sequence of $k$-marked groups converging to $(G,S)$ in the
Cayley topology. The following statements are equivalent.
\begin{description}
\item [{(i)}] The diagonal product $(\Delta,S)$ of the sequence $\left(\left(\Gamma_{i},S_{i}\right)\right)_{i=1}^{\infty}$
is of sub-exponential growth.
\item [{(ii)}] For every $i\in\mathbb{N}$, $(\Gamma_{i},S_{i})$ is of
sub-exponential growth and the limit group $(G,S)$ is of sub-exponential
growth.
\end{description}
\end{lemma}

\begin{proof}

The (i)$\Rightarrow$(ii) direction is obvious because the groups
$(\Gamma_{i},S_{i})$ and $(G,S)$ are quotients of $(\Delta,S)$. 

We now prove (ii)$\Rightarrow$(i). Denote by $R_{i}$ be the largest
radius $r$ such that the ball of radius $r$ around identity in $(\Gamma_{i},S_{i})$
coincide with the ball of same radius around identity in $(G,S)$.
Given an index $i_{0}\in\mathbb{N}$, consider the diagonal product
$\left(\Delta_{>i_{0}},S\right)$ of the collection $\left(\left(\Gamma_{i},S_{i}\right)\right)_{i>i_{0}}^{\infty}$.
Then by definition the balls of radius $r_{i_{0}}:=\inf_{i>i_{0}}R_{i}$
around the identities in the Cayley graphs of $\left(\Delta_{>i_{0}},S\right)$
and $(G,S)$ coincide. Since the volume function is sub-multiplicative,
we have for any $n>r_{i_{0}}$,
\[
v_{\Delta_{>i_{0}},S}(n)\le v_{G,S}(r_{i_{0}})^{\left\lceil n/r_{i_{0}}\right\rceil }.
\]
Regard $\left(\Delta,S\right)$ as the diagonal product of $\left(\left(\Gamma_{i},S_{i}\right)\right)_{i\le i_{0}}$
and $\left(\Delta_{>i_{0}},S\right)$, we have 
\[
v_{\Delta,S}(n)\le\left(\prod_{i=1}^{i_{0}}v_{\Gamma_{i},S_{i}}(n)\right)v_{G,S}(r_{i_{0}})^{\left\lceil n/r_{i_{0}}\right\rceil }.
\]
Let $\varepsilon>0$ be any small positive constant. Since $G$ is
of subexponential growth, there is a radius $t_{\varepsilon}$ such
that $v_{G,S}(n)\le\exp(\varepsilon n/4)$ for all $n>t_{\varepsilon}$.
Since $\left(\left(\Gamma_{i},S_{i}\right)\right)_{i=1}^{\infty}$
converge to $(G,S)$ in the Chabauty topology, there exists an index
$i_{0}$ such that $r_{i_{0}}>t_{\varepsilon}$. It follows that for
any $n>r_{i_{0}}>t_{\varepsilon}$, 
\begin{equation}
v_{\Delta,S}(n)\le\left(\prod_{i=1}^{i_{0}}v_{\Gamma_{i},S_{i}}(n)\right)v_{G,S}(r_{i_{0}})^{\left\lceil n/r_{i_{0}}\right\rceil }\le\left(\prod_{i=1}^{i_{0}}v_{\Gamma_{i},S_{i}}(n)\right)e^{\varepsilon n/2}.\label{eq:delta1}
\end{equation}
Since each group $\Gamma_{i}$ is assumed to be of sub-exponential
growth, there exists a constant $n_{i_{0}}$ such that for any $i\le i_{0}$
and $n\ge n_{i_{0}}$, we have that $v_{\Gamma_{i},S_{i}}(n)\le e^{\varepsilon n/2i_{0}}.$
It follows then from (\ref{eq:delta1}) that for any $n\ge\max\left\{ n_{i_{0}},r_{i_{0}}\right\} $,
$v_{\Delta,S}(n)\le e^{\varepsilon n}$. Since $\varepsilon$ is arbitrary,
we conclude that $(\Delta,S)$ is of subexponential growth.

\end{proof}

\subsection{The direct sum assumption}

In this subsection we consider a splitting condition which makes the
structure of the diagonal product more transparent. 

\begin{definition}[Direct sum assumption (D)]\label{splitting(S)}

We say a sequence of $k$-marked groups $(\Gamma_{i},S_{i})$, $i\in\mathbb{N}$
satisfies the direct sum condition (D) over $(G,S)$ if
\begin{itemize}
\item $(\Gamma_{i},S_{i})\to(G,S)$ when $i\to\infty$ in the Cayley topology, 
\item for each $i$, there is a marked quotient $G_{i}$ of $\Gamma_{i}$,
such that the diagonal product $\Gamma$ of the sequence $\left(\left(\Gamma_{i},S_{i}\right)\right)_{i\in\mathbb{N}}$
satisfies that 
\[
\ker(\Gamma\to G)=\oplus_{i\in\mathbb{N}}\ker\left(\Gamma_{i}\to G_{i}\right).
\]
\end{itemize}
\end{definition}

If satisfied, the direct sum assumption (D) plays an important role
in understanding the structure of the diagonal product $\Gamma$.
The groups we consider in Section \ref{sec:construction} satisfy
the direct sum assumption (D). Typically, in situations where (D)
can be verified, there is a natural choice of $(G_{i})$, e.g., $\Gamma_{i}$
by construction is an extension of $G_{i}$, where $(G_{i},S_{i})$
is a sequence of quotient groups of $(G,S)$ that converges to $(G,S)$
when $i\to\infty$. To verify (D), it then suffices to show: 
\begin{description}
\item [{(1)}] the length of the shortest nontrivial element in $\ker\left(\Gamma_{i}\to G_{i}\right)$
goes to infinity as $i\to\infty$; 
\item [{(2)}] for each $i$, one can find words in $\mathbf{F}_{k}$ such
that their images in any $\Gamma_{j}$, $j>i$ are trivial and the
normal closure of their images in $\Gamma_{i}$ is $\ker\left(\Gamma_{i}\to G_{i}\right)$. 
\end{description}

\section{The definition of $\Gamma(\mathcal{L})$ and formal recursions \label{sec:definition}}

This section is a preparation for the study of finitely generated
subgroups of $\Gamma(\mathcal{L})$ in the next section. Throughout
this section, denote by the $\mathbf{F}$ the free product $\mathbf{F}=(\mathbb{Z}/2\mathbb{Z})\ast(\mathbb{Z}/2\mathbb{Z}\times\mathbb{Z}/2\mathbb{Z})$.
Mark $\mathbf{F}$ with the generating tuple $\mathbf{S}=\left(\mathbf{a},\mathbf{b},\mathbf{c},\mathbf{d}\right)$,
where $\mathbf{a},\mathbf{b},\mathbf{c},\mathbf{d}$ are involutions
and $\mathbf{F}=\left\langle \mathbf{a}\right\rangle \ast\left(\left\langle \mathbf{b}\right\rangle \times\left\langle \mathbf{c}\right\rangle \right)$,
$\mathbf{d}=\mathbf{bc}$. In Subsection \ref{subsec:def} we present
the type of construction in \cite[Section 2]{Erschler06} and \cite{KassabovPak}
in the language of permutation wreath products, which is algebraically
rather transparent. In Subsection \ref{subsec:formal-recursion},
we set up notations for the formal wreath recursion and explain the
branching structure on the formal level. Automorphisms of $\Gamma(\mathcal{L})$
are discussed in Subsection \ref{subsec:automorphisms} and a sufficient
condition for $\Gamma(\mathcal{L})$ to satisfy the direct sum assumption
(D) is provided in Subsection \ref{subsec:sufficient(D)}. 

\subsection{The definition of $\Gamma(\mathcal{L},\omega)$\label{subsec:def}}

We use notations introduced in Subsection \ref{subsec:Grigorchuk-groups}.
Let $\mathcal{L}=(A_{n})_{n=1}^{\infty}$ be a sequence of marked
quotients of $\mathbf{F}$ and $\omega$ be a string in $\{0,1,2\}^{\infty}$.
For each $n$, fix a bijection 
\begin{equation}
\psi_{n}:\left\{ b_{\mathfrak{s}^{n}\omega},c_{\mathfrak{s}^{n}\omega},d_{\mathfrak{s}^{n}\omega}\right\} \to\left\{ b,c,d\right\} .\label{eq:bij}
\end{equation}
For example, one can choose $\psi_{n}$ to send $x_{\mathfrak{s}^{n}(\omega)}\mapsto x$,
for $x=b,c,d$. 

Recall that each letter $0,1,2$ denotes a nontrivial homomorphism
$\{id,b,c,d\}\to\{id,a\}$. Recall the recursive definition of the
tree automorphisms $b_{\omega},c_{\omega}$ and $d_{\omega}$. Denote
by $\pi_{n}$ the natural projection ${\rm Aut}(\mathsf{T})\to{\rm Aut}(\mathsf{T}^{n}).$
Consider the permutation wreath product $A_{n}\wr_{\mathsf{L}_{n}}\pi_{n}\left(G_{\omega}\right)$
and its subgroup $\Gamma_{n}$ defined as
\begin{align}
\Gamma_{n}^{\omega} & =\left\langle S_{n}\right\rangle ,\ S_{n}=\left(a_{n},b_{n}^{\omega},c_{n}^{\omega},d_{n}^{\omega}\right),\label{eq:Gamma_n}\\
\mbox{where }a_{n}=(id,a), & \ x_{n}^{\omega}=\left(\delta_{1^{n}}^{\psi_{n}\left(x_{\mathfrak{s}^{n}\omega}\right)}+\delta_{1^{n-1}0}^{\omega_{n-1}(x)},x\right)\mbox{ for }x\in\{b,c,d\}.\nonumber 
\end{align}
Note that the definition of $x_{n}^{\omega}$ mimics the sections
of the generator $x_{\omega}$ on the level $n$ in $G_{\omega}$,
$x\in\{b,c,d\}$, while the difference is that the lamp group is $A_{n}$
in $\Gamma_{n}^{\omega}$. 

\begin{definition}\label{def:gamma(L)}

Given a sequence $\mathcal{L}$ of finite quotients of $\mathbf{F}$,
$\omega\in\{0,1,2\}^{\infty}$ and bijections $\left(\psi_{n}\right)_{n=1}^{\infty}$
as in (\ref{eq:bij}). Let the sequence of marked groups $\left(\left(\Gamma_{n}^{\omega},S_{n}^{\omega}\right)\right)_{n=1}^{\infty}$
be defined as in (\ref{eq:Gamma_n}). We call the diagonal product
of the sequence $\left(\left(\Gamma_{n}^{\omega},S_{n}^{\omega}\right)\right)_{n=1}^{\infty}$
the\emph{ extension of $G_{\omega}$ with $\mathcal{L}$ as input}
and denote it as $\Gamma(\mathcal{L},\omega)$. 

\end{definition}

Although most of the statements in this section can be generalized
to $\Gamma(\mathcal{L},\omega)$, where all three letters $0,1,2$
appear infinitely often in $\omega$, in order not to burden the reader
with heavier notations, we will focus on the first Grigorchuk group
$\mathfrak{G}=G_{(012)^{\infty}}$. For the first Grigorchuk group,
we suppress the reference to the string $(012)^{\infty}$ and 
\[
\Gamma(\mathcal{L})=\Gamma\left(\mathcal{L},(012)^{\infty}\right),
\]
where the bijections are: for $\omega=(012)^{\infty}$, 
\begin{align}
\psi_{n}\left(b_{\mathfrak{s}^{n}\omega}\right) & =b,\ \psi_{n}\left(c_{\mathfrak{s}^{n}\omega}\right)=c,\ \psi_{n}\left(d_{\mathfrak{s}^{n}\omega}\right)=d\mbox{ for }n\equiv0\mod3,\label{eq:phi_n1}\\
\psi_{n}\left(b_{\mathfrak{s}^{n}\omega}\right) & =c,\ \psi_{n}\left(c_{\mathfrak{s}^{n}\omega}\right)=d,\ \psi_{n}\left(d_{\mathfrak{s}^{n}\omega}\right)=b\mbox{ for }n\equiv1\mod3,\nonumber \\
\psi_{n}\left(b_{\mathfrak{s}^{n}\omega}\right) & =d,\ \psi_{n}\left(c_{\mathfrak{s}^{n}\omega}\right)=b,\ \psi_{n}\left(d_{\mathfrak{s}^{n}\omega}\right)=c\mbox{ for }n\equiv2\mod3.\nonumber 
\end{align}

Indexing by $\mathbb{N}$ in $\mathcal{L}$ is for convenience. In
particular, the shift $\mathfrak{s}$, where $\mathfrak{s}\mathcal{L}=\left(A_{n+1}\right)_{n=1}^{\infty}$,
will be useful in the recursion. In our notation, if the diagonal
product is taken over a subsequence $(n_{i})$, then for a level $n\notin\{n_{i}:i\in\mathbb{N}\}$,
the corresponding $A_{n}$ is the trivial group $\{id\}$. 

\subsection{Formal wreath recursion\label{subsec:formal-recursion}}

Denote by $\mathfrak{S}_{2}$ the permutation group of $\{0,1\}$
and $\varepsilon$ the transposition $(0,1)$. Recall that under the
canonical wreath recursion, the generators of the Grigorchuk group
$\mathfrak{G}$ give
\[
a=\left(id,id\right)\varepsilon,\ b=(a,c),\ c=(a,d),\ d=(id,b).
\]
We now consider the recursion rules on the formal level. Let $\boldsymbol{\varphi}$
be the homomorphism 
\[
\boldsymbol{\varphi}:\mathbf{F}\to\mathbf{F}\wr_{\{0,1\}}\mathfrak{S}_{2}
\]
determined by 
\begin{align*}
{\bf a} & \mapsto\left(id,\varepsilon\right),\\
{\bf b} & \mapsto\left(\delta_{1}^{{\bf \mathbf{c}}}+\delta_{0}^{{\bf a}},id\right),\\
{\bf c} & \mapsto\left(\delta_{1}^{{\bf d}}+\delta_{0}^{{\bf a}},id\right),\\
{\bf d} & \mapsto\left(\delta_{1}^{{\bf b}},id\right).
\end{align*}
Following the notation of canonical wreath recursion on ${\rm Aut}(\mathsf{T})$,
we record $\boldsymbol{\varphi}(w)=\left(\delta_{0}^{w_{0}}+\delta_{1}^{w_{1}},\varepsilon^{s}\right)$
as $(w_{0},w_{1})\varepsilon^{s}$. We refer to the homomorphism $\boldsymbol{\varphi}$
as the (1-step) \emph{formal wreath recursion}. The word \textquotedbl formal\textquotedbl{}
to indicate that the map is considered on the free product $\mathbf{F}$
instead of rooted tree automorphisms. The homomorphism $\boldsymbol{\varphi}$
can be applied recursively and we have 
\[
\boldsymbol{\varphi}^{k}:\mathbf{F}\to\mathbf{F}\wr_{\mathsf{L}_{k}}G_{k}.
\]
Mark $\boldsymbol{\varphi}^{k}(\mathbf{F})$ by the generating tuple
of $\left(\boldsymbol{\varphi}^{k}(\mathbf{a}),\boldsymbol{\varphi}^{k}(\mathbf{b}),\boldsymbol{\varphi}^{k}(\mathbf{c}),\boldsymbol{\varphi}^{k}(\mathbf{d})\right)$.
Denote by $\theta_{k}$ the projection $\mathbf{F}\wr_{\mathsf{L}_{k}}G_{k}\to G_{k}$.
Because of the recursion rules, the following diagram commute: 

\[ \begin{tikzcd}
\mathbf{F} \arrow{r}{\boldsymbol{\varphi}^{k}} \arrow[swap]{d}{\pi} & 
\boldsymbol{\varphi}^{k}(\mathbf{F}) \arrow{d}{\theta_k} \\% 
\mathfrak{G} \arrow{r}{\pi_k}& G_k 
\end{tikzcd} \]

Let $\sigma:\{a,b,c,d\}^{\ast}\to\{a,b,c,d\}^{\ast}$ be the substitution
\begin{equation}
\sigma:a\mapsto aba,\ b\mapsto d,\ c\mapsto b,\ d\mapsto c.\label{eq:L-sub}
\end{equation}
A similar substitution (often referred to as the Lysionok substitution)
which sends $a\mapsto aca,b\mapsto d,c\mapsto b,d\mapsto c$, appears
in the recursive presentation \cite{Lysionok} of the Grigorchuk group
$\mathfrak{G}$. We may regard the substitution $\sigma$ as a homomorphism
$\mathbf{F}\to\mathbf{F}$. 

Denote by $\mathbf{K}$ the normal closure of $\left[\mathbf{a},\mathbf{b}\right]$
in $\mathbf{F}$. Let $K$ be the image of $\mathbf{K}$ under the
projection $\pi:\mathbf{F}\to\mathfrak{G}$. The group $\mathfrak{G}$
is regularly branching over $K$, see e.g., \cite[Subsection 1.6.6]{BGShandbook}.
Fact \ref{K-recursion} expresses the branching structure in terms
of the formal recursion. Denote by $\mathbf{K}_{v}$ the subgroup
of $\mathbf{F}\wr_{\mathsf{L}_{k}}G_{k}$, where $k=|v|$, defined
as
\[
\mathbf{K}_{v}:=\left\{ \left(\delta_{v}^{w},id\right):w\in\mathbf{K}\right\} .
\]

\begin{fact}\label{K-recursion}

The substitution $\sigma$ satisfies that $\sigma\left(\mathbf{K}\right)\subseteq\mathbf{K}$
and for any $w\in\mathbf{K}$, $\boldsymbol{\varphi}\left(\sigma(w)\right)=(id,w)$.
Moreover, we have 
\begin{equation}
\prod_{v\in\mathsf{L}_{k}}\mathbf{K}_{v}<\boldsymbol{\varphi}^{k}(\mathbf{F}).\label{eq:Kprod}
\end{equation}

\end{fact}

\begin{proof}

It suffices to check on the generator $\left[\mathbf{a},\mathbf{b}\right]$:
\[
\sigma(\left[\mathbf{a},\mathbf{b}\right])=\mathbf{abadabad}=\left[\mathbf{a},\mathbf{b}\right]\left[\mathbf{a},\mathbf{b}\right]^{\mathbf{d}}\in\mathbf{K}
\]
and 
\[
\boldsymbol{\varphi}\left(\sigma(\left[\mathbf{a},\mathbf{b}\right])\right)=\boldsymbol{\varphi}\left(\mathbf{abadabad}\right)=\left(id,\mathbf{abab}\right).
\]
The calculation above also shows that $\boldsymbol{\varphi}\left(\left\langle \mathbf{\left[\mathbf{b}^{a},\mathbf{d}\right]}\right\rangle ^{\mathbf{F}}\right)=\mathbf{K}\times\mathbf{K}$. 

Iterate $k$ times, we have that for $w\in\mathbf{K}$, $\boldsymbol{\varphi}^{k}\left(\sigma^{k}(w)\right)=\left(\delta_{1^{k}}^{w},id\right)$.
Then (\ref{eq:Kprod}), follows. 

\end{proof}

The following length reduction property can be seen directly from
the rule of multiplication in $\mathbf{F}\wr_{\mathsf{L}_{k}}G_{k}$.
For an element $w\in\mathbf{F}$, under the formal wreath recursion
$\boldsymbol{\varphi}^{k}:\mathbf{F}\to\mathbf{F}\wr_{\mathsf{L}_{k}}G_{k}$,
write 
\[
\boldsymbol{\varphi}^{k}(w)=\left(\left(w_{v}\right)_{v\in\mathsf{L}_{k}},\pi_{k}(w)\right).
\]
The element $w_{v}$ is referred to as the \emph{section of $w$ at
$v$ under }$\boldsymbol{\varphi}^{k}$. Recall that $d_{S}$ denotes
the graph distance on the orbital Schreier graph. 

\begin{fact}[Length reduction, c.f. {\cite[Lemma 5.11]{KassabovPak}}]\label{short-section}

Suppose $w\in\mathbf{F}$ is a word length $\ell\le2^{k-1}-1$, then
under $\boldsymbol{\varphi}^{k}$, the word length of the section
$w_{v}$ is bounded by $1$. More precisely, $w_{v}\in\{id,\mathbf{b},\mathbf{c},\mathbf{d}\}$
for $v\in\mathsf{L}_{k}$ such that $d_{S}(v,1^{k})\le2^{k-1}-1$;
and $w_{v}\in\{id,\mathbf{a}\}$ for $v\in\mathsf{L}_{k}$ such that
$d_{S}(v,1^{k-1}0)\le2^{k-1}-1$.

\end{fact}

Fact \ref{short-section} implies that the sequence of marked groups
$\left((\Gamma_{n},S_{n})\right)_{n=1}^{\infty}$, as in Definition
\ref{def:gamma(L)}, converges to the Grigorchuk group $\left(\mathfrak{G},S\right)$
in the Cayley topology. 

\subsection{Lifting automorphisms of $\mathfrak{G}$ to $\Gamma(\mathcal{L})$\label{subsec:automorphisms}}

The automorphism group of $\mathfrak{G}$ is determined by Grigorchuk
and Sidki in \cite{Grigorchuk-Sidki}. For $i\in\mathbb{N}$, let
$\theta_{i}$ be the element in ${\rm Aut}(\mathsf{T})$ defined as
$\theta_{i}\in{\rm St}_{{\rm Aut}(\mathsf{T})}(i)$ and the sections
in level $i$ are given by 
\[
\left(\theta_{i}\right)_{v}=ad,\ \mbox{for all }v\in\mathsf{L}_{i}.
\]
By the main theorem of \cite{Grigorchuk-Sidki}, ${\rm Aut}(\mathfrak{G})$
is isomorphic to the normalizer ${\rm N}_{{\rm Aut}(\mathsf{T})}(\mathfrak{G})$
of $\mathfrak{G}$ in ${\rm Aut}(\mathsf{T})$ and 
\begin{equation}
{\rm N}_{{\rm Aut}(\mathsf{T})}(\mathfrak{G})=\mathfrak{G}V,\label{eq:NG-1}
\end{equation}
where $V$ is the subgroup generated by the collection $\left\{ \theta_{i},i\in\mathbb{N}\right\} $. 

In the next section we will need to lift automorphisms of $\mathfrak{G}$.
Without further assumptions it is not always guaranteed that an outer
automorphism of $\mathfrak{G}$ can be lifted to $\Gamma\left(\mathcal{L}\right)$,
however the following weaker version will be sufficient for our purposes.
Denote by $\Delta_{>k}=\Delta_{>k}(\mathcal{L})$ the diagonal product
of $\left(\left(\Gamma_{n},S_{n}\right)\right)_{n=k+1}^{\infty}$.
Note that $\Delta_{>k}$ can be identified with $\Gamma(\mathcal{L}^{k})$,
where 
\[
\mathcal{L}^{k}=\left(\underbrace{\{id\},\ldots,\{id\}}_{k\mbox{ trivial groups}},A_{k+1},A_{k+2},\ldots\right).
\]
Denote by $\pi_{>k}$ the marked projection $\mathbf{F}\to\Delta_{>k}$.
By its definition, we have the embedding
\begin{equation}
\varphi_{k}:\Delta_{>k}\hookrightarrow\Gamma(\mathfrak{s}^{k}\mathcal{L})\wr_{\mathsf{L}_{k}}G_{k},\label{eq:phi_k2}
\end{equation}
where $\varphi_{k}$ sends 
\[
a\mapsto\left(id,a\right),\ x\mapsto\left(\delta_{1^{k}}^{\psi_{k}(x)}+\delta_{1^{k-1}0}^{\omega_{k-1}(x)},x\right),\ x\in\{b,c,d\},
\]
where $\omega=(012)^{\infty}$ and $\psi_{k}$ is specified in (\ref{eq:phi_n1}).
Consider the quotient map 
\[
\varrho_{k}:\mathbf{F}\wr_{\mathsf{L}_{k}}G_{k}\to\Gamma(\mathfrak{s}^{k}\mathcal{L})\wr_{\mathsf{L}_{k}}G_{k}
\]
which is induced by the marked projection $\mathbf{F}\to\Gamma(\mathfrak{s}^{k}\mathcal{L})$.
The following diagram commutes and $\varphi_{k}$ is an isomorphic:

\[ \begin{tikzcd}
\mathbf{F} \arrow{r}{\boldsymbol{\varphi}^{k}} \arrow[swap]{d}{\pi_{>k}} & 
\boldsymbol{\varphi}^{k}(\mathbf{F}) \arrow{d}{\varrho_{k}} \\% 
\Delta_{> k} \arrow{r}{\varphi_k}&  \varrho_{k}\circ\boldsymbol{\varphi}^{k}(\mathbf{F})
\end{tikzcd} \]In particular, $\Delta_{>k}$ can be viewed as a marked quotient of
$\boldsymbol{\varphi}^{k}(\mathbf{F})$.

\begin{lemma}\label{lift-auto1}

Let $\tau:\mathfrak{G}\to\mathfrak{G}$ be an automorphism. Then there
exists an integer $k\in\mathbb{N}$ and an automorphism $\tilde{\tau}:\Delta_{>k}\to\Delta_{>k}$
such that $\pi\circ\tilde{\tau}=\tau\circ\pi$:

\[ \begin{tikzcd}
\Delta_{>k} \arrow{r}{\tilde{\tau}} \arrow[swap]{d}{\pi} & 
\Delta_{>k}\arrow{d}{\pi} \\% 
\mathfrak{G} \arrow{r}{\tau}&  \mathfrak{G}.
\end{tikzcd} \]

\end{lemma}

\begin{proof}

By \cite{Grigorchuk-Sidki}, the automorphism $\tau$ of $\mathfrak{G}$
can be written as a finite composition $\tau=\alpha_{1}\ldots\alpha_{r}$,
where each $\alpha_{j}$ is either an inner automorphism or conjugation
by $\theta_{i_{j}}$ for some $i\in\mathbb{N}$. Set $i_{j}=0$ if
$\alpha_{j}$ is inner. Let 
\[
k=\max_{1\le j\le r}i_{j}+1.
\]
To show that $\tau$ can be lifted to an automorphism of $\Delta_{>k}$,
it suffices to show that for each $i\in\{1,\ldots,k-1\}$, conjugation
by $\theta_{i}$ can be lifted to an automorphism of $\Delta_{>k}$. 

Let $W_{i}=\mathbf{F}\wr_{\mathsf{L}_{i}}G_{i}$ and $\boldsymbol{\theta}_{i}$
be the element of $W_{i}$ given by 
\[
\boldsymbol{\theta}_{i}=\left(f,id_{G_{i}}\right),\ \mbox{where }f(v)=\mathbf{ad}\mbox{ for all }v\in\mathsf{L}_{i}.
\]
By explicit calculation on the generators, we have that in $W_{1}$,
\begin{equation}
\left[\boldsymbol{\theta}_{1},\boldsymbol{\varphi}(\mathbf{a})\right]=1,\ \left[\boldsymbol{\theta}_{1},\boldsymbol{\varphi}(\mathbf{c})\right]=(\mathbf{adad},\mathbf{adad}),\ \left[\boldsymbol{\theta}_{1},\boldsymbol{\varphi}(\mathbf{d})\right]=\left(1,\left[\mathbf{ad},\mathbf{b}\right]\right).\label{eq:i=00003D1}
\end{equation}
Since $(\mathbf{ad})^{4}\in\ker(\boldsymbol{\varphi})$, under the
formal recursion, we have 
\begin{equation}
\boldsymbol{\varphi}\left(\left[\boldsymbol{\theta}_{1},\boldsymbol{\varphi}(\mathbf{c})\right]\right)=\boldsymbol{\varphi}\left(\left(\mathbf{dada},\mathbf{adad}\right)\right)=\boldsymbol{\varphi}^{2}\left((\mathbf{ac})^{4}\right).\label{eq:i=00003D1sec}
\end{equation}

Apply the Lysionok substitution to $(\mathbf{ac})^{4}$, we obtain
the word $(\mathbf{acab})^{4}$. Note that $(\mathbf{acab})^{2}\in\mathbf{K}$
and $\boldsymbol{\varphi}\left((\mathbf{acab})^{4}\right)=\left((\mathbf{da})^{4},(\mathbf{ac})^{4}\right)$.
Inductively, for $i\ge2$, we have 

\begin{equation}
\left[\boldsymbol{\theta}_{i},\boldsymbol{\varphi}^{i}(\mathbf{a})\right]=1,\ \boldsymbol{\varphi}\left(\left[\boldsymbol{\theta}_{i},\boldsymbol{\varphi}^{i}(\mathbf{c})\right]\right),\boldsymbol{\varphi}\left(\left[\boldsymbol{\theta}_{i},\boldsymbol{\varphi}^{i}(\mathbf{d})\right]\right)\in\boldsymbol{\varphi}^{3}\left(\mathbf{K}_{1^{i-2}}\right).\label{eq:theta-comm}
\end{equation}

Now regard $\Delta_{>k}$ as a subgroup of $\Gamma(\mathfrak{s}^{k}\mathcal{L})\wr_{\mathsf{L}_{k}}G_{k}$
via the embedding $\varphi_{k}$. Define the map $\tilde{\tau}_{i}$
on $\Delta_{>k}$ to be 
\begin{equation}
\tilde{\tau}_{i}(g)=g^{\tilde{\theta}_{i}}\ \mbox{where }\tilde{\theta}_{i}=\varrho_{k}\left(\boldsymbol{\varphi}^{k-i}\left(\boldsymbol{\theta}_{i}\right)\right).\label{eq:theta-lift}
\end{equation}
 By its definition, it is clear that when projected to $\mathfrak{G}$,
the map is conjugation by $\theta_{i}$. It remains to show that for
any $g\in\Delta_{>k}$, $g^{\tilde{\theta}_{i}}$ is in $\Delta_{>k}$,
equivalently, $[\tilde{\theta}_{i},g]\in\Delta_{>k}$. It suffices
to check on the generators $a,c,d$. For the generator $a$, we have
$\left([\tilde{\theta}_{i},a]\right)=id\in\Delta_{>k}$. For the other
two generators, if $i=1$, then by direct calculations as in (\ref{eq:i=00003D1})
and (\ref{eq:i=00003D1sec}) we have that $\left[\tilde{\theta}_{1},c\right],\left[\tilde{\theta}_{1},d\right]\in\Delta_{>k}$.
For $i\ge2$, by (\ref{eq:theta-comm}) we have that $\left[\tilde{\theta}_{i},c\right]$
and $\left[\tilde{\theta}_{i},d\right]$ are in the projection of
$\varrho_{k}\left(\boldsymbol{\varphi}^{k-i+2}\left(\mathbf{K}_{1^{i-2}}\right)\right)$.
By (\ref{eq:Kprod}), $\mathbf{K}_{1^{i-2}}<\boldsymbol{\varphi}^{i-2}(\mathbf{F})$.
It follows then $\left[\tilde{\theta}_{i},c\right],\left[\tilde{\theta}_{i},d\right]\in\Delta_{>k}$. 

\end{proof}

Recall the direct sum assumption (D) as in Definition \ref{splitting(S)}.
Under (D), Lemma \ref{lift-auto1} can be improved:

\begin{corollary}\label{lift-auto2}

Suppose $\Gamma(\mathcal{L})$ satisfies ${\rm (D)}$. Let $\tau:\mathfrak{G}\to\mathfrak{G}$
be an automorphism. Then there exists an automorphism $\tilde{\tau}:\Gamma(\mathcal{L})\to\Gamma(\mathcal{L})$
such that $\pi\circ\tilde{\tau}=\tau\circ\pi$. 

\end{corollary}

\begin{proof}

It suffices to prove the statement for the automorphism $\tau_{i}$
of $\mathfrak{G}$ which is the conjugation by $\theta_{i}$, $i\in\mathbb{N}$.
By Lemma \ref{lift-auto1}, $\tau_{i}$ can be lifted to an automorphism
of $\Delta_{>i+1}$ given by conjugation by $\tilde{\theta}_{i}$
as defined in (\ref{eq:theta-lift}) with $k=i+1$. Denote by $\Delta_{\le k}$
the diagonal product of the first $k$ factors $\left(\left(\Gamma_{n},S_{n}\right)\right)_{n=1}^{k}$.
Regard $\Gamma=\Gamma(\mathcal{L})$ as the diagonal product of $\Delta_{\le i+1}$
and $\Delta_{>i+1}$ and record its elements as $(f,g)$, where $f\in\Delta_{\le i+1}$
and $g\in\Delta_{>i+1}$. Define the map $\tilde{\tau}_{i}$ to be
\begin{equation}
\tilde{\tau}_{i}\left(\left(f,g\right)\right)=\left(f,g^{\tilde{\theta}_{i}}\right).\label{eq:def-auto}
\end{equation}
In words, the first factor remains the same, while the second factor,
regarded as an element in $\Gamma\left(\mathfrak{s}^{i+1}\mathcal{L}\right)\wr_{\mathsf{L}_{i+1}}G_{i+1}$,
is conjugated by $\tilde{\theta}_{i}$. To show that map $\tilde{\tau}_{i}$
is an automorphism of $\Gamma$, we need to show that for $(f,g)\in\Gamma$,
the image $\left(f,g^{\tilde{\theta}_{i}}\right)$ is in $\Gamma$,
equivalently, $\left(id,[\tilde{\theta}_{i},g]\right)\in\Gamma$.
By Lemma \ref{lift-auto1}, we have that $[\tilde{\theta}_{i},g]\in\Delta_{>i+1}$.
Under (D), a pair $(f,g)\in\Delta_{\le i+1}\times\Delta_{>i+1}$ is
in the diagonal product $\Delta$ if and only if the projections to
$G_{i+1}$ and $\mathfrak{G}$ are consistent, that is $\bar{\pi}_{i+1}(f)=\pi_{i+1}\circ\pi(g)$,
where $\bar{\pi}_{i+1}$ is the projection $\Delta_{\le i+1}\to G_{i+1}$,
$\pi$ is the projection $\Delta_{>i+1}\to\mathfrak{G}$. By direct
inspection on the generators, we have that $\left[\tilde{\theta}_{i},a\right],\left[\tilde{\theta}_{i},c\right]$
and $\left[\tilde{\theta}_{i},d\right]$ all project to identity in
$G_{i+1}$. It follows that $[\tilde{\theta}_{i},g]$ projects to
identity in $G_{i+1}$ for any $g\in\Delta_{>i+1}$ and $\left(id,[\tilde{\theta}_{i},g]\right)\in\Gamma$. 

\end{proof}

\subsection{A sufficient condition for (D)\label{subsec:sufficient(D)}}

Although not needed in the next section, we explain a condition on
the sequence $\mathcal{L}$ that not only guarantees the direct sum
assumption (D), but also allows to identify the FC-center of $\Gamma(\mathcal{L})$
explicitly. 

\begin{lemma}\label{K-split}

Suppose that $\mathcal{L}=(A_{n})_{n=1}^{\infty}$ is a sequence of
marked quotients of $\mathbf{F}$ such that the subgroup $\mathbf{K}\cap\ker\left(\boldsymbol{\varphi}\right)$
projects onto the commutator $[A_{n},A_{n}]$ for all $n\in\mathbb{N}$.
Then $\left(\left(\Gamma_{n},S_{n}\right)\right)_{n=1}^{\infty}$
satisfies the direct sum assumption ${\rm (D)}$ over $\mathfrak{G}$
and the FC-center of $\Gamma(\mathcal{L})$ is 
\[
Z_{\mathrm{FC}}\left(\Gamma\left(\mathcal{L}\right)\right)\simeq\bigoplus_{n=1}^{\infty}\bigoplus_{v\in\mathsf{L}_{n}}\left(\left[A_{n},A_{n}\right]\right)_{v}.
\]

\end{lemma}

\begin{proof}

By Fact \ref{short-section}, the sequence $\left((\Gamma_{n},S_{n})\right)_{n=1}^{\infty}$
converges to the Grigorchuk group $\left(\mathfrak{G},S\right)$ in
the Cayley topology. 

We first verify that for any $n\in\mathbb{N}$ and $g\in\left[A_{n},A_{n}\right]$,
there exists $\gamma\in\Gamma(\mathcal{L})$ such that the projection
of $\gamma$ to $\Gamma_{j}$, $j\ge n+1$, is trivial; while the
projection of $\gamma$ to $\Gamma_{n}$ is $\left(\delta_{1^{n}}^{g},id\right)$.
Let $w\in\mathbf{K}\cap\ker\left(\boldsymbol{\varphi}\right)$. Apply
the substitution $\sigma$ to $w$ for $n$ times, by Fact \ref{K-recursion},
we have $\boldsymbol{\varphi}^{n}\left(\sigma^{n}(w)\right)=\left(\delta_{1^{n}}^{w},id\right)$.
Since $w\in\ker(\boldsymbol{\varphi})$, further recursions give trivial
image, that is, $\boldsymbol{\varphi}^{j}\left(\sigma^{n}(w)\right)=id$
for all $j\ge n+1$. Thus if $g$ is the image of $w$ in the marked
quotient $A_{n}$, we can take the claimed element $\gamma$ to be
the image of $\sigma^{n}(w)$ in $\Gamma(\mathcal{L})$. 

The claim in the previous paragraph implies that if $\mathbf{K}\cap\ker\left(\boldsymbol{\varphi}\right)$
projects onto the commutator $[A_{n},A_{n}]$ for all $n\in\mathbb{N}$,
then
\[
\Gamma\left(\mathcal{L}\right)>\bigoplus_{n=1}^{\infty}\bigoplus_{v\in\mathsf{L}_{n}}\left(\left[A_{n},A_{n}\right]\right)_{v}.
\]
Denote by $\bar{\Gamma}_{n}$ the quotient group $\Gamma_{n}/\bigoplus_{v\in\mathsf{L}_{n}}\left(\left[A_{n},A_{n}\right]\right)_{v}$.
Write $\bar{\mathcal{L}}=\left(\bar{A}_{n}\right){}_{n=1}^{\infty}$,
where $\bar{A}_{n}$ is the liberalization $A_{n}/\left[A_{n},A_{n}\right]$.
By definition of $\Gamma_{n}$ and Fact \ref{short-section}, we have
that $\left(\bar{\Gamma}_{n},\bar{S}_{n}\right)$ converges to $(\mathfrak{G},S)$
and the diagonal product $\left(\Gamma\left(\bar{\mathcal{L}}\right),S\right)\simeq(\mathfrak{G},S)$.
It follows that the second item in (D) is satisfied with the sequence
$\left(\bar{\Gamma}_{n},\bar{S}_{n}\right)$. By definition it is
clear that an element in $\left(\left[A_{n},A_{n}\right]\right)_{u}$,
$|u|=n$, is contained in the finite normal subgroup $\bigoplus_{v\in\mathsf{L}_{n}}\left(\left[A_{n},A_{n}\right]\right)_{v}$,
thus in the FC-center of $\Gamma(\mathcal{L})$. Since $\Gamma\left(\bar{\mathcal{L}}\right)$
is isomorphic to $\mathfrak{G}$, in particular, it is ICC, we conclude
that the FC-center of $\Gamma\left(\mathcal{L}\right)$ is $\bigoplus_{n=1}^{\infty}\bigoplus_{v\in\mathsf{L}_{n}}\left(\left[A_{n},A_{n}\right]\right)_{v}$. 

\end{proof}

\begin{example}

Recall that $\mathbf{K}=\left\langle [{\bf a,b}]\right\rangle ^{\mathbf{F}}=[\mathbf{B},\mathbf{F}]$.
By direct calculation, the word $(\mathbf{ad})^{4}\in\ker\left(\boldsymbol{\varphi}\right)$.
Therefore the element $\left[\mathbf{b},(\mathbf{ad})^{4}\right]\in\mathbf{K}\cap\ker\left(\boldsymbol{\varphi}\right)$.
If each $A_{n}$ is normally generated by the element $\left[b,(ad)^{4}\right]$,
then the sequence $\mathcal{L}=(A_{n})$ satisfies the assumption
of Lemma \ref{K-split}. By \cite[Lemma 6.1]{KassabovPak}, if $n$
is an integer such that the only prime factors of $n$ are of the
form $p\equiv1\mod4$, then there exists a generating set $\{a,b,c,d\}$
of ${\rm PSL}_{2}(\mathbb{Z}/n\mathbb{Z})$ such that the element
$[c,[d,[b,(ad)^{4}]]]$ normally generates ${\rm PSL}_{2}(\mathbb{Z}/n\mathbb{Z})$. 

\end{example}

\section{Finitely generated subgroups of $\Gamma$($\mathcal{L}$)\label{sec:Subgroups}}

In this section we study in detail finitely generated subgroups of
$\Gamma(\mathcal{L})$. We continue to use notations introduced in
Section \ref{sec:definition}. The main point is that via natural
lifting arguments, certain properties of finitely generated subgroups
of the Grigorchuk $\mathfrak{G}$ are inherited by $\Gamma(\mathcal{L})$.
In \cite{Grigorchuk-Wilson}, it is proved that $\mathfrak{G}$ is
subgroup separable and any finitely generated subgroup $H$ of $\mathfrak{G}$
is either finite or abstractly commensurable to $\mathfrak{G}$. We
show similar statements for $\Gamma(\mathcal{L})$. 

\begin{theorem}\label{comm-1}

Let $\mathcal{L}=(A_{n})_{n=1}^{\infty}$ be a sequence of finite
marked quotients of $\mathbf{F}$. Let $H$ be a finitely generated
subgroup of $\Gamma\left(\mathcal{L}\right)$. Then there exists $\ell,p\in\mathbb{N}$
such that $H$ is abstractly commensurable to $\prod_{i=1}^{p}\Gamma\left(\mathfrak{s}^{\ell}\mathcal{L}\right)$. 

\end{theorem}

\begin{theorem}\label{separable}

Let $\mathcal{L}=(A_{n})_{n=1}^{\infty}$ be a sequence of finite
marked quotients of $\mathbf{F}$. Then $\Gamma\left(\mathcal{L}\right)$
is subgroup separable. 

\end{theorem}

\subsection{Generalized sections}

Recall the notion of sections of a subgroup $H<{\rm Aut}(\mathsf{T})$.
Given a level $k$ and a vertex $v\in\mathsf{L}_{k}$, consider the
level stabilizer ${\rm St}_{H}(k)=\{h\in H:x\cdot h=x\mbox{ for all }x\in\mathsf{L}_{k}\}$.
Under the canonical wreath recursion, $\varphi^{k}({\rm St}_{H}(k))<\prod_{u\in\mathsf{L}_{k}}{\rm Aut}(\mathsf{T}_{u})$.
Denote by $\phi_{v}$ the natural projection from the product $\prod_{u\in\mathsf{L}_{k}}{\rm Aut}(\mathsf{T}_{u})$
to ${\rm Aut}(\mathsf{T}_{v})$. The \emph{section of group }$H$
at $v$ is defined as 
\[
H_{v}:=\phi_{v}\left({\rm St}_{H}(k)\right),\mbox{ \ensuremath{k=|v|}.}
\]

For the diagonal product $\Gamma(\mathcal{L})$, we use the following
generalized notion of level stabilizers and sections. 

\begin{notation}[Generalized stabilizers and sections]

Denote by $\Delta_{>k}=\Delta_{>k}(\mathcal{L})$ the diagonal product
of $\left(\left(\Gamma_{n},S_{n}\right)\right)_{n=k+1}^{\infty}$
and $\pi_{>k}:\Gamma(\mathcal{L})\to\Delta_{>k}$ the marked projection.
Let $H$ be a subgroup of $\Gamma(\mathcal{L})$. Given a level $k$,
define the \emph{generalized level stabilizer} as
\[
{\rm St}_{H}(k):=\left\{ h\in H:x\cdot h=x\mbox{ for all }x\in\mathsf{L}_{k}\right\} ,
\]
where the action of $\Gamma(\mathcal{L})$ on $\mathsf{L}_{k}$ factors
through the projection $\Gamma(\mathcal{L})\to\mathfrak{G}$. Under
the recursion $\varphi_{k}$ given in (\ref{eq:phi_k2}), we have
\[
\varphi_{k}:{\rm St}_{H}(k)\to\prod_{u\in\mathsf{L}_{k}}\Gamma(\mathfrak{s}^{k}\mathcal{L}).
\]
Denote by $\phi_{v}$ the natural projection from $\prod_{u\in\mathsf{L}_{k}}\Gamma(\mathfrak{s}^{k}\mathcal{L})$
to the coordinate indexed by $v$. The \emph{generalized section}
of $H$ at $v$ is defined as 
\[
H_{v}:=\phi_{v}\circ\varphi_{k}\left({\rm St}_{H}(k)\right).
\]
For an element $\gamma\in{\rm St}_{H}(k)$, write $\gamma_{v}$ for
the section
\[
\gamma_{v}:=\phi_{v}\circ\varphi_{k}\left(\gamma\right).
\]

\end{notation}

We introduce one more piece of notation. 

\begin{notation}[Saturated subgroup]

We say a subgroup $H$ of a product $W_{1}\times\ldots\times W_{\ell}$
is \emph{saturated} if $\phi_{j}(H)=W_{j}$, for every $j\in\{1,\ldots,\ell\}$,
where $\phi_{j}$ is the natural projection $W_{1}\times\ldots\times W_{\ell}\to W_{j}$.
In this case we write $H<_{s}W_{1}\times\ldots\times W_{\ell}$. For
example, by the definitions, $\varphi_{k}\left({\rm St}_{H}(k)\right)$
is a saturated subgroup of $\prod_{v\in\mathsf{L}_{k}}H_{v}$. 

\end{notation}

Suppose $H<_{s}W_{1}\times W_{2}$. Write $L_{1}=H\cap\left(W_{1}\times\left\{ id_{W_{2}}\right\} \right)$
and $L_{2}=H\cap\left(\left\{ id_{W_{1}}\right\} \times W_{2}\right)$.
Then there is an isomorphic between $W_{1}/L_{1}$ and $W_{2}/L_{2}$.
In particular, if both $W_{1}$ and $W_{2}$ are just-infinite, then
either $L_{i}<_{f.i.}W_{i}$ for $i=1,2$ and $H$ is a finite index
subgroup of $W_{1}\times W_{2}$; or $L_{i}$ is trivial and $H$
is isomorphic to $W_{1}$. 

\subsection{Ingredients in the proofs \label{subsec:Ingredients}}

The key ingredient in the proof of Theorem \ref{comm-1} is the following
property of finitely generated subgroups of the Grigorchuk group $\mathfrak{G}$
shown in the work of Grigorchuk and Wilson \cite{Grigorchuk-Wilson}.

\begin{lemma}[Consequence of {\cite[Theorem 3]{Grigorchuk-Wilson}}]\label{dicho}

Let $H$ be a finitely generated subgroup of $\mathfrak{G}$. Then
there exists a finite level $\ell$ such that for each vertex $v\in\mathsf{T}_{\ell}$,
the section $H_{v}$ of $H$ is either finite or equal to $\mathfrak{G}$. 

\end{lemma}

\begin{proof}

Denote by $\mathcal{X}$ the set of subgroups of $\mathfrak{G}$ satisfying
the statement. To show that $\mathcal{X}$ contains all finitely generated
subgroups of $\mathfrak{G}$, we check that the three conditions in
\cite[Theorem 3]{Grigorchuk-Wilson} are satisfied.

(i) Clearly $\{id\}\in\mathcal{X}$ and $\mathfrak{G}\in\mathcal{X}$.

(ii) Suppose $H\in\mathcal{X}$ is a subgroup such that for each vertex
$v\in\mathsf{T}_{\ell}$, the section $H_{v}$ of $H$ is either finite
or equal to $\mathfrak{G}$. Let $L$ be a subgroup of $\mathfrak{G}$
such that $H<L$ and $|L:H|<\infty$. Then ${\rm St}_{H}(\ell)<_{f.i.}{\rm St}_{L}(\ell)$.
It follows that at each vertex $v\in\mathsf{T}_{\ell}$, the section
$H_{v}$ is a finite index subgroup of the section $L_{v}$. Therefore
the sections $L_{v}$ are either finite or equal to $\mathfrak{G}$,
that is, $L\in\mathcal{X}$ as well. 

(iii) If $H$ is finitely generated subgroup of ${\rm St}_{\mathfrak{G}}(1)$,
and both sections $H_{0}$ and $H_{1}$ are in $\mathcal{X}$, then
by definition of $\mathcal{X}$ we have that $H\in\mathcal{X}$.

Then $\mathcal{X}$ contains all finitely generated subgroups of $\mathfrak{G}$
by \cite[Theorem 3]{Grigorchuk-Wilson}.

\end{proof}

Another ingredient is the following basic lifting property.

\begin{lemma}[Lifting]\label{finite kernel}

Suppose $\Gamma$ is a finitely generated FC-central extension of
$G$ that fits into $1\to N\to\Gamma\overset{\pi}{\to}G\to1$, where
$N$ is torsion. Let $H$ be a subgroup of $\Gamma$.
\begin{description}
\item [{(i)}] If $\pi(H)$ is finite and $H$ is finitely generated, then
$H$ is finite.
\item [{(ii)}] If $\pi(H)$ is finite index in $G$, then $H$ is finite
index in $\Gamma$. 
\end{description}
\end{lemma}

\begin{proof}

(i) Suppose $H=\left\langle T\right\rangle $ where $T\subseteq\Gamma$
is finite, and $\pi(H)$ is finite. Then $H\cap N$ is a finite index
subgroup of $H$. It follows that $H\cap N$ is a finitely generated
torsion FC-group, thus finite. It follows that $H$ is finite.

(ii) Consider the subgroup $\tilde{H}<\Gamma$ given by $\tilde{H}=\{\gamma\in\Gamma:\pi(\gamma)\in\pi(H)\}$.
Then $\tilde{H}$ is a finite index subgroup of $\Gamma$. Let $S=\{s_{1},\ldots,s_{k}\}$
be a finite generating set of $\tilde{H}$. Since $\pi(\tilde{H})=\pi(H)$,
each generator $s_{i}=n_{i}s_{i}'$, where $s_{i}'\in H$, $n_{i}\in N$.
Since $N<Z_{{\rm FC}}(\Gamma)$ and $N$ is torsion, it follows that
$\tilde{H}$ is contained in the union of finitely many cosets of
$\left\langle s_{1}',\ldots,s_{k}'\right\rangle $. It follows that
$H$ is of finite index in $\Gamma$. 

\end{proof}

The lifting lemma \ref{finite kernel} implies that the branching
structure of $\mathfrak{G}$ is inherited by $\Gamma(\mathcal{L})$.
Consider the following subgroup of $\Gamma=\Gamma(\mathcal{L})$ which
can be viewed as a \emph{generalized vertex rigid stabilizer}. Given
a vertex $v\in\mathsf{L}_{\ell}$ and $m\in\mathbb{N}$, define
\[
R_{\Gamma}(v):=\left\{ \gamma\in{\rm St}_{\Gamma}(\ell):\ell=|v|,\ \pi_{\le\ell}(\gamma)=id,\ \gamma_{u}=id\mbox{ for all }u\in\mathsf{L}_{\ell}\setminus\{v\}\right\} ,
\]
where $\pi_{\le\ell}$ is the projection $\Gamma(\mathcal{L})\to\Delta_{\le\ell}$.
We may also regard $R_{\Gamma}(v)$ as a subgroup of $\Gamma\left(\mathfrak{s}^{\ell}\mathcal{L}\right)$.
Similar to level rigid stabilizers, we define 
\begin{equation}
R_{m}^{\Gamma}(\mathsf{T}_{v})=\prod_{u=vz,|z|=m}R_{\Gamma}(u).\label{eq:Rm}
\end{equation}

\begin{lemma}[Branching in $\Gamma(\mathcal{L})$]\label{Gamma-branch}

The generalized level rigid stabilizer $R_{m}^{\Gamma}(\mathsf{T}_{v})$
is of finite index in $\Gamma\left(\mathfrak{s}^{\ell}\mathcal{L}\right)$,
$\ell=|v|$. 

\end{lemma}

\begin{proof}

Let $u$ be a vertex in the level $m$ of the subtree $\mathsf{T}_{v}$.
In the group $\mathfrak{G}$, the level rigid stabilizer $R_{\mathfrak{G}}\left(u\right)$
is finitely generated and of finite index in the section $\mathfrak{G}_{u}$.
Let $\{g_{1}^{u},\ldots,g_{k}^{u}\}$ be a generating set of $R_{\mathfrak{G}}\left(u\right)$.
For each $g_{j}^{u}$, take an element $\tilde{g}_{j}^{u}\in\Gamma(\mathcal{L})$
such that $\pi\left(\tilde{g}_{j}^{u}\right)=g_{j}$. Consider the
subgroup $F^{u}$ generated by $\left\{ \tilde{g}_{1}^{u},\ldots,\tilde{g}_{k}^{u}\right\} $.
Applying Lemma \ref{finite kernel}, we have that the generalized
section of $F^{u}$ at $u$ is of finite index in $\Gamma(\mathfrak{s}^{\ell+m}\mathcal{L})$
and at any vertex $v\in\mathsf{L}_{\ell+m}\setminus\{u\}$, the generalized
section is of finite index. Moreover, $\Delta_{\le\ell+m}$ is a finite
group. It follows that $F^{u}\cap R_{\Gamma}(u)$ is of finite index
in $F^{u}$, thus also of finite index in $\Gamma(\mathfrak{s}^{\ell+m}\mathcal{L})$.
Take a product over $u$ in the level $m$ of $\mathsf{T}_{v}$, we
obtain the statement. 

\end{proof}

The branching lemma \ref{Gamma-branch} will be useful in the proof
of subgroup separability. We will also need to lift automorphisms
of $\mathfrak{G}$ by Lemma \ref{lift-auto1}, see Lemma \ref{finitepair}
below.

\subsection{Proofs of Theorem \ref{comm-1} and Theorem \ref{separable}}

Throughout this subsection, let $H$ be a subgroup of $\Gamma$ generated
by a finite set $T$. Denote by $Q=\pi(H)$ its projection to $\mathfrak{G}$. 

Applying Lemma \ref{dicho} to the finitely generated subgroup $Q=\left\langle \pi(T)\right\rangle $
of $\mathfrak{G}$, there is a finite level $\ell$ such that the
sections $Q_{v}$ for $v\in\mathsf{T}_{\ell}$ are either finite or
equal to $\mathfrak{G}$. Fix such a level $\ell$ for $Q$. Denote
by $\mathsf{L}_{\ell}'$ the subset of level $\ell$ vertices with
$Q_{v}=\mathfrak{G}$. 

Denote by $\pi_{u,v}$ the natural projection $\prod_{x\in\mathsf{L}_{\ell}}Q_{x}\to Q_{u}\times Q_{v}$.
Define the following equivalence relation $\sim$ on $\mathsf{L}_{\ell}'$:
$u\sim v$ if and only if $\pi_{u,v}\left({\rm St}_{Q}(\ell)\right)\cap\left(Q_{u}\times\left\{ id_{Q_{v}}\right\} \right)=\{id\}$.
In other words, $u\sim v$ if and only if there exists an isomorphism
$\tau_{u,v}:\mathfrak{G}\to\mathfrak{G}$ such that for any $g\in{\rm St}_{Q}(\ell)$,
$g_{v}=\tau_{u,v}(g_{u})$. It is clear that $\sim$ is an equivalence
relation. Denote by $J_{1},\ldots,J_{p}$ the $\sim$ classes in $\mathsf{L}_{\ell}'$.
For each equivalence class $J_{i}$, fix a representative vertex $u_{i}\in J_{i}$. 

For $1\le i\le p$, consider the following subgroup of $Q_{u_{i}}$:
\begin{equation}
\Lambda_{i}=\left\{ g\in Q_{u_{i}}:\exists\tilde{g}\in{\rm St}_{Q}(\ell)\mbox{ such that }\left(\tilde{g}\right)_{v}=\begin{cases}
\tau_{u_{i},v}(g) & \mbox{ for }v\in J_{i},\\
id & \mbox{otherwise}.
\end{cases}\right\} .\label{eq:Ki}
\end{equation}
Note that $\Lambda_{i}$ is a nontrivial normal subgroup in $Q_{u_{i}}=\mathfrak{G}$.
Since $\mathfrak{G}$ is just-infinite, $\Lambda_{i}$ is of finite
index in $Q_{u_{i}}$. 

Now we lift from $Q$ to $H$. By the lifting property in Lemma \ref{finite kernel},
we have:

\begin{lemma}\label{ge-sections}

Let $\ell$ be a level such that the sections $Q_{v}$ for $v\in\mathsf{T}_{\ell}$
are either finite or equal to $\mathfrak{G}$. Then for $v\in\mathsf{L}_{\ell}$,
the generalized section $H_{v}$ is either finite or of finite index
in $\Gamma(\mathfrak{s}^{\ell}\mathcal{L})$. 

\end{lemma}

\begin{proof}

Recall that under the recursion $\varphi_{\ell}$, we have 

\[
\varphi_{\ell}\left({\rm St}_{H}(\ell)\right)<_{s}\prod_{v\in\mathsf{L}_{\ell}}H_{v},
\]
and the projection to $\mathfrak{G}$ satisfies 
\[
\pi\left({\rm St}_{H}(\ell)\right)<_{s}\prod_{v\in\mathsf{L}_{\ell}}\pi\left(H_{v}\right),\ \mbox{where each }\pi\left(H_{v}\right)\mbox{ is finite or }\mathfrak{G}.
\]

First note that $H_{v}$ is finitely generated: it is a quotient of
${\rm St}_{H}(\ell)$ and ${\rm St}_{H}(\ell)$ is of finite index
in $H$. Recall that $H_{v}$ is regarded as a subgroup of $\Gamma\left(\mathfrak{s}^{\ell}\mathcal{L}\right)$
and $\Gamma\left(\mathfrak{s}^{\ell}\mathcal{L}\right)$ is an FC-central
extension of $\mathfrak{G}$. Then by Lemma \ref{finite kernel},
if $\pi(H_{v})$ generates a finite subgroup of $\mathfrak{G}$, then
$H_{v}$ is finite; and if $\pi(H_{v})=\mathfrak{G}$, then $H_{v}$
is finite index in $\Gamma\left(\mathfrak{s}^{\ell}\mathcal{L}\right)$.

\end{proof}

Recall the definition of the equivalence classes $J_{i}$, $1\le i\le p$,
in $\mathsf{L}_{\ell}$, which depends on the quotient group $Q$.
The automorphism lifting lemma \ref{lift-auto1} implies the following.

\begin{lemma}\label{finitepair}

Let $v$ be a vertex in $J_{i}$ and $u_{i}$ be the fixed representative
of $J_{i}$. Then the subgroup $L_{v}$ of $H_{v}$ 
\begin{equation}
L_{v}:=\left\{ h\in H_{v}:\exists\gamma\in{\rm St}_{H}(\ell),\ \gamma_{v}=h\mbox{ and }\gamma_{u_{i}}=id\right\} \label{eq:Lv}
\end{equation}
is a finite group contained in $\ker\left(H_{v}\to\mathfrak{G}\right)$.

\end{lemma}

\begin{proof}

Recall that an element $g\in{\rm St}_{Q}(\ell)$ satisfies for $v\in J_{i}$,
$g_{v}=\tau_{u_{i},v}\left(g_{u_{i}}\right)$, where $\tau_{u_{i},v}:\mathfrak{G}\to\mathfrak{G}$
is an automorphism. It follows that $L_{v}$ is contained in $\ker\left(H_{v}\to\mathfrak{G}\right)$. 

By Lemma \ref{lift-auto2}, there exists $k\in\mathbb{N}$ such that
the automorphism $\tau_{u_{i},v}$ can be lifted to an automorphism
$\tilde{\tau}_{u_{i},v}:\Delta_{>k}(\mathfrak{s}^{\ell}\mathcal{L})\to\Delta_{>k}(\mathfrak{s}^{\ell}\mathcal{L})$.
That is, the following diagram commute:

\[ \begin{tikzcd}
\Delta_{>k}(\mathfrak{s}^{\ell}\mathcal{L}) \arrow{r}{\tilde{\tau}_{u_{i},v}} \arrow[swap]{d}{\pi} & 
\Delta_{>k}(\mathfrak{s}^{\ell}\mathcal{L})\arrow{d}{\pi} \\% 
\mathfrak{G} \arrow{r}{\tau_{u_{i},v}}&  \mathfrak{G}.
\end{tikzcd} \]Regard $\Gamma\left(\mathfrak{s}^{\ell}\mathcal{L}\right)$ as the
diagonal subgroup of $\Delta_{\le k}(\mathfrak{s}^{\ell}\mathcal{L})\times\Delta_{>k}(\mathfrak{s}^{\ell}\mathcal{L})$
and write its elements as $\gamma=(\gamma',\gamma'')$. Then for $\gamma\in{\rm St}_{H}(\ell)$,
its sections at $v$ and $u_{i}$ are related by
\begin{align}
\gamma_{v} & =\left(\gamma'_{v},\gamma_{v}''\right)=\left(\gamma_{v}',n_{\gamma,v}\tilde{\theta}_{u_{i},v}\left(\gamma''_{u_{i}}\right)\right),\label{eq:lifted}\\
\mbox{where } & n_{\gamma,v}=\gamma_{v}''\tilde{\theta}_{u_{i},v}\left(\gamma''_{u_{i}}\right)^{-1}\in\ker\left(\Delta_{>k}\left(\mathfrak{s}^{\ell}\mathcal{L}\right)\to\mathfrak{G}\right).
\end{align}

Recall that ${\rm St}_{H}(\ell)$ is finitely generated. Take a symmetric
finite generating set $\{\gamma_{1},\ldots,\gamma_{m}\}$ for ${\rm St}_{H}(\ell)$.
Let $h\in L_{v}$ and $\gamma\in{\rm St}_{H}(\ell)$ be such that
$h=\gamma_{v}$ and $\gamma_{u_{i}}=id$. Express $\gamma$ as a product
of generators, $\gamma=\gamma_{j_{1}}\ldots\gamma_{j_{m}}$. Then
by (\ref{eq:lifted}), we have that the second coordinate $\gamma_{v}''$
can be written as
\begin{align*}
\gamma_{v}'' & =n_{\gamma_{j_{1}},v}\tilde{\theta}_{u_{i},v}\left(\left(\gamma_{j_{1}}\right)''_{u_{i}}\right)\ldots n_{\gamma_{j_{m}},v}\tilde{\theta}_{u_{i},v}\left(\left(\gamma_{j_{m}}\right)''_{u_{i}}\right)\\
 & =n_{\gamma}\tilde{\theta}_{u_{i},v}\left(\left(\gamma_{j_{1}}\right)''_{u_{i}}\ldots\left(\gamma_{j_{m}}\right)''_{u_{i}}\right)\\
 & =n_{\gamma}\tilde{\theta}_{u,v}\left(\gamma_{u_{i}}''\right),
\end{align*}
where $n_{\gamma}$ is in the normal closure of $\left\{ n_{\gamma_{1},v},\ldots,n_{\gamma_{k},v}\right\} $
in $\Delta_{>k}\left(\mathfrak{s}^{\ell}\mathcal{L}\right)$. Since
$\gamma_{u_{i}}=id$, we have that $\tilde{\theta}_{u,v}\left(\gamma_{u_{i}}''\right)=id$
and $\gamma_{v}''=n_{\gamma}$. The kernel $\ker\left(\Delta_{>k}\left(\mathfrak{s}^{\ell}\mathcal{L}\right)\to\mathfrak{G}\right)$
is torsion and contained in the FC-center of $\Delta_{>k}\left(\mathfrak{s}^{\ell}\mathcal{L}\right)$,
therefore the normal closure of $\left\{ n_{\gamma_{1},v},\ldots,n_{\gamma_{k},v}\right\} $
in $\Delta_{>k}\left(\mathfrak{s}^{\ell}\mathcal{L}\right)$ is finite.
Thus $\gamma_{v}''=n_{\gamma}$ is contained in the finite group $\left\langle n_{\gamma_{1},v},\ldots,n_{\gamma_{k},v}\right\rangle ^{\Delta_{>k}\left(\mathfrak{s}^{\ell}\mathcal{L}\right)}$.
Since $\Delta_{\le k}\left(\mathfrak{s}^{\ell}\mathcal{L}\right)$
is finite, we conclude that $L_{v}$ is contained in a finite group. 

\end{proof}

\begin{proof}[Proof of Theorem \ref{comm-1}] 

Let $H=\left\langle T\right\rangle $ be a finitely generated subgroup
of $\Gamma(\mathcal{L})$, $Q$ be the projection of $H$ to $\mathfrak{G}$,
and the level $\ell$ be as in the beginning of this subsection. Since
the subgroup $\pi_{>\ell}\left({\rm St}_{H}(\ell)\right)$ is of finite
index in $\pi_{>\ell}(H)$ and the kernel of $H\to\pi_{>\ell}(H)$
is finite, $H$ and $\pi_{\ge\ell}\left({\rm St}_{H}(\ell)\right)$
are commensurable up to finite kernels. 

For the product $\prod_{v\in\mathsf{L}_{\ell}}H_{v}$, take its quotient
group $\prod_{v\in\mathsf{L}_{\ell}}\bar{H}_{v}$ where 
\begin{itemize}
\item for $v\in\{u_{1},\ldots,u_{p}\}$, $\bar{H}_{v}=H_{v}$, 
\item for $v\in J_{i}\setminus\{u_{i}\}$, $1\le i\le p$, $\bar{H}_{v}=H_{v}/L_{v}$,
where $L_{v}$ is defined in (\ref{eq:Lv}),
\item for $v\notin\mathsf{L}_{i}'$, $\bar{H}_{v}=\{id\}$. 
\end{itemize}
Denote by $\bar{H}_{\ell}$ the image of $\pi_{>\ell}\left({\rm St}_{H}(\ell)\right)$
in the quotient. Since for $v\in J_{i}\setminus\{u_{i}\}$, $L_{v}$
is finite by Lemma \ref{finitepair}, we have that $\bar{H}_{\ell}$
and $H$ are commensurable up to finite kernels. Note that the projection
${\rm St}_{H}(\ell)\to{\rm St}_{Q}(\ell)$ factors through $\bar{H}_{\ell}$.
With slight abuse of notation we still denote the projection $\bar{H}_{\ell}\to{\rm St}_{Q}(\ell)$
by $\pi$. 

Next we check that $\bar{H}_{\ell}$ and $\prod_{i=1}^{p}H_{u_{i}}$
are commensurable. Note that for $v\in J_{i}$, there is a homomorphism
$\tilde{\theta}_{u_{i},v}:H_{u_{i}}\to H_{v}/L_{v}$ such that for
any $\gamma\in\bar{H}_{\ell}$, we have $\gamma_{v}=\tilde{\theta}_{u_{i},v}(\gamma_{u_{i}})$.
It follows that $\bar{H}_{\ell}$ is isomorphic to its natural projection
to $\prod_{i=1}^{p}H_{u_{i}}$. Recall the subgroup $\Lambda_{i}$
of $Q_{u_{i}}$ defined in (\ref{eq:Ki}) and the property that $\Lambda_{i}$
is of finite index in $Q_{u_{i}}=\mathfrak{G}$, in particular, it
is finitely generated. Take a finite generating set $\{\gamma_{1},\ldots,\gamma_{k}\}$
of $\Lambda_{i}$ and denote by $g_{j}$ the element in ${\rm St}_{Q}(\ell)$
such that $(g_{j})_{v}=\tau_{u_{i},v}\left(\gamma_{j}\right)$ for
$v\in J_{i}$ and $\left(g_{j}\right)_{v}=id$ for $v\in\mathsf{L}_{\ell}\setminus J_{i}$.
Now choose an element $\tilde{g}_{j}\in\bar{H}_{\ell}$ such that
$\pi(\tilde{g}_{j})=g_{j}$. Consider the subgroup of $\bar{H}_{\ell}$
generated by $\left\{ \tilde{g}_{1},\ldots,\tilde{g}_{k}\right\} $.
Let $R_{i}$ be the subgroup of $\prod_{v\in\mathsf{L}_{\ell}}\bar{H}_{v}$
defined as 
\[
R_{i}=\left\{ h:\ h_{u_{i}}\in H_{u_{i}},\ h_{v}=\tilde{\theta}_{u_{i},v}(h_{u_{i}})\mbox{ for }v\in J_{i},\ h_{v}=id\mbox{ for }v\notin J_{i}\right\} .
\]
Note that $R_{i}$ is isomorphic to $H_{u_{i}}$. By Lemma \ref{finite kernel},
we have that the group 
\begin{equation}
\mathfrak{R}_{i}=\left\langle \tilde{g}_{1},\ldots,\tilde{g}_{k}\right\rangle \cap R_{i}\label{eq:Ri}
\end{equation}
is of finite index in both $\left\langle \tilde{g}_{1},\ldots,\tilde{g}_{k}\right\rangle $
and $R_{i}$. Since this is true for each $i\in\{1,\ldots,p\}$, we
conclude that $\bar{H}_{\ell}$ and $\prod_{i=1}^{p}H_{u_{i}}$ are
commensurable.

Recall that by Lemma \ref{ge-sections}, $H_{u_{i}}$ is of finite
index in $\Gamma\left(\mathfrak{s}^{\ell}\mathcal{L}\right)$. To
summarize, we have shown 
\[
H\leftarrow{\rm St}_{H}(\ell)\rightarrow\bar{H_{\ell}}\leftarrow\prod_{i=1}^{p}\mathfrak{R}_{i}\rightarrow\prod_{i=1}^{p}\Gamma\left(\mathfrak{s}^{\ell}\mathcal{L}\right),
\]
where each arrow indicates a homomorphism of groups with finite kernel
and with image of finite index. We conclude that $H$ and $\prod_{i=1}^{p}\Gamma\left(\mathfrak{s}^{\ell}\mathcal{L}\right)$
are commensurable up to finite kernels. Since they are both residually
finite, it follows that $H$ and $\prod_{i=1}^{p}\Gamma\left(\mathfrak{s}^{\ell}\mathcal{L}\right)$
are abstractly commensurable, see e.g., \cite[IV. 28 ]{delaHarpeBook}.

\end{proof}

We now move on to prove subgroup separability. Since ${\rm St}_{H}(\ell)$
is of finite index in $H$, to prove that $H$ is separable it suffices
to show that ${\rm St}_{H}(\ell)$ is separable, see e.g., \cite[Lemma 11]{Grigorchuk-Wilson}.
In the proof we make use of the generalized rigid stabilizers as in
the branching lemma \ref{Gamma-branch} to find an explicit sequence
of subgroups of finite index in $\Gamma(\mathcal{L})$ whose intersection
is ${\rm St}_{H}(\ell)$.

\begin{proof}[Proof of Theorem \ref{separable}]

We continue to use notations introduced in this section. Recall that
for $1\le i\le p$, $u_{i}$ is a chosen representative of the equivalence
class $J_{i}$. Write 
\[
A=\mathsf{L}_{\ell}\setminus\left\{ u_{1},\ldots,u_{p}\right\} .
\]
For a vertex $v\in A$, either $H_{v}$ is finite; or $H_{v}$ is
infinite and $v$ is in the equivalence class $J_{i}$ for some $1\le i\le p$. 

Let $R_{m}^{\Gamma}(\mathsf{T}_{v})$ be the generalized level rigid
stabilizer of $\Gamma(\mathcal{L})$ in the subtree $\mathsf{T}_{v}$
defined in (\ref{eq:Rm}). Take the subgroup $U_{m}$ of $\Gamma(\mathcal{L})$
defined as
\[
U_{m}=\left\langle {\rm St}_{H}(\ell),\prod_{v\in A}R_{m}^{\Gamma}(\mathsf{T}_{v})\right\rangle .
\]
We claim that the sequence $\left(U_{m}\right)$ is the desired approximation
for ${\rm St}_{H}(\ell)$:

\begin{claim}\label{H_l intersection}

For each ${\normalcolor m\in\mathbb{N}}$, the subgroup $U_{m}$ is
of finite index in $\Gamma(\mathcal{L})$, and ${\rm St}_{H}(\ell)=\cap_{m=1}^{\infty}U_{m}$. 

\end{claim}

\begin{proof}

By Lemma \ref{Gamma-branch}, the generalized rigid stabilizer $R_{m}^{\Gamma}(\mathsf{T}_{v})$
is of finite index in $\iota_{v}\left(\Gamma\left(\mathfrak{s}^{\ell}\Gamma\right)\right)$.
To show that $U_{m}$ is of finite index in $\Gamma$, it suffices
to show that for each $v\in\mathsf{L}_{\ell}$, $U_{m}\cap R_{\Gamma}(v)$
is of finite index in $R_{\Gamma}(v)$. If $v\in A$, then by its
definition $R_{m}^{\Gamma}(\mathsf{T}_{v})<U_{m}\cap R_{\Gamma}(v)$.
If $v\notin A$ then $v=u_{i}$ for some $1\le i\le p$. Recall the
subgroup $\Lambda_{i}$ of $Q_{u_{i}}$ defined in (\ref{eq:Ki}),
$\Lambda_{i}$ is of finite index in $Q_{u_{i}}$. Consider the subgroup
of ${\rm St}_{H}(\ell)$ defined as 
\[
\tilde{\Lambda}_{i}=\left\{ \gamma\in{\rm St}_{H}(\ell):\gamma_{u_{i}}\in\Lambda_{i},\ \gamma_{u}=id\mbox{ for }u\notin J_{i},\ \gamma_{v}\in R_{m}^{\Gamma}(\mathsf{T}_{v})\mbox{ for }v\in J_{i}\setminus\{u_{i}\}\right\} .
\]
By its definition, we have that $\tilde{\Lambda}_{i}<U_{m}$. By the
same argument which lifts up a generating set of $\Lambda_{i}$ to
${\rm St}_{H}(\ell)$, we have that $\phi_{u_{i}}\left(\tilde{\Lambda}_{i}\right)$
is of finite index in $H_{u_{i}}$. It follows that $U_{m}\cap R_{\Gamma}(v)$
is of finite index in $R_{\Gamma}(u_{i})$. 

We now show that $\cap_{m=1}^{\infty}U_{m}\subseteq{\rm St}_{H}(\ell)$.
Suppose $\gamma\in\cap_{m=1}^{\infty}U_{m}$, then for any $m\in\mathbb{N}$,
one can write $\gamma$ as a product $\gamma=h_{m}r_{m}$, where $h_{m}\in{\rm St}_{H}(\ell)$
and $r_{m}\in\prod_{v\in A}R_{m}^{\Gamma}(\mathsf{T}_{v})$. In particular,
for any $m,n\in\mathbb{N}$, $\gamma=h_{m}r_{m}=h_{n}r_{n}$. It follows
that for $1\le i\le p$, $h_{m}h_{n}^{-1}$ has trivial section at
$u_{i}$: $(h_{m}h_{n}^{-1})_{u_{i}}=(r_{n}r_{m}^{-1})_{u_{i}}=id$.
By Lemma \ref{finitepair}, $h_{m}h_{n}^{-1}$ is contained in the
finite group $\prod_{v\in\mathsf{L}_{\ell}}L_{v}$. Since $\bigcap_{m=1}^{\infty}\prod_{v\in A}R_{m}^{\Gamma}(\mathsf{T}_{v})=\{id\}$,
we have that there is a sufficiently large index $m_{0}$ such that
$R_{m_{0}}^{\Gamma}(\mathsf{T}_{v})\cap L_{v}=\{id\}$. Then $r_{m}=r_{n}$
for all $n,m\ge m_{0}$. But $\cap R_{m}^{\Gamma}=\{id\}$, we conclude
that $r_{m_{0}}=id$ and $\gamma=h_{m_{0}}\in{\rm St}_{H}(\ell)$.

\end{proof}

Claim \ref{H_l intersection} implies that ${\rm St}_{H}(\ell)$ is
separable in $\Gamma(\mathcal{L})$. Since ${\rm St}_{H}(\ell)$ is
a finite index subgroup of $H$, it follows that $H$ is separable
as well. 

\end{proof}

\section{Construction of intermediate growth group with $\mathbb{Z}^{\infty}$
as center\label{sec:central}}

Our goal in this section is to exhibit examples of group of intermediate
growth whose center contains $\mathbb{Z}^{\infty}$. \textcolor{black}{Consider
a permutation wreath product $W=L\wr_{X}G$. In Tappe \cite{Tappe},
a full description of the Schur multiplier of $W$ is given. For a
pair $(x_{1},x_{2})\in X\times X$, we say its orbit $M=(x_{1},x_{2})\cdot G$
under the diagonal action of $G$ on $X\times X$ is an }\textcolor{black}{\emph{orbit
of trivial sign,}}\textcolor{black}{{} if $(x_{2},x_{1})\notin M$.
Denote by $\mathtt{t}$ the number of orbits of trivial sign in $X\times X$.
Then by \cite{Tappe}, the Schur multiplier $H^{2}(W,\mathbb{Z})$
contains the direct sum of $\mathtt{t}$ copies of $H_{1}(L)\otimes H_{1}(L)$. }

The same as in the previous sections, denote by $G_{n}$ the finite
quotient of the first Grigorchuk group $\mathfrak{G}$ under the projection
${\rm Aut}(\mathsf{T})\to{\rm Aut}(\mathsf{T}_{n})$. We still denote
by $\left\{ a,b,c,d\right\} $ the image of the generating set of
$\mathfrak{G}$ under the projection $\mathfrak{G}\to G_{n}$. The
finite group $G_{n}$ acts faithfully and level transitively on the
finite rooted tree $\mathsf{T}_{n}$ of $n$ levels. 

With the Schur multiplier in mind, consider the following kind of
finite sets where the diagonal action of $\mathfrak{G}$ on the pairs
has orbits of trivial sign. Take $X_{n}=\mathsf{L}_{n}\times G_{3}$,
where $\mathsf{L}_{n}$ is the level $n$ vertices of the rooted tree
$\mathsf{T}$. Consider the following action of $\mathfrak{G}$ on
$X_{n}$ which factors through the quotient $\mathfrak{G}\to G_{n+3}$.
Recall that the canonical wreath recursion in ${\rm Aut}(\mathsf{T})$
induces an embedding
\begin{align*}
\varphi_{n+3}:G_{n+3} & \hookrightarrow G_{3}\wr_{\mathsf{L}_{n}}G_{n},\\
g & \mapsto\left(\left(g_{v}\right){}_{v\in\mathsf{L}_{n}},\pi_{n}(g)\right).
\end{align*}
Consider the action of $G_{n+3}$ on $X_{n}$ given by
\[
(v,\gamma)\cdot g=(v\cdot g,\gamma g_{v}),\mbox{ where }v\in\mathsf{L}_{n},\ \gamma\in G_{3}.
\]
The first coordinate $v\cdot g$ is the image of $v$ under the action
of $g$ on $\mathsf{L}_{n}$, and the second coordinate $\gamma g_{v}$
is the multiplication of $\gamma$ and the section $g_{v}$ in the
finite group $G_{3}$. Note that the generators $a,b,c,d$ of $\mathfrak{G}$
have distinct images in $G_{3}$. The reason to take $X_{n}=\mathsf{L}_{n}\times G_{3}$
instead of $\mathsf{L}_{n+3}$ is to have the following orbit of trivial
sign:

\begin{lemma}\label{nonsym-orbit}

Let $G_{n+3}\curvearrowright X_{n}$ be defined as above. Then $G_{n+3}\curvearrowright X_{n}$
is transitive and the orbit of $\left((1^{n},id),(1^{n},ab)\right)$
is of trivial sign. 

\end{lemma}

\begin{proof}

Recall that $\mathfrak{G}$ acts level transitively on $\mathsf{T}$
and for any $v\in\mathsf{L}_{n}$, the section $\mathfrak{G}_{v}=\mathfrak{G}$.
It follows that the action $G_{n+3}\curvearrowright X_{n}$ is transitive. 

Assume $g\in G_{n+3}$ is such that $\left((1^{n},id),(1^{n},ab)\right)\cdot g=\left((1^{n},ab),(1^{n},id)\right)$.
Then $1^{n}\cdot g=1^{n}$, and the sections satisfy $id\cdot g_{v}=ab$
and $ab\cdot g_{v}=id$. This would imply that $abab=id$ in $G_{3}$,
which is a contradiction since $abab=(ca,ac)\neq id_{G_{3}}$. It
follows that such an element $g$ does not exist, the orbit of $\left((1^{n},id),(1^{n},ab)\right)$
under $G_{n+3}$ is of trivial sign. 

\end{proof}

\textcolor{black}{By \cite{Tappe}}, Lemma \ref{nonsym-orbit} implies
that the Schur multiplier $H^{2}\left(\mathbb{Z}\wr_{X_{n}}G_{n+3},\mathbb{Z}\right)$
has a direct summand $H_{1}(\mathbb{Z})\otimes H_{1}(\mathbb{Z})$
indexed by the orbit of $\left((1^{n},id),(1^{n},ab)\right)$. Let
$[\beta]$ be a generator of this summand $H_{1}(\mathbb{Z})\otimes H_{1}(\mathbb{Z})$
(as an abelian group), then we can take the central extension of $\mathbb{Z}\wr_{X_{n}}G_{n+2}$
corresponding to the 2-cocycle $[\beta]$. 

More explicitly, we take the following central extension of $\mathbb{Z}\wr_{X_{n}}G_{n+3}$.
Denote by $M_{n}$ the orbit of $\left((1^{n},id),(1^{n},ab)\right)$
under the diagonal action of $G_{n+3}$. Consider the step-$2$ free
nilpotent group $N_{n}$ on generators $\left\{ b_{x},x\in X_{n}\right\} $.
In the free nilpotent group $N_{n}$, impose the following additional
relations: 
\begin{align}
[b_{x},b_{y}] & =1\mbox{ if neither }(x,y)\mbox{ or }(y,x)\mbox{ is in }M_{n},\nonumber \\
\left[b_{x\cdot g},b_{y\cdot g}\right]=\left[b_{x},b_{y}\right] & \mbox{ for any }(x,y)\in M_{n},g\in G_{n+3}.\label{eq:relations}
\end{align}
Denote by $\bar{N}_{n}$ the resulting quotient group of $N_{n}$.
Since $M_{n}$ is an orbit of trivial sign, we have that $\bar{N}_{n}$
fits into the exact sequence 
\[
1\to\mathbb{Z}\to\bar{N}_{n}\to\mathbb{Z}^{|X_{n}|}\to1,
\]
where $\mathbb{Z}$ is the center of $\bar{N}_{n}$. With slight abuse
of notation we still denote by $b_{x}$ the image of the generator
$b_{x}$ in $\bar{N}_{n}$. The group $G_{n+3}$ acts on $\bar{N}_{n}$
by permuting the generators: $b_{x}\cdot g=b_{x\cdot g}$. 

Now take the semi-direct product $\Gamma_{n}=\bar{N}_{n}\rtimes G_{n+3}$.
An element of $\Gamma_{n}$ is recorded as $(h,g)$, where $h\in\bar{N}_{n}$
and $g\in G_{n+3}$. Because of the relations (\ref{eq:relations}),
$G_{n+3}$ acts trivially on the commutator subgroup of $\bar{N}_{n}$.
Therefore $\Gamma_{n}$ is a central extension of $\mathbb{Z}\wr_{X_{n}}G_{n+3}$:

\begin{lemma}\label{center0}

Let $\Gamma_{n}=\bar{N}_{n}\rtimes G_{n+3}$ be defined above. Then
the center of $\Gamma_{n}$ is
\[
Z\left(\Gamma_{n}\right)=\ker\left(\Gamma_{n}\to\mathbb{Z}\wr_{X_{n}}G_{n+3}\right)\simeq\mathbb{Z}.
\]

\end{lemma}

\begin{proof}

By the definitions, we have that 
\[
\ker\left(\Gamma_{n}\to\mathbb{Z}\wr_{X_{n}}G_{n+3}\right)=\ker\left(\bar{N}_{n}\to\oplus_{x\in\mathsf{L}_{n}}\left\langle b_{x}\right\rangle \right).
\]
Since $G_{n+3}$ acts trivially on $\left[\bar{N}_{n},\bar{N}_{n}\right]$,
$\ker\left(\Gamma_{n}\to\mathbb{Z}\wr_{X_{n}}G_{n+3}\right)\subseteq Z\left(\Gamma_{n}\right)$. 

Let $\gamma\in\Gamma_{n}$ be an element with nontrivial projection
to $\mathbb{Z}\wr_{X_{n}}G_{n+3}$, we need to show that it is not
in the center of $\Gamma_{n}$. Case (i): the projection of $\gamma$
to $G_{n+3}$ is nontrivial, then since $G_{n+3}$ acts faithfully
on $X_{n}$, there exists $x\in X_{n}$ such that $x\cdot\pi_{n+3}(\gamma)\neq x$.
Take an element $h\in\Gamma_{n}$ such that its projection to $\mathbb{Z}\wr_{X_{n}}G_{n+3}$
is $\left(b_{x},id_{G_{n+3}}\right)$. Then $[\gamma,h]$ projects
to $\left(b_{x\cdot\pi_{n+3}(\gamma)}-b_{x},id_{G_{n+3}}\right)\neq id$,
it follows that in this case $\gamma\notin Z(\Gamma_{n})$. Case (ii):
the projection of $\gamma$ to $\mathbb{Z}\wr_{X_{n}}G_{n+3}$ is
of the form $\left(\sum_{x\in X_{n}}\lambda_{x}b_{x},id_{G_{n+3}}\right)$,
where there exists at least one $x$ with $\lambda_{x}\neq0$. Since
$G_{n+3}$ acts transitively on $X_{n+3}$, there exists $g\in G_{n+3}$
such that $(1^{n},id)\cdot g=x$. Let $y=(1^{n},ab)\cdot g$. Take
an element $h\in\Gamma_{n}$ such that its projection to $\mathbb{Z}\wr_{X_{n}}G_{n+3}$
is $\left(b_{y},id_{G_{n+3}}\right)$. Then we have $\left[\gamma,h\right]=\left(\left[\lambda_{x}b_{x},b_{y}\right],id_{G_{n+3}}\right)=\left(\lambda_{x}\left[b_{(1^{n},id)},b_{(1^{n},ab)}\right],id_{G_{n+3}}\right)\neq id$.
Combining the two cases, we conclude that $Z\left(\Gamma_{n}\right)\subseteq\ker\left(\Gamma_{n}\to\mathbb{Z}\wr_{X_{n}}G_{n+3}\right)$. 

\end{proof}

Mark the group $\Gamma_{n}$ with the generating tuple $T_{n}=(a_{n},b_{n},c_{n},d_{n},t_{n})$,
where 
\[
x_{n}=\left(id_{\bar{N}_{n}},x\right)\mbox{ for }x\in\{a,b,c,d\}\mbox{ and }t_{n}=\left(b_{(1^{n},id)},id_{G_{n+3}}\right).
\]

We have defined a sequence of marked groups $\left(\left(\Gamma_{n},T_{n}\right)\right)_{n=1}^{\infty}$.
Next we show that the sequence $(\Gamma_{n},T_{n})$ converges in
the Cayley topology and identify the limit. Let $\mathcal{S}=1^{\infty}\cdot\mathfrak{G}$
be the orbit of the right most ray under the action of $\mathfrak{G}$,
$\mathcal{S}$ consists of infinite strings cofinal with $1^{\infty}$.
Consider the embedding 
\begin{equation}
\vartheta:\mathfrak{G}\hookrightarrow\left(\mathbb{Z}/2\mathbb{Z}\times\mathbb{Z}/2\mathbb{Z}\right)\wr_{\mathcal{S}}\mathfrak{G},\label{eq:theta-em}
\end{equation}
given by $\vartheta(a)=({\bf 0},a)$, $\vartheta(b)=(\delta_{1^{\infty}}^{b},b)$,
$\vartheta(c)=(\delta_{1^{\infty}}^{c},c)$ and $\vartheta(d)=(\delta_{1^{\infty}}^{d},d)$.
Under the embedding $\vartheta$, we write 
\[
\vartheta(g)=\left(\Phi_{g},g\right),
\]
where the function $\Phi_{g}:\mathcal{S}\to\mathbb{Z}/2\mathbb{Z}\times\mathbb{Z}/2\mathbb{Z}$
can be viewed as the \textquotedbl germ configuration\textquotedbl{}
of $g$. Consider the action of $\mathfrak{G}$ on the space $\mathcal{S}\times\left(\mathbb{Z}/2\mathbb{Z}\times\mathbb{Z}/2\mathbb{Z}\right)$
given by
\[
(x,\gamma)\cdot g=\left(x\cdot g,\gamma\Phi_{g}(x)\right),\ \mbox{where }x\in\mathcal{S},\gamma\in\mathbb{Z}/2\mathbb{Z}\times\mathbb{Z}/2\mathbb{Z}.
\]
Denote by $X$ the orbit of $\left(1^{\infty},id\right)$ under the
action of $\mathfrak{G}$. Denote by $W$ the permutation wreath product
$W=\mathbb{Z}\wr_{X}\mathfrak{G}$, marked with the generating tuple
$T=(\tilde{a},\tilde{b},\tilde{c},\tilde{d},t)$, where $\tilde{a}=({\bf 0},a)$,
$\tilde{b}=({\bf 0},b)$, $\tilde{c}=({\bf 0},c)$, $\tilde{d}=({\bf 0},d)$
and $t=\left(\delta_{\left(1^{\infty},id\right)}^{1},id_{\mathfrak{G}}\right)$. 

\begin{proposition}\label{center1}

The sequence $\left(\Gamma_{n},T_{n}\right)$ converges to $(W,T)$
in the Cayley topology when $n\to\infty$. 

\end{proposition}

\begin{proof}

We show that the balls of radius $2^{n-1}-1$ around the identities
in $\left(\Gamma_{n},T_{n}\right)$ and $(W,T)$ are identical. In
what follows we denote by $\pi_{n}$ the marked projection to $G_{n}$,
where $t$ is sent to $id_{G_{n}}$. 

Take a word $w=x_{1}\ldots x_{\ell}$ in the letters $\{a,b,c,d,t\}$
of length $\ell\le2^{n-1}-1$. Consider the image of $w$ in $\Gamma_{n}$.
Consider the inverted orbit of $\left(1^{n},id\right)$:
\[
\mathcal{O}_{n}(w)=\left\{ \left(1^{n},id\right)\cdot\pi_{n+2}\left(\left(x_{1}\ldots x_{i}\right)^{-1}\right):0\le i\le\ell\right\} .
\]
Write $w_{i}=x_{1}\ldots x_{i}$. By the definition of the action,
for $(v,\gamma)=\left(1^{n},id\right)\cdot\pi_{n+2}\left(w_{i}\right)^{-1}$,
the second coordinate is
\[
\gamma=\pi_{2}\left(\left(w_{i}\right)_{1^{n}\cdot\pi_{n}(w_{i})^{-1}}\right).
\]
Note that $d\left(1^{n},1^{n}\cdot\pi_{n}\left(w_{i}\right)^{-1}\right)\le|w_{i}|\le\ell$.
By Lemma \ref{short-section}, we have that 
\[
\left(w_{i}\right)_{1^{n}\cdot\pi_{n}(w_{i})^{-1}}\in\{id,b,c,d\}.
\]
It follows that $\gamma\in\{id,b,c,d\}$. Recall that $M_{n}$ is
the orbit of $\left((1^{n},id),(1^{n},ab)\right)$ under the diagonal
action of $G_{n+3}$. Note that for a pair $\left(\left(v,\gamma\right),(u,\gamma')\right)$
to be in the orbit $M_{n}$, it is necessarily that $v=u$ and $\pi_{1}(\gamma^{-1}\gamma')=a$.
The fact that the second coordinate $\gamma$ is in $\{id,b,c,d\}$
for any $(v,\gamma)\in\mathcal{O}_{n}(w)$ implies that
\[
\left(\mathcal{O}_{n}(w)\times\mathcal{O}_{n}(w)\right)\cap M_{n}=\emptyset.
\]
Because of the relations imposed in $\bar{N}_{n}$, the subset $\left\{ b_{(v,\gamma)}\right\} _{(v,\gamma)\in\mathcal{O}_{n}(w)}$
of basis elements generates an abelian subgroup of $\bar{N}_{n}$.
Therefore the image of $w$ in $\Gamma_{n}$ can be written as 
\[
\tilde{\pi}_{n}(w)=\left(\sum_{(v,\gamma)\in\mathcal{O}_{n}(w)}z_{(v,\gamma)}b_{(v,\gamma)},\pi_{n+3}(w)\right),\ \mbox{where }z_{(v,\gamma)}\in\mathbb{Z}.
\]

Next we consider the image of the same word $w$ in $W$. By definition
of the action on $X$ and the rule of multiplication in $W$, we have
that the inverted orbit $\mathcal{O}(w)=\left\{ \left(1^{\infty},id\right)\cdot\pi\left(w_{1}\ldots w_{i}\right)^{-1}:0\le i\le\ell\right\} $
satisfies that for any $(v,\gamma)\in\mathcal{O}_{n}(w)$, the second
coordinate $\gamma\in\{id,b,c,d\}$. The image of $w$ in $W$ is
\[
\tilde{\pi}(w)=\left(\sum_{(u,\gamma)\in\mathcal{O}(w)}x_{(u,\gamma)}\delta_{(u,\gamma)},\pi(w)\right),\ \mbox{where }x_{(u,\gamma)}\in\mathbb{Z}.
\]

For $(v,\gamma)\in\mathcal{O}_{n}(w)$, since $d(v,1^{n})\le2^{n-1}-1$,
we have that the last digit of $v$ must be $1$. Similarly, for $(u,\gamma)\in\mathcal{O}(w)$,
$d(u,1^{\infty})\le2^{n-1}-1$ implies $u$ is of the form $u=u_{1}\ldots u_{n-1}1^{\infty}$.
Together with the property that the second coordinates must be in
$\{id,b,c,d\}$ as shown in the previous paragraphs, we have that
the following map is a bijection: 
\begin{align*}
\mathcal{O}_{n}(w) & \to\mathcal{O}(w)\\
(v,\gamma) & \mapsto(v1^{\infty},\gamma),
\end{align*}
Moreover, the multiplication rule in $\Gamma_{n}$ and $W$ imply
that $z_{(v,\gamma)}=x_{(v1^{\infty},\gamma)}$. Finally, by Lemma
\ref{short-section} we have that the injective radius of the marked
projection $\mathfrak{G}\to G_{n+3}$ is at least $2^{n-1}-1$. Thus
$\pi_{n+3}(w)\in G_{n+3}$ can be identifies with $\pi(w)\in\mathfrak{G}$.
We conclude that $\tilde{\pi}_{n}(w)$ can be identified with $\tilde{\pi}(w)$. 

\end{proof}

Next we examine the center of the diagonal product of the sequence
$\left(\left(\Gamma_{n},S_{n}\right)\right)_{n=1}^{\infty}$. 

\begin{proposition}\label{center2}

Let $\Gamma$ be the diagonal product of the sequence $\left(\left(\Gamma_{n},S_{n}\right)\right)_{n=1}^{\infty}$.
Then the center of $\Gamma$ is isomorphic to the direct sum 
\[
Z\left(\Gamma\right)\simeq\oplus_{n=1}^{\infty}Z\left(\Gamma_{n}\right).
\]

\end{proposition}

\begin{proof}

In the factor $\Gamma_{n}$, the kernel of $\Gamma_{n}\to\mathbb{Z}\wr_{X_{n}}G_{n+3}$
is the center $Z\left(\Gamma_{n}\right)$. By the definition of $\bar{N}_{n}$,
its center is generated by the commutator $\left[b_{x_{n}},b_{y_{n}}\right]$,
where $x_{n}=(1^{n},id)$ and $y_{n}=(1^{n},ab)$. Recall in the marking
$T_{n}$ of $\Gamma_{n}$, the generator $t_{n}=\left(b_{x_{n}},id_{G_{n+3}}\right)$.
Take an element $g\in{\rm St}_{\mathfrak{G}}(n)$ with the section
$g_{1^{n}}=ab$. Take a word $w$ in $\{a,b,c,d\}$ which represents
$g$. Then the image of $w$ in $\Gamma_{n}$ is $\tilde{\pi}_{n}(w)=\left(id,\pi_{n+3}(g)\right)$.
Then in $\Gamma_{n}$, we have 
\[
\tilde{\pi}_{n}(w)t_{n}\tilde{\pi}_{n}(w)^{-1}=\left(b_{x_{n}\cdot\pi_{n+3}(g)},id_{G_{n+3}}\right)=\left(b_{y_{n}},id_{G_{n+3}}\right).
\]
It follows that the center of $\Gamma_{n}$ is generated by $\left[t_{n},\tilde{\pi}_{n}(w)\right]$.

We now show that for the word $w$ chosen in the previous paragraph,
the image of $\left[t,w\right]$ in $\Gamma_{j}$ is trivial, for
all $j\neq n$. 
\begin{itemize}
\item For $j>n$, since the section of $g$ at $1^{n}$ is $ab$, we have
that $1^{j}\cdot g=1^{n}(1^{j-n}\cdot ab)=1^{n}0u$, where $u$ is
a string of length $j-n-1$. Therefore $\tilde{\pi}_{j}(w)t_{j}\tilde{\pi}_{j}(w)^{-1}=\left(b_{\left(1^{n}0u,\gamma\right)},id_{G_{n+3}}\right).$
Note that neither the pair $\left(\left(1^{j},id\right),\left(1^{n}0u,\gamma\right)\right)$
nor the pair $\left(\left(1^{n}0u,\gamma\right),\left(1^{j},id\right)\right)$
can be in the orbit $M_{j}$ of $\left(\left(1^{j},id\right),\left(1^{j},ab\right)\right)$.
Thus the corresponding basis elements commute in $\bar{N}_{j}$. It
follows that $\tilde{\pi}_{j}\left(\left[w,t\right]\right)=id$.
\item For $j\le n-3$, $\pi_{j+3}(g)=id$ because $g\in{\rm St}_{\mathfrak{G}}(n)$.
In this case $x_{j}\cdot\pi_{j+3}(g)=x_{j}$, and it follows that
$\tilde{\pi}_{j}\left(\left[w,t\right]\right)=id$. 
\item For $j\in\{n-2,n-1\}$, we have that 
\[
x_{j}\cdot\pi_{j+3}(g)=\left(1^{j},\pi_{3}\left(g_{1^{j}}\right)\right).
\]
Since $g\in{\rm St}_{\mathfrak{G}}(1^{n})$, we have that $\pi_{1}(g_{1^{j}})=id$.
The pair $(x_{j},x_{j}\cdot\pi_{j+3}(g))$ or $(x_{j}\cdot\pi_{j+3}(g),x_{j})$
being in the orbit of $\left(\left(1^{j},id\right),\left(1^{j},ab\right)\right)$
would imply that $\pi_{1}\left(g_{1^{j}}\right)=a$, a contradiction.
We conclude that $\tilde{\pi}_{j}\left(\left[w,t\right]\right)=id$
for $j=n-2,n-1$. 
\end{itemize}
We have proved that $\tilde{\pi}_{n}\left(\left[t,w\right]\right)$
generates the center of $\Gamma_{n}$ and $\tilde{\pi}_{j}\left(\left[t,w\right]\right)=id$
for $j\neq n$. It follows that $Z(\Gamma_{n})=\ker\left(\Gamma_{n}\to\mathbb{Z}\wr_{X_{n}}G_{n+3}\right)$
is a direct summand in the kernel of the projection $\Gamma\to\Delta$,
where $\Delta$ is the diagonal product of the quotient sequence $\left(\left(\mathbb{Z}\wr_{X_{n}}G_{n+3},\bar{T}_{n}\right)\right)_{n=1}^{\infty}$.
Then we have 
\[
\ker\left(\Gamma\to\Delta\right)=\bigoplus_{n=1}^{\infty}\ker\left(\Gamma_{n}\to\mathbb{Z}\wr_{X_{n}}G_{n+3}\right).
\]

The inclusion $\ker\left(\Gamma\to\Delta\right)\subseteq Z(\Gamma)$
is clear by definitions. It remains to verify that $Z(\Gamma)\subseteq\ker\left(\Gamma\to\Delta\right)$.
Take an element $\gamma\notin\ker\left(\Gamma\to\Delta\right)$. Then
by Lemma there exists an index $j$ such that its projection $\gamma_{j}$
to $\Gamma_{j}$ is not in $Z(\Gamma_{j})$. It follows $\gamma\notin Z(\Gamma)$. 

\end{proof}

\begin{remark}

The statement of Proposition \ref{center2} clearly passes to subsequences:
for any increasing subsequence $\left(n_{i}\right)_{i=1}^{\infty}$,
the diagonal product of $\left(\left(\Gamma_{n_{i}},S_{n_{i}}\right)\right)_{i=1}^{\infty}$
has center isomorphic to the direct sum $\oplus_{i=1}^{\infty}Z\left(\Gamma_{n_{i}}\right)$. 

\end{remark}

We are now ready to prove Theorem \ref{large-center} which states
that there exists a torsion free group of intermediate growth whose
center is isomorphic to $\mathbb{Z}^{\infty}$. The proof proceeds
by showing the diagonal product $\Gamma$ in Proposition \ref{center2}
is of sub-exponential growth, then $\Gamma$ will satisfy the statement
of Theorem \ref{large-center} except that $\Gamma$ has torsion.
To have an example of a torsion free group as stated, one can further
take the diagonal product of $\Gamma$ and a suitable torsion-free
group of intermediate growth. Instead of $\Gamma$ we may take the
diagonal product of a subsequence $\left(\left(\Gamma_{n_{i}},S_{n_{i}}\right)\right)_{i=1}^{\infty}$,
thus there are uncountably many examples satisfying the statement
of Theorem \ref{large-center}. 

\begin{proof}[Proof of Theorem \ref{large-center}]

We first show that $\Gamma$ is of subexponential growth. By construction,
each factor group $\Gamma_{n}$ is virtually nilpotent, thus of polynomial
growth. By Proposition \ref{center1}, $\left(\Gamma_{n},T_{n}\right)$
converges to $(W,T)$ in the Cayley topology when $n\to\infty$. Then
by Lemma \ref{FC}, to show that $\Gamma$ is of exponential growth,
it suffices to show that $W$ is of subexponential growth. 

Recall the embedding $\vartheta$ of $\mathfrak{G}$ into $\left(\mathbb{Z}/2\mathbb{Z}\times\mathbb{Z}/2\mathbb{Z}\right)\wr_{\mathcal{S}}\mathfrak{G}$
in (\ref{eq:theta-em}), where we write $\vartheta(g)=\left(\Phi_{g},g\right)$.
The group $W$ can be viewed as a subgroup of $A\wr_{\mathcal{S}}\mathfrak{G}$,
where $A=\mathbb{Z}\wr\left(\mathbb{Z}/2\mathbb{Z}\times\mathbb{Z}/2\mathbb{Z}\right)$,
via the embedding 
\begin{align*}
W & \to A\wr_{\mathcal{S}}\mathfrak{G}\\
(f,g) & \mapsto\left(\psi_{f},g\right)\mbox{ where }\psi_{f}(v)=\left(\left(f((v,\gamma))\right)_{\gamma\in\mathbb{Z}/2\mathbb{Z}\times\mathbb{Z}/2\mathbb{Z}},\Phi_{g}(v)\right),\ v\in\mathcal{S}.
\end{align*}
Since $A$ contains $\mathbb{Z}^{4}$ as a finite index subgroup,
it is of polynomial growth. Then by \cite[Lemma 5.1]{BE1}, the growth
function of the permutation wreath product $A\wr_{\mathcal{S}}\mathfrak{G}$
is equivalent to $\exp\left(n^{\alpha_{0}}\log n\right)$. It follows
that the subgroup $W$ also has growth function bounded by $\exp\left(n^{\alpha_{0}}\log n\right)$. 

Combined with Proposition \ref{center2}, we have that for any subsequence
$(n_{i})$, the diagonal product $\Delta=\Delta_{(n_{i})}$ is a group
of subexponential growth whose center is isomorphic to $\mathbb{Z}^{\infty}$. 

The first example of a torsion free group of intermediate growth is
constructed by Grigorchuk in \cite{Grigorchuk85}. More precisely,
in \cite{Grigorchuk85} a torsion-free group $\tilde{G}$ is constructed,
which is an extension of $\mathfrak{G}$ and of intermediate growth.
The group $\tilde{G}$ can be described as follows. Take the subgroup
$\Lambda$ of $\mathbb{Z}^{3}\wr_{\mathcal{S}}\mathfrak{G}$ generated
by 
\[
\tilde{a}=({\bf 0},a),\ \tilde{b}=(\delta_{1^{\infty}}^{e_{1}},b),\ \tilde{c}=(\delta_{1^{\infty}}^{e_{2}},c),\ \tilde{d}=(\delta_{1^{\infty}}^{e_{3}},d),
\]
where $\left\{ e_{1},e_{2},e_{3}\right\} $ is the standard basis
of $\mathbb{Z}^{3}$. Note that $\vartheta(\mathfrak{G})\simeq\mathfrak{G}$
is a marked quotient of $\Lambda$ and 
\[
\ker\left(\Lambda\to\mathfrak{G}\right)=\oplus_{v\in\mathcal{S}}\left\langle e_{1}^{2},e_{2}^{2},e_{3}^{2}\right\rangle \simeq\oplus_{v\in\mathcal{S}}\mathbb{Z}^{3}.
\]
Take $\mathbb{Z}=\left\langle \hat{a}\right\rangle $ and mark it
with the generating tuple $\left(\hat{a},\hat{b},\hat{c},\hat{d}\right)$,
where $\hat{b}=\hat{c}=\hat{d}=0$. Then take $\tilde{G}$ to be the
diagonal product of $\left(\Lambda,\left(\tilde{a},\tilde{b},\tilde{c},\tilde{d}\right)\right)$
and $\left(\mathbb{Z},(\hat{a},\hat{b},\hat{c},\hat{d})\right)$.
Add a trivial generator $t=id$ and mark $\tilde{G}$ with $\tilde{T}=\{a,b,c,d,t\}$.
By \cite{Grigorchuk85}, $\tilde{G}$ is a torsion free group of intermediate
growth. 

Finally, take the diagonal product $\tilde{\Delta}$ of $\left(\Delta,T\right)$
and $\left(\tilde{G},\tilde{T}\right)$. Since both $\left(\Delta,T\right)$
and $\left(\tilde{G},\tilde{T}\right)$ are of sub-exponential growth,
$\tilde{\Delta}$ is of subexponential growth as well. We now show
$\tilde{\Delta}$ is torsion free. If the image of a word $w$ in
$\left\{ a,b,c,d,t\right\} $ in $\tilde{\Delta}$ is a non-trivial
finite order element, then its projection to $\tilde{G}$ must be
trivial since $\tilde{G}$ is torsion free. It follows that the projection
of $w$ to $\Delta$ is a nontrivial finite order element in $\ker\left(\Delta\to\mathfrak{G}\right)$.
However by its construction, $\ker\left(\Delta\to\mathfrak{G}\right)$
is torsion free, a contradiction. 

The center of $\tilde{G}$ is $\left\langle a^{2}\right\rangle $
and the center of $\Delta$ is contained in the normal closure of
$t$. Since $a^{2}$ evaluates to identity in $\Delta$ and $t$ evaluates
to identity in $\tilde{G}$, we conclude that 
\[
Z\left(\tilde{\Delta}\right)=Z\left(\Delta\right)\times Z(\tilde{G})\simeq\oplus_{i=1}^{\infty}\ker\left(\Gamma_{n_{i}}\to\mathbb{Z}\wr_{X_{n_{i}}}G_{n_{i}+3}\right)\oplus\left\langle a^{2}\right\rangle \simeq\mathbb{Z}^{\infty}.
\]
The group $\tilde{\Delta}$ satisfies all the requirements. 

\end{proof}

\section{Extensions of $W_{\omega}$\label{sec:construction}}

\subsection{The construction of diagonal products\label{subsec:construction}}

Let $U,V$ be two nontrivial finite groups such that one of them has
at least $3$ elements. In this section we consider the extensions
of $W_{\omega}=(U\times V)\wr_{\mathcal{S}}G_{\omega}$, where $\mathcal{S}$
is the orbit of $1^{\infty}$ under the action of the Grigorchuk $G_{\omega}$.
Recall that for the sequence $\omega=\omega_{0}\omega_{1}\ldots\in\{0,1,2\}^{\infty}$,
each symbol $\omega_{k}$ corresponds to a homomorphism $\{id,b,c,d\}\to\{id,a\}$.

Similar to the extensions of $\mathfrak{G}$ in Section \ref{sec:definition},
the input to the construction is a sequence of finite marked quotients
of $U\ast V$. Enumerate the non-identity elements of $U$ and $V$
as $\{u_{1},\ldots,u_{p}\}$, $p=|U|-1$ and $\{v_{1},\ldots,v_{q}\}$,
$q=|V|-1$. Let $\mathcal{L}=\left(F_{n}\right)_{n=1}^{\infty}$ be
a sequence of quotients of $U\ast V$, where each $L_{n}$ is marked
with the image of $\left(u_{1}\ldots,u_{p},v_{1},\ldots v_{p}\right)$
under the projection $U\ast V\to F_{n}$. 

Consider the permutation wreath product $\Delta_{n}=F_{n}\wr_{\mathsf{L}_{n}}G_{\omega}$,
marked with the following specific generating tuple
\begin{align}
T_{n} & =\left(\left(id,a\right),\left(id,b\right),\left(id,c\right),\left(id,d\right),u_{1,n},\ldots,u_{p,n},v_{1,n},\ldots,v_{q,n}\right),\label{eq:Tn}\\
\mbox{where } & u_{i,n}=\left(\delta_{1^{n}}^{u_{i}},id\right)\mbox{ and }v_{j,n}=\left(\delta_{1^{n-1}0}^{v_{j}},id\right),\ 1\le i\le p,\ 1\le j\le q.\nonumber 
\end{align}
With slight abuse of notation we still write $a$ for the generator
$(id,a)$, similarly for $b,c,d$. We now explain the choice of generators
$u_{i,n}$ and $v_{j,n}$, which mimics the wreath recursion in $G_{\omega}$.
Recall that at each level of the rooted binary tree $\mathsf{T}$,
the orbital Schreier graph of $1^{n}$ is a finite line segment of
$2^{n}$ vertices, connected with self-loops and multiple edges. The
vertex $1^{n}$ is at one end of the Schreier graph and $1^{n-1}0$
at the other end. Thus on the Schreier graph the generators $\left\{ u_{1},\ldots,u_{p}\right\} $
are placed at one end $1^{n}$; the generators $\left\{ v_{1},\ldots,v_{q}\right\} $
are placed at the other end $1^{n-1}0$. It will be clear later that
the choice of these two locations is essential: they are at the opposite
ends of the finite Schreier graph (thus distance $2^{n}-1$ apart)
and at the same time they are siblings on the tree $\mathsf{T}$.

Let ${\bf M}$ be the free product 
\[
{\bf M}=(\mathbb{Z}/2\mathbb{Z})\ast\left(\mathbb{Z}/2\mathbb{Z}\times\mathbb{Z}/2\mathbb{Z}\right)\ast U\ast V,
\]
marked with generating tuple ${\bf T}=\left\{ a,b,c,d,u_{1}\ldots,u_{p},v_{1}\ldots,v_{q}\right\} $.
Then by the definitions, each $\Delta_{n}$ is a marked quotient of
$\mathbf{M}$ with the generating set ${\bf T}$ projected onto $T_{n}$. 

\begin{definition}\label{defdelta}

Given a sequence of marked quotients $\mathcal{L}=(F_{n})_{n=1}^{\infty}$
of $U\ast V$, the group $\Delta=\Delta\left(\mathcal{L},\omega\right)$
is defined as the diagonal product of marked groups $\left(\Delta_{n},T_{n}\right)_{n=1}^{\infty}$,
where $\Delta_{n}=F_{n}\wr_{\mathsf{L}_{n}}G_{\omega}$ and the marking
$T_{n}$ is given in (\ref{eq:Tn}). Denote the marking of $\Delta$
as $T=\left\{ a,b,c,d,\tilde{u}_{1}\ldots,\tilde{u}_{p},\tilde{v}_{1}\ldots,\tilde{v}_{q}\right\} $. 

\end{definition}

\subsection{The FC-center of $\Delta(\mathcal{L},\omega)$}

Mark $W_{\omega}=\left(U\times V\right)\wr_{\mathcal{S}}G_{\omega}$
by the generating tuple 
\[
T_{0}=\left\{ a,b,c,d,\left(\delta_{1^{\infty}}^{u_{1}},id\right),\ldots\left(\delta_{1^{\infty}}^{u_{p}},id\right),\left(\delta_{1^{\infty}}^{v_{1}},id\right),\ldots\left(\delta_{1^{\infty}}^{v_{q}},id\right)\right\} .
\]
Given a vertex $v=v_{1}\ldots v_{n}\in\mathsf{L}_{n}$, denote by
$\check{v}$ its sibling, that is the other child of $v_{1}\ldots v_{n-1}$.
Note the following. 

\begin{fact}\label{Delta-W}

Let $\mathcal{L}=(F_{n})_{n=1}^{\infty}$ be a sequence of marked
quotients of $U\ast V$. Then the sequence $\left(\Delta_{n},T_{n}\right)_{n=1}^{\infty}$
converges to $\left(W_{\omega},T_{0}\right)$ in the Cayley topology.

\end{fact}

\begin{proof}

Because the $U$-generators and the $V$-generators are place $2^{n}-1$
apart on the level $n$ Schreier graph in $T_{n}$, the ball of radius
$2^{n-1}-1$ around identity in $\left(\Delta_{n},T_{n}\right)$ coincide
with the ball of same radius around identity in $\left(W_{\omega},T_{0}\right)$.
An explicit identification is given by for $(f,g)\in\Delta_{n}$,
where the function $f$ satisfies that for $d(1^{n},v)\le2^{n-1}-1$,
$f(v)\in U$ and for $d(1^{n},v)>2^{n-1}-1,$ $f(v)\in V$, it is
identified with $\left(\phi,g\right)\in W_{\omega}$, where 
\[
\phi(v1^{\infty})=\left(f(v),f(\check{v})\right)\in U\times V,\ \mbox{for }d(1^{n},v)\le2^{n-1}-1;\ \phi(x)=id\mbox{ otherwise}.
\]

\end{proof}

Given an element $\gamma=(\gamma_{n})\in\Delta$, denote by $\gamma_{\infty}$
its projection to $W_{\omega}$ as explained Fact \ref{Delta-W}.
Fact \ref{Delta-W} implies the following.

\begin{fact}\label{FC-delta2}

Suppose each group $F_{n}$ in $\mathcal{L}$ is an FC-group, then
the FC-center of $\Delta=\Delta(\mathcal{L},\omega)$ is
\[
Z_{{\rm FC}}(\Delta)=\ker\left(\Delta\to W_{\omega}\right).
\]

\end{fact}

\begin{proof}

Since the group $W_{\omega}$ is ICC, we have that $Z_{\mathbf{FC}}(\Delta)\subseteq\ker\left(\Delta\to W_{\omega}\right)$.
To show containment in the other direction, for any element $\gamma\in\ker\left(\Delta\to W_{\omega}\right)$,
since $(\Delta_{n},T_{n})\to(W_{\omega},T_{0})$ when $n\to\infty$,
there exists an index $n_{0}$ such that the projection $\gamma_{n}$
of $\gamma$ to $\Delta_{n}$ is trivial for all $n\ge n_{0}$. It
follows also that the projection of $\gamma$ to $G_{\omega}$ is
trivial. Then the conjugacy class of $\gamma$ in $\Delta$ is contained
in the finite product of the conjugacy classes of $\gamma_{n}$ in
$\oplus_{x\in\mathsf{L}_{n}}F_{n}$, $n<n_{0}$. We conclude that
$\gamma$ is in the FC-center of $\Delta$. 

\end{proof}

We now examine the kernel of $\Delta\left(\mathcal{L},\omega\right)\to W_{\omega}$.
Denote by $\left\langle [U,V]\right\rangle ^{F_{n}}$ the normal closure
of $\left[U,V\right]$ in $F_{n}$.

\begin{lemma}\label{KDelta}

Suppose $\mathcal{L}=\left(F_{n}\right)_{n=1}^{\infty}$ is a sequence
of marked quotients of $U\ast V$. Then 
\[
\ker\left(\Delta\to W_{\omega}\right)\supseteq\oplus_{n=1}^{\infty}\oplus_{x\in\mathsf{L}_{n}}\left\langle [U,V]\right\rangle ^{F_{n}}.
\]

\end{lemma}

\begin{proof}

The proof is similar to Lemma \ref{K-split} or Lemma \ref{center2}.
In the factor group $\Gamma_{n}$, take $g$ to be a shortest element
in $G_{\omega}$ such that $1^{n-1}0=1^{n}\cdot g$. The word length
of $g$ is $2^{n}-1$. Then 
\begin{align*}
gv_{j,n}g^{-1} & =\left(\delta_{1^{n-1}0\cdot g^{-1}}^{v_{j}},id\right)=\left(\delta_{1^{n}}^{v_{j}},id\right),\\
\left[u_{i,n},gv_{j,n}g^{-1}\right] & =\left(\delta_{1^{n}}^{\left[u_{i},v_{j}\right]},id\right).
\end{align*}
For the choice of $g$ above, for any $k<n$, $1^{k}\cdot g=1^{k}$.
For $k>n$, since $d(1^{k},1^{k-1}0)=2^{k}-1$, we have that $1^{k}\cdot g\neq1^{k-1}0$.
Therefore if $k\neq n$, then $1^{k-1}0\cdot g^{-1}\neq1^{k}$ and
\begin{align*}
gv_{j,k}g^{-1} & =\left(\delta_{1^{k-1}0\cdot g^{-1}}^{v_{j}},id\right),\\
\left[u_{i,k},gv_{j,k}g^{-1}\right] & =\left(\left[\delta_{1^{k}}^{u_{i}},\delta_{1^{k-1}0\cdot g^{-1}}^{v_{j}}\right],id\right)=id.
\end{align*}
The calculation above implies that for each $n$, $\left\langle [U,V]\right\rangle ^{F_{n}}$
is a direct summand in $\ker\left(\Gamma\to W_{\omega}\right)$. Taking
the normal closure of these summands, we obtain the statement.

\end{proof}

Write $\bar{F}_{n}=F_{n}/\left\langle \left[U,V\right]\right\rangle ^{F_{n}}$
and $\bar{\pi}_{n}:F_{n}\to\bar{F_{n}}$. With slight abuse of notation
we also write $\bar{\pi}_{n}$ for the projection $U\times V\to\bar{F}_{n}$.
Write $\pi_{U}:U\times V\to V$, $\pi_{V}:U\times V\to V$. Recall
that $\check{v}$ is the sibling of $v$. Suppose now $U$ and $V$
are abelian. Consistent with the markings $T_{0}$ and $T_{n}$, there
is a projection map 
\begin{align*}
\varrho_{n} & :\left(U\times V\right)\wr_{\mathcal{S}}G_{\omega}\to\bar{F}_{n}\wr_{\mathsf{L}_{n}}G_{\omega}\\
 & \left(f,g\right)\mapsto\left(\psi_{f},g\right),
\end{align*}
where 
\[
\psi_{f}(v)=\sum_{z\in\mathcal{S}}\bar{\pi}_{n}\left(\pi_{U}(f(vz)),\pi_{V}\left(f(\check{v}x)\right)\right).
\]
This map is a homomorphism because for any tree automorphism $g$,
if $v=u\cdot g$ then $\check{v}=\check{u}\cdot g$. By induction
on word length, one can verify the following diagram commutes:

\[ \begin{tikzcd}\label{consist}
\Delta \arrow{r}{\pi_n} \arrow[swap]{d}{\pi} & 
\Delta_n \arrow{d}{} \\% 
W_{\omega} \arrow{r}{\varrho_n}&  \bar{F}_{n}\wr_{\mathsf{L}_{n}}G_{\omega}.
\end{tikzcd} \] In other word, for any element $\gamma=(\gamma_{n})_{n=1}^{\infty}\in\Delta$,
$\gamma_{n}$ and $\gamma_{\infty}$ have consistent projections to
$\bar{F}_{n}\wr_{\mathsf{L}_{n}}G_{\omega}$. 

\begin{corollary}\label{FC-abelianUV}

Suppose $\mathcal{L}=\left(F_{n}\right)_{n=1}^{\infty}$ is a sequence
of marked quotients of $U\ast V$, where $U,V$ are nontrivial finite
abelian groups. Suppose each group $F_{n}$ in $\mathcal{L}$ is an
FC-group. Then the FC-center of $\Delta=\Delta(\mathcal{L},\omega)$
is 
\[
Z_{{\rm FC}}(\Delta)=\oplus_{n=1}^{\infty}\oplus_{x\in\mathsf{L}_{n}}\left\langle [U,V]\right\rangle ^{F_{n}}.
\]

\end{corollary}

\begin{proof}

By Fact \ref{FC-delta2} and Lemma \ref{KDelta}, it remains to show
that $\ker\left(\Delta\to W_{\omega}\right)\subseteq\oplus_{n=1}^{\infty}\oplus_{x\in\mathsf{L}_{n}}\left\langle [U,V]\right\rangle ^{F_{n}}$.
If $\gamma\in\ker\left(\Delta(\mathcal{L},\omega)\to W_{\omega}\right)$,
then as in the proof of Fact \ref{FC-delta2}, there exists an index
$n_{0}$ such that the projection $\gamma_{n}$ of $\gamma$ to $\Delta_{n}$
is trivial for all $n\ge n_{0}$. For $n<n_{0}$, $\gamma_{n}\in\oplus_{x\in\mathsf{L}_{n}}F_{n}$.
Suppose there is an index $n$ and $x\in\mathsf{L}_{n}$ such that
$(\gamma_{n})_{x}\notin\left\langle \left[U,V\right]\right\rangle ^{F_{n}}$,
in other words the projection of $\left(\gamma_{n}\right)_{x}$ to
$F_{n}/\left\langle \left[U,V\right]\right\rangle ^{F_{n}}$ is nontrivial.
We have then the projection $\gamma_{\infty}$ of $\gamma$ to $W_{\omega}$
is nontrivial, contradicting with the choice of $\gamma$. It follows
that $\gamma\in\oplus_{n<n_{0}}\oplus_{x\in\mathsf{L}_{n}}\left\langle \left[U,V\right]\right\rangle ^{F_{n}}$. 

\end{proof}

\subsection{Recursions for $\Delta\left(\mathcal{L},\omega\right)$}

Let a sequence $\mathcal{L}=\left(F_{n}\right)_{n=1}^{\infty}$ of
marked quotients of $U\ast V$ and a string $\omega=\omega_{0}\omega_{1}\ldots\in\{0,1,2\}^{\infty}$
be given. We continue to use notations introduced in the previous
subsection. Denote by $\Delta_{>k}$ the diagonal product of factors
with index $n>k$, $\left(\left(\Delta_{n},T_{n}\right)\right)_{n=k+1}^{\infty}$.
We will also consider the group $\Delta\left(\mathfrak{s}^{k}\mathcal{L},\mathfrak{s}^{k}\omega\right)$
with shifted parameters, where $\mathfrak{s}^{k}\mathcal{L}=\left(F_{k+1},F_{k+2},\ldots\right)$
and $\mathfrak{s}^{k}\omega=\omega_{k}\omega_{k+1}\ldots$. 

Similar to Section \ref{sec:definition}, we first consider a formal
recursion on the level of the free product $\mathbf{M}$, then project
down to its quotients. Recall that $\mathbf{M}$ is the free product
$(\mathbb{Z}/2\mathbb{Z})\ast\left(\mathbb{Z}/2\mathbb{Z}\times\mathbb{Z}/2\mathbb{Z}\right)\ast U\ast V,$
marked with the generating tuple $\mathbf{T}$. Denote by $\mathfrak{S}_{2}$
the symmetric group of $\{0,1\}$, generated by the involution $\varepsilon=(0,1)$.
Consider the homomorphism 
\begin{align*}
\boldsymbol{\theta} & :\mathbf{M}\to\mathbf{M}\wr_{\{0,1\}}\mathfrak{S}_{2}\\
 & \mathbf{a}\mapsto(id,\varepsilon),\ b\mapsto\left(\delta_{1}^{\mathbf{b}}+\delta_{0}^{\mathbf{\omega_{0}(\mathbf{b})}},id\right),\ c\mapsto\left(\delta_{1}^{\mathbf{c}}+\delta_{0}^{\omega_{0}(\mathbf{c})},id\right),\ d\mapsto\left(\delta_{1}^{\mathbf{d}}+\delta_{0}^{\omega_{0}(\mathbf{d})},id\right)\\
 & u_{i}\mapsto\left(\delta_{1}^{u_{i}},id\right),\ v_{j}\mapsto\left(\delta_{1}^{v_{j}},id\right),\ 1\le i\le p\mbox{ and }1\le j\le q.
\end{align*}
Denote by $\pi_{1}$ the projection $\mathbf{M}\wr_{\{0,1\}}\mathfrak{S}_{2}\to\Delta(\mathfrak{s}\mathcal{L},\mathfrak{s}\omega)\wr_{\{0,1\}}\mathfrak{S}_{2}$
induced by the marked projection $\mathbf{M}\to\Delta(\mathfrak{s}\mathcal{L},\mathfrak{s}\omega)$. 

\begin{lemma}\label{M-recursion}

The homomorphism $\boldsymbol{\theta}$ induces an embedding $\theta:\Delta_{>1}\to\Delta\left(\mathfrak{s}\mathcal{L},\mathfrak{s}\omega\right)\wr_{\{0,1\}}\mathfrak{S}_{2}$
and the following diagram commute:

\[ \begin{tikzcd}
\mathbf{M} \arrow{r}{\boldsymbol{\theta}} \arrow[swap]{d}{\pi} & 
\boldsymbol{\theta}(\mathbf{M}) \arrow{d}{\pi_1} \\% 
\Delta_{>1} \arrow{r}{\theta}&  \theta(\Delta_{>1}).
\end{tikzcd} \]

\end{lemma}

\begin{proof}

We first describe $\theta$. The group $\Delta_{>1}$ is defined as
the diagonal product of $\left(\left(\Delta_{n},T_{n}\right)\right)_{n=2}^{\infty}$.
Componentwise, for each $n\ge2$, we have the embedding 
\begin{align*}
\theta_{n} & :F_{n}\wr_{\mathsf{L}_{n}}G_{\omega}\to\left(F_{n}\wr_{\mathsf{L}_{n-1}}G_{\mathfrak{s}\omega}\right)\wr_{\{0,1\}}\mathfrak{S}_{2},\\
 & (f,g)\mapsto\left(\psi_{f,g},\pi_{1}(g)\right),
\end{align*}
where $\psi_{f,g}:\{0,1\}\to F_{n}\wr_{\mathsf{L}_{n-1}}G_{\mathfrak{s}\omega}$
is 
\[
\psi_{f,g}(i)=\left(\left(f(iu)\right)_{u\in\mathsf{L}_{n-1}},g_{i}\right),\ i\in\{0,1\},
\]
and $(g_{0},g_{1})\pi_{1}(g)$ is the image of $g$ under the canonical
wreath recursion in ${\rm Aut}(\mathsf{T})$. Explicitly, for the
generators in $T_{n}$ (\ref{eq:Tn}), we have 
\begin{align*}
\theta_{n}\left(x\right) & =(id,\varepsilon);\\
\theta_{n}\left(x\right) & =\left(\psi_{x},id\right),\ \mbox{where }\psi_{x}(0)=\left(id,\omega_{0}(x)\right),\ \psi_{x}(1)=\left(id,x\right),\ x\in\{b,c,d\};\\
\theta_{n}\left(u_{i,n}\right) & =\left(\delta_{1}^{u_{i}},id\right),\ \theta_{n}\left(v_{j,n}\right)=\left(\delta_{1}^{v_{j}},id\right),\ 1\le i\le p,\ 1\le j\le q.
\end{align*}

The map $\theta$ on $\Delta_{>1}$ is defined as for $\gamma=\left(\gamma_{n}\right)_{n=2}^{\infty}$,
where $\gamma_{n}\in\Delta_{n}$, $\theta\left(\left(\gamma_{n}\right)_{n=2}^{\infty}\right)=\left(\theta_{n}(\gamma_{n})\right)_{n=2}^{\infty}$.
By inspecting the generators, we have that $\theta$ maps $\Delta_{>1}$
into the permutation wreath product $\Delta\left(\mathfrak{s}\mathcal{L},\mathfrak{s}\omega\right)\wr_{\{0,1\}}\mathfrak{S}_{2}$,
where $a\mapsto(id,\varepsilon)$, $x\mapsto\left(\delta_{1}^{x}+\delta_{0}^{\omega_{0}(x)},id\right)$
for $x\in\{b,d,d\}$, $\tilde{u}_{j}\mapsto\left(\delta_{1}^{u_{j}},id\right)$
and $\tilde{v}_{j}\mapsto\left(\delta_{1}^{v_{j}},id\right)$. The
map $\theta$ is an embedding because componentwise, each $\theta_{n}$
is an embedding.

The commutative diagram can be checked by induction on the word length
of $w\in\mathbf{M}$. 

\end{proof}

Iterate the embedding $\theta$, we have that the group $\Delta_{>n}=\Delta_{>n}\left(\mathcal{L},\omega\right)$
embeds into the permutation wreath product $\Delta(\mathfrak{s}^{n}\mathcal{L},\mathfrak{s}^{n}\omega)\wr_{\mathsf{L}_{n}}\pi_{n}(G_{\omega})$.
Explicitly, the image of the generating tuple $T$ under $\theta^{n}$
is:

\begin{align}
\theta^{n}:\Delta_{>n} & \to\Delta(\mathfrak{s}^{n}\mathcal{L},\mathfrak{s}^{n}\omega)\wr_{\mathsf{L}_{n}}\pi_{n}(G_{\omega})\label{eq:theta_n}\\
a & \mapsto(id,a),\nonumber \\
x & \mapsto\left(\delta_{1^{n-1}0}^{\omega_{n-1}(x)}+\delta_{1^{n}}^{x},s\right),\mbox{ for generator }x\in\{b,c,d\},\nonumber \\
\gamma & \mapsto\left(\delta_{1^{n}}^{\gamma},id\right),\mbox{ for generator }\gamma\in\left\{ \tilde{u}_{1},\ldots,\tilde{u}_{p},\tilde{v}_{1},\ldots,\tilde{v}_{1}\right\} .\nonumber 
\end{align}

\section{The traverse fields and contraction properties\label{sec:tr-metric}}

In this section, we consider the traverse field associated with a
word $w$ in the letters $\{a,b,c,d\}$, under the action of the first
Grigorchuk group $\mathfrak{G}$ on $\mathsf{T}$. We show that the
traverse fields are compatible with the formal recursion on words.
In what follows $a,b,c,d$ acts on $\mathsf{T}$ as the generators
of $\mathfrak{G}$. 

Given a word $w=z_{1}...z_{m}$ in the alphabet $\left\{ a,b,c,d\right\} $
and a level $n$, consider the inverted orbits of the pair $\left(1^{n},1^{n-1}0\right)$
under the action of $\mathfrak{G}$:
\[
\mathcal{I}_{n}(w):=\left(\left(1^{n},1^{n-1}0\right),\left(1^{n},1^{n-1}0\right)\cdot z_{1}{}^{-1},...,\left(1^{n},1^{n-1}0\right)\cdot\left(z_{1}\ldots z_{m}\right){}^{-1}\right).
\]
Note that $\mathcal{I}_{n}(w)$ is an ordered sequence of points rather
than a set. 

For a vertex $v\in\mathsf{L}_{n}$, keep a record of the pattern of
visits from the inverted orbit $\mathcal{I}_{n}(w)$. Formally, for
$x\in\mathsf{L}_{n}$, let $\tilde{P}(x,w)$ be a string in $\{0,1\}$
which is defined recursively as: 
\[
\tilde{P}(x,id)=\begin{cases}
1 & \mbox{if }x=1^{n},\\
0 & \mbox{if }x=1^{n-1}0\\
\emptyset & \mbox{otherwise};
\end{cases},
\]
and for $s\in\{a,b,c,d\}$,
\[
\tilde{P}(x,ws)=\begin{cases}
\tilde{P}(x,w) & \mbox{ if }x\notin\left\{ 1^{n}\cdot(ws)^{-1},1^{n-1}0\cdot(ws)^{-1}\right\} ,\\
\tilde{P}(x,w)1 & \mbox{ if }x=1^{n}\cdot(ws)^{-1},\\
\tilde{P}(x,w)0 & \mbox{ if }x=1^{n-1}0\cdot(ws)^{-1}.
\end{cases}
\]
After collapse consecutive $1$'s into a single $1$, consecutive
$0$'s into a single $0$, we obtain from $\tilde{P}(x,w)$ a string
$P(x,w)$ alternating in $0$ and $1$. For example, if $\tilde{P}(x,w)=000111100$
then $P(x,w)=010$. 

\begin{definition}[Traverse field]

The collection $\left\{ P(x,w)\right\} _{x\in\mathsf{L}_{n}}$ defined
above is called the \emph{traverse field} of the word $w$ on level
$n$ under the action of $\mathfrak{G}$. 

\end{definition}

The following example illustrates the definitions.

\begin{example}

Let $n=2$ and $w=abaca$. Then 
\[
\mathcal{I}_{2}(abaca)=\left(\left(11,10\right),\left(01,00\right),\left(01,00\right),\left(10,11\right),\left(10,11\right),\left(00,01\right)\right).
\]
For the vertex $v=01$, we have $\tilde{P}(01,abaca)=110$ and $P(01,abaca)=10$. 

\end{example}

Consider the formal recursion $\boldsymbol{\varphi}$ on the level
of words where 
\[
a\mapsto\left(\emptyset,\emptyset\right)\varepsilon,\ b\mapsto\left(a,c\right),\ c\mapsto(a,d),\ d\mapsto(\emptyset,b).
\]
Write $\boldsymbol{\varphi}(w)=(w_{0},w_{1})\varepsilon^{s}$. Note
that no reduction is performed on the words $w_{0},w_{1}$. Under
the formal recursion we have the following. Given two strings $u=u_{1}\ldots u_{k}$
and $v=v_{1}\ldots v_{\ell}$ in $\{0,1\}$, we write $u\subset v$
if there is an increasing, injective map $\tau:\{1,\ldots,k\}\to\left\{ 1,\ldots,\ell\right\} $
such that $v_{\tau(i)}=u_{i}$ for all $1\le i\le k$. In other words,
$u\subset v$ if the string $u$ can be embedded into $v$ in an order
preserving way. 

\begin{lemma}\label{P-recursion}

Let $w$ be a word in $\{a,b,c,d\}$ and $\boldsymbol{\varphi}(w)=(w_{0},w_{1})\varepsilon^{s}$.
Then for $x\in\{0,1\}^{n}$, $n\ge1$, we have that 
\[
P(0x,w)\subset P\left(x,w_{0}\right)\mbox{ and }P(1x,w)\subset P\left(x,w_{1}\right).
\]

\end{lemma}

\begin{proof}

We prove the claim by induction on the length of $w$. It is trivially
true for the empty word. Suppose the claim is true for all words of
length at most $m$. 

Let $w$ be a word of length $m+1$. Write $w=w'x$, $x\in\{a,b,c,d\}$
and $\boldsymbol{\varphi}(w')=(w_{0}',w_{1}')\varepsilon^{s}$.

Case 1: $w$ ends in $a$. For $s=0$, the new point in the inverted
orbit $\mathcal{I}_{n}(w)$ is 
\[
\left(1^{n},1^{n-1}0\right)\cdot w{}^{-1}=\left(0\left(1^{n-1}\cdot(w_{0}')^{-1}\right),0\left(1^{n-2}0\cdot(w_{0}')^{-1}\right)\right).
\]
In the left subtree, note that $(u,v)=\left(1^{n-1}\cdot(w_{0}')^{-1},1^{n-2}0\cdot(w_{0}')^{-1}\right)$
is the same as the last point in the inverted orbit of $(1^{n-1},1^{n-2}0)$
under $w_{0}'$. By the induction hypothesis, we have $P(0x,w')\subset P(x,w_{0}')$.
Since the new point $(u,v)$ is a repetition of the last point in
the inverted orbit under $w_{0}'$, we have that $P(0x,w)\subset P(x,w_{0}')$
as well. On the right subtree we have $P(1x,w)=P(1x,w')$, which is
contained by $P(x,w_{1}')$ by the induction hypothesis. The argument
for $s=1$ is the same with left and right subtrees swapped. 

Case 2: $w$ ends in $b,c,d$. In this case, the last two points in
the inverted orbit $\mathcal{I}_{n}(w)$ are the same: $\left(1^{n},1^{n-1}0\right)\cdot w{}^{-1}=\left(1^{n},1^{n-1}0\right)\cdot w'{}^{-1}$.
Therefore $P(ix,w)=P(ix,w')$ for $i\in\{0,1\}$, $x\in\left\{ 0,1\right\} ^{n}$.
Since there is no reduction on the words, $w_{i}'$ is a prefix of
$w_{i}$, thus $P(x,w_{i}')\subset P(x,w_{i})$. It follows by the
induction hypothesis that $P(ix,w)=P(ix,w')\subset P(x,w_{i}')\subset P(x,w_{i})$. 

\end{proof}

Next we consider admissible word reductions. Following \cite{BE1},
we say a word $w$ is \emph{pre-reduced}, if it does not contain consecutive
occurrences of $b,c,d$ (while consecutive occurrences of $a$ is
allowed). Given a word $w$, denote by $\mathring{w}$ its pre-reduction.
Since $\left(1^{n},1^{n-1}0\right)$ is fixed by $b,c,d$, by the
definitions we have:

\begin{fact}\label{preduction}

Let $n\ge1$. Then for $x\in\mathsf{L}_{n}$, $P(x,w)=P(x,\mathring{w})$,
that is, on level $n$, the traverse fields of $w$ and its pre-reduction
$\mathring{w}$ are the same. 

\end{fact}

Since $\left(1^{n},1^{n-1}0\right)$ is not fixed by $a$, inserting
or deleting $a^{2}$ may change the traverse field. This is the reason
that we consider pre-reduced words, instead of reduced words in the
free product $\mathbf{F}$. 

In \cite[Proposition 4.2]{Bartholdi98} the following length contraction
bound is established, see also \cite[Lemma 4.2]{BE1}: there is a
norm $\left\Vert \cdot\right\Vert $ such that for any pre-reduced
word $w$ in $\{a,b,c,d\}$, we have 
\begin{equation}
\left\Vert w_{0}\right\Vert +\left\Vert w_{1}\right\Vert \le\eta\left\Vert w\right\Vert +C,\label{eq:eta-contraction}
\end{equation}
where $\boldsymbol{\varphi}(w)=(w_{0},w_{1})\varepsilon^{s}$, $\eta$
is the real root of $X^{3}+X^{2}+X-1$, $C=\eta\left\Vert a\right\Vert $. 

By Lemma \ref{P-recursion} and Fact \ref{preduction}, we may apply
the contraction inequality (\ref{eq:eta-contraction}) to the traverse
fields. Denote by $A\left(n,w\right)$ the sum 
\begin{equation}
A\left(n,w\right):=\sum_{x\in\mathsf{L}_{n}}\left|P(x,w)\right|,\label{eq:A-1}
\end{equation}
where $\left|\cdot\right|$ is the length of the string. 

\begin{lemma}\label{A-contract-1}

There exists a constant $C>0$ such that for any pre-reduced word
$w$ in $\{a,b,c,d\}$ and $n\ge1$, we have
\[
A(n,w)\le C\left(\eta^{k}|w|+2^{k}\right)\mbox{ for any }1\le k\le n-2.
\]

\end{lemma}

\begin{proof}

Lemma \ref{P-recursion} implies that for a word $w$ such that $\boldsymbol{\varphi}(w)=(w_{0},w_{1})\varepsilon^{s}$,
$s\in\{0,1\}$, we have
\begin{equation}
A\left(n,w\right)\le A\left(n-1,w_{0}\right)+A\left(n-1,w_{1}\right).\label{eq:sum-1-1}
\end{equation}
By Fact \ref{preduction}, $A\left(n-1,w_{i}\right)=A\left(n-1,\mathring{w}_{i}\right)$.
Iterating $k$ times, we have that 
\[
A(n,w)\le\sum_{v\in\mathsf{L}_{k}}A(n-k,w_{v}),
\]
where each step of the recursion involves first applying $\varphi$
to the word $w$, then pre-reducing $w_{0}$ and $w_{1}$. 

Iterate the contraction inequality (\ref{eq:eta-contraction}), we
have that 
\[
\sum_{v\in\mathsf{L}_{k}}\left\Vert w\right\Vert _{v}\le\eta^{k}\left\Vert w\right\Vert +C2^{k}\frac{\eta}{1-\eta/2},
\]
where $\left\Vert \cdot\right\Vert $ is the norm in the contraction
inequality (\ref{eq:eta-contraction}). 

By the definition of the traverse field, we have the obvious bound
\[
A(k,w)=\sum_{x\in\mathsf{L}_{k}}\left|P(x,w)\right|\le2|w|+2\le2C'\left\Vert w\right\Vert +2.
\]
Putting these bounds together, we have 
\[
A(n,w)\le\sum_{v\in\mathsf{L}_{n-2}}\left(2C'\left\Vert w_{v}\right\Vert +2\right)\le2C'\eta^{k}\left\Vert w\right\Vert +C''2^{k}.
\]
The statement follows by changing $\left\Vert w\right\Vert $ back
to $|w|$. 

\end{proof}

\section{Volume growth estimates on $\Delta$ and proof of Theorem \ref{growth-pres}\label{sec:volume}}

In this section we estimate volume growth of the diagonal product
$\Delta\left(\mathcal{L}\right)=\Delta\left(\mathcal{L},(012)^{\infty}\right)$
defined in the Section \ref{sec:construction}. The main point is
that volume growth of $\Delta$ is controlled by the traverse field
(described in Section \ref{sec:tr-metric}) and growth in the groups
$\mathcal{L}=(F_{n})_{n=1}^{\infty}$. Our estimates on $\Delta$
are summarized in Theorem \ref{thm:main} which is stated and proved
in Subsection \ref{subsec:main-estimate}. 

Theorem \ref{thm:main} and the flexibility of choices of $\mathcal{L}$
in the construction allow us to establish Theorem \ref{growth-pres}.
Similar estimates can be shown for $\Delta\left(\mathcal{L},\omega\right)$,
where $\omega$ is a string such that all three letters $0,1,2$ appear
infinitely often. The change needed is to apply results in \cite{BE2}
and replace the contraction rate $\eta$ in $\mathfrak{G}$ with a
sequence of contraction rates $\eta_{1},\eta_{2},\ldots$ associated
with $\omega$. Since Theorem \ref{thm:main} is sufficient for our
purposes, we do not write the estimates for more general $\omega$
which would involve heavier notations.

\subsection{Growth estimates in each factor group}

In this subsection we focus on one factor group in the diagonal product
$\Delta(\mathcal{L})$. Let $F=\left\langle U,V\right\rangle $ be
a group generated by two finite subgroups $U,V$. Denote by $v_{F}$
the volume growth function of $F$ with respect to the generating
set $U\cup V$. 

Let $n\in\mathbb{N}$, consider the permutation wreath product $\Delta_{n}=F\wr_{\mathsf{L}_{n}}\mathfrak{G}$
marked with the generating set $T_{n}$ (\ref{eq:Tn}). Recall that
$T_{n}=\left(a,b,,c,d,u_{1,n},\ldots,u_{p,n},v_{1,n},\ldots,v_{q,n}\right)$.
Note that in $\Delta_{n}$, $u_{i,n}$ and $v_{j,n}$ commute, for
$1\le i\le p$, $q\le j\le q$. 

Consider a reduced word $w$ in the free product $\mathbf{M}$ as
in Subsection \ref{subsec:construction}. Denote by $\hat{\omega}$
the word obtained from $w$ by deleting letters in $\left\{ u_{1},\ldots,u_{p},v_{1},\ldots,v_{q}\right\} $.
That is, the word $w$ is of the form $w=w_{1}z_{1}w_{2}z_{2}\ldots w_{k}z_{k}$,
where each $w_{i}$ is a word in $\{a,b,c,d\}$ and $z_{i}\in U\ast V$.
And the corresponding $\hat{w}=w_{1}w_{2}\ldots w_{k}$. 

Now evaluate the word $w=w_{1}z_{1}w_{2}z_{2}\ldots w_{k}z_{k}$ in
$\left(\Delta_{n},T_{n}\right)$. Denote by $\pi_{U}$ the projection
$U\ast V\to U=\left\{ id,u_{1,n},\ldots,u_{p,n}\right\} $ and $\pi_{V}$
the projection $U\ast V\to V=\left\{ id,v_{1,n},\ldots,v_{q,n}\right\} $.
Denote the image of $w$ in $\Delta_{n}$ by $\left(\left(f_{w}(x)\right)_{x\in\mathsf{L}_{n}},\pi(w)\right)$,
where $\pi$ is the marked projection $\mathbf{M}\to\mathfrak{G}$.
By the multiplication rule in the permutation wreath product, at each
point $x\in\mathsf{L}_{n}$, the lamp configuration is the ordered
product
\begin{equation}
f_{w}(x)=\prod_{i=1}^{k}\left(\pi_{U}\left(z_{i}\right)\mathbf{1}_{\left\{ 1^{n}\cdot(w_{1}\ldots w_{i})^{-1}=x\right\} }+\pi_{V}(z_{i})\mathbf{1}_{\left\{ 1^{n-1}0\cdot\pi(w_{1}\ldots w_{i})^{-1}=x\right\} }\right),\label{eq:lampproduct}
\end{equation}
with the convention that if $x\notin\left\{ 1^{n}\cdot(w_{1}\ldots w_{i})^{-1},1^{n-1}0\cdot(w_{1}\ldots w_{i})^{-1}\right\} $,
then the \emph{i}-th factor is identity. From (\ref{eq:lampproduct}),
we have that the length of the word $f_{w}(x)$ in the letters $U\cup V$
is dominated by the number of alternating visits of the inverted orbit
of $1^{m}$ and $1^{m-1}0$ under $w_{1}\ldots w_{k}$, which is recorded
in the traverse field $P(x,w)$. In particular, we have
\begin{equation}
\left|f_{w}(x)\right|{}_{U\cup V}\le\left|P(x,\hat{w})\right|.\label{eq:length-T}
\end{equation}

We now proceed to show an upper bound of volume growth of $\left(\Delta_{n},T_{n}\right)$
in terms of the quantity $A(k,w)=\sum_{x\in\mathsf{L}_{k}}P(x,w)$
defined in (\ref{eq:A-1}). Denote $\mathcal{A}(n,r)$ the maximum
\[
\mathcal{A}(n,r):=\max\left\{ A(n,w):\ w\mbox{ is a word in }\left\{ a,b,c,d\right\} \mbox{ of length }\le r\right\} .
\]

\begin{lemma}[Growth upper bound]\label{V-A}

The volume function of $(\Delta_{n},T_{n})$ satisfies:
\begin{itemize}
\item for $r\le\left(2/\eta\right)^{n}$, $\log v_{\Delta_{n},T_{n}}(r)\le Cn^{\alpha_{0}}$,
where $C$ only depends on $|U|+|V|$;
\item suppose $\Phi_{F}$ is a function such that 
\[
v_{F,U\cup V}(n)\le\exp\left(\Phi_{F}(n)\right)\mbox{ for all }n\mbox{ and }\Phi_{F}(x)\mbox{ is non-decreasing and concave}.
\]
Then we have for $r>\left(2/\eta\right)^{n}$,
\[
\log\left|B_{\Delta_{n}}(id,r)\cap\ker(\Delta_{n}\to\mathfrak{G})\right|\le2^{n}\left(\Phi_{F}\left(\frac{\mathcal{A}(n,r)}{2^{n}}\right)+\log\left(1+\frac{\mathcal{A}(n,r)}{2^{n}}\right)+1\right).
\]
\end{itemize}
\end{lemma}

\begin{proof}

For $r<(2/\eta)^{n}$, write $k=\left\lfloor \log_{2/\eta}r\right\rfloor $.
Regard $\Delta_{n}$ as the permutation wreath product on level $k$,
\[
\Delta_{n}\simeq\left(F\wr_{\mathsf{L}_{n-k}}G_{\mathfrak{s}^{k}\omega}\right)\wr_{\mathsf{L}_{k}}\pi_{k}(G_{\omega}),
\]
where $\omega=(012)^{\infty}$. Then by \cite[Lemma 5.1]{BE1}, bounding
the growth of $F\wr_{\mathsf{L}_{n-k}}G_{\mathfrak{s}^{k}\omega}$
by the exponential function, we have $v_{\Delta_{n},T_{n}}(r)\le\exp(C\eta^{k}r)\cdot\left|\pi_{k}(G_{\omega})\right|\le\exp\left(Cn^{\alpha_{0}}\right)$,
where $\alpha_{0}=\log2/\log(2/\eta)$. 

By \ref{eq:length-T}, if $|w|\le r$, then 
\[
\sum_{x\in\mathsf{L}_{n}}\left|f_{w}(x)\right|_{U\cup V}\le\sum_{x\in\mathsf{L}_{n}}\left|P(x,\hat{w})\right|=A(n,\hat{w})\le\mathcal{A}(n,r).
\]
It follows that the set $B_{\Delta_{n}}(id,r)\cap\ker(\Delta_{n}\to\mathfrak{G})$
in the set $B'(r)$ of elements $\left(\left(f(x)\right)_{x\in\mathsf{L}_{n}},id\right)$,
where $g\in B_{\mathfrak{G}}(id,r)$ and $\sum_{x\in\mathsf{L}_{n}}\left|f(x)\right|_{U\cup V}\le\mathcal{A}(n,r)$.
The size of $B'(r)$ is bounded by
\[
\left|B'(r)\right|\le\left(\begin{array}{c}
\mathcal{A}(n,r)+2^{n}\\
2^{n}
\end{array}\right)\cdot\exp\left(2^{n}\Phi_{F}\left(\frac{\mathcal{A}(n,r)}{2^{n}}\right)\right).
\]
The second item in the statement follows then by plugging in the bound
$\left(\begin{array}{c}
n\\
k
\end{array}\right)\le\left(ne/k\right)^{k}$. 

\end{proof}

\begin{remark}\label{V-A-remark}

When $F$ is finite, we may replace $\Phi_{F}$ by $\max\left\{ \Phi_{F},\log|F|\right\} $.
In particular, for large $r$ such that $\mathcal{A}(n,r)\ge2^{n}{\rm Diam}(F,U\cup V)$,
we simply take the total volume of $\oplus_{x\in\mathsf{L}_{n}}F$
and have
\[
\log v_{\Delta_{n},T_{n}}(r)\le2^{n}|F|+\log v_{\mathfrak{G,}S}(r).
\]

\end{remark}

By the contraction inequality in Lemma \ref{A-contract-1}, we have
\[
\mathcal{A}(n,r)\le C\left(\eta^{n}r+2^{n}\right).
\]
Next we explain that the inequality is nearly optimal. Consider the
substitution 
\[
\zeta:ab\mapsto abadac,\ ac\mapsto abab,\ ad\mapsto acac.
\]
Take $w_{n}=\zeta^{n}(ad)$. It is shown in \cite[Proposition 4.7]{BE1}
that the sequence $w_{j}=\zeta^{j}(ad)$ has asymptotically maximal
inverted orbit growth. Recall that the inverted orbit of $x$ under
a word $w=z_{1}\ldots z_{k}$ is defined as the set $\mathcal{O}(x,w)=\left\{ x,x\cdot z_{1}^{-1},\ldots,x\cdot(z_{1}\ldots z_{k})^{-1}\right\} $. 

\begin{fact}\label{zeta-word}

The inverted orbit of $1^{n}$ under $w_{n}$ contains $\mathsf{L}_{n}$.

\end{fact}

\begin{proof}

Let $w$ be a word in $\{a,b,c,d\}$ and write $\boldsymbol{\varphi}(w)=\left(w_{0},w_{1}\right)\varepsilon^{s}$.
Then similar to the proof of Lemma \ref{P-recursion}, we have 
\[
\mathcal{O}(1^{n},w)=0\mathcal{O}(1^{n-1},w_{0})\cup1\mathcal{O}(1^{n-1},w_{1}).
\]
Under the formal recursion we have $\varphi(w_{n})=\left(w_{n-1}^{-1},w_{n-1}\right)$.
Then the statement follows by induction on $n$.

\end{proof}

Since $\mathfrak{G}$ acts by tree automorphisms, Fact \ref{zeta-word}
implies that the inverted orbit of $1^{n-1}0$ under $w_{n}$ covers
$\mathsf{L}_{n}$ as well. Note also that the image of $w_{n}$ in
$\mathfrak{G}$ is in the level-$n$ stabilizer. It follows that for
the word $\underbrace{w_{n}w_{n}^{-1}...w_{n}w_{n}^{-1}}_{r}$, that
is, the concatenation of $w_{n}$ and its reverse word $w_{n}^{-1}$
for $r$ times, its traverse field at each point $x\in\mathsf{L}_{n}$
is of length at least $2r$. The length of $w_{n}$ is comparable
to $\left(2/\eta\right)^{n}$, see the proof of \cite[Proposition 4.7]{BE1}.
The existence of such words implies that there is a constant $C>0$
which doesn't depend on $n,r$, such that 
\[
\mathcal{A}\left(n,C\left(2/\eta\right)^{n}r\right)\ge2^{n}r.
\]
Using the word $w_{n}$ we can find explicitly distinct words in $\Delta_{n}$
within a distance to the identity. 

\begin{lemma}[Growth lower bound]\label{w-lower}

There exists a constant $C>0$ such that 
\begin{itemize}
\item for $r\le\left(2/\eta\right)^{n}$, $\log v_{\Delta_{n},T_{n}}(r)\ge\frac{1}{C_{0}}n^{\alpha_{0}}$;
\item for $r>\left(2/\eta\right)^{n}$, 
\[
\log v_{\Delta_{n},T_{n}}(r)\ge2^{n}\log v_{F,U\cup V}\left(\frac{\eta^{n}r}{C2^{n}}\right).
\]
\end{itemize}
\end{lemma}

\begin{proof}

Let $C$ be a constant such that $\left|w_{j}\right|\le C\left(2/\eta\right)^{j}$
for all $n\in\mathbb{N}$. 

For $r>\left(2/\eta\right)^{n}$, in $\Delta_{n}$ consider configurations
$\left(f_{x}\right)_{x\in\mathsf{L}_{n}}\in\oplus_{x\in\mathsf{L}_{n}}F$
such that $\left|f_{x}\right|_{U\cup V}\le\ell$. For each $f_{x}$,
fix a representing word $r_{x}$ in $U$ and $V$ of shortest length
for $f_{x}$. We explicitly find a word in $\mathbf{M}$ that writes
an element $\left(\left(f_{x}\right),id\right)$. 

Write $w_{n}=z_{1}z_{2}\ldots z_{k_{n}}$. Recall that the inverted
orbit of $w_{n}$ visits all points on $\mathsf{L}_{n}$. For $x\in\mathsf{L}_{n}$,
denote by $j_{x}$ the smallest index $j$ such that $1^{n}\cdot\left(z_{1}\ldots z_{j}\right)^{-1}=x$.
For $\ell=1$, for each $x$ such that $f_{x}\in\left\{ u_{1},\ldots u_{p}\right\} $,
insert the word $f_{x}$ after $x_{j_{x}}$; after the first round
of insertion, for each $x$ such that $f_{x}\in\left\{ v_{1},\ldots,v_{q}\right\} $,
insert the word $f_{x}$ after $z_{j_{\check{x}}}$, where $\check{x}$
is the sibling of $x$. Then by the multiplication formula (\ref{eq:lampproduct}),
we have that after the insertions the resulting word evaluate to $\left(\left(f_{x}\right),\pi(w_{n})\right)$
in $\Delta_{n}$. For general $\ell$, one can repeat this procedure
$\ell$ times, each time reduce $\max_{x\in\mathsf{L}_{n}}\left|f_{x}\right|_{U\cup V}$
by $1$. The total length of the word is bounded by $\ell\left(\left|w_{n}\right|+2^{n}\right)$.
Therefore for $\ell\in\mathbb{N}$, we have
\[
v_{\Delta_{n},T_{n}}\left(\ell\left(\left|w_{n}\right|+2^{n}\right)\right)\ge\left(v_{F,U\cup V}(\ell)\right)^{2^{n}}.
\]

For $r<(2/\eta)^{n}$, write $j=\left\lfloor \log_{2/\eta}r\right\rfloor $.
Then by argument in the previous paragraph for $\ell=1$ with $w_{n}$
replaced by $w_{j}$ shows that 
\[
v_{\Delta_{n},T_{n}}\left(\left|w_{j}\right|+2^{j}\right)\ge2^{j}.
\]
The statement follows. 

\end{proof}

The estimates on one factor group $\Delta_{n}$ are summarized in
the following. We use the notation that for two positive numbers $x,y$,
$x\asymp_{C}y$ if $y/C\le x\le Cy$. 

\begin{proposition}[Piecewise description of the growth function of $(\Delta_n,T_n)$]\label{growth-factor}

Suppose that the growth function $v_{F,U\cup V}$ satisfies that there
exists a constant $C_{0}>1$ and a non-decreasing concave function
$\Phi_{F}$ such that for $1\le\ell\le{\rm Diam}(F,U\cup V)$, 
\[
\log v_{F,U\cup V}(\ell)\asymp_{C_{0}}\Phi_{F}(\ell)\mbox{ and }\Phi_{F}(\ell)\ge\frac{1}{C_{0}}\log(1+\ell),\ \Phi_{F}(1)=1.
\]
There exists a constant $C>1$, which only depends on $C_{0}$, such
that the volume growth function of $\left(\Delta_{n},T_{n}\right)$
satisfies: 
\begin{description}
\item [{(i)}] For $r\le\left(2/\eta\right)^{n}$, 
\[
\log\left|v_{\Delta_{n},T_{n}}(n)\right|\asymp_{C}n^{\alpha_{0}}.
\]
\item [{(ii)}] For $\left(2/\eta\right)^{n}\le r<\left(2/\eta\right)^{n}{\rm Diam}(F,U\cup V)$,
\[
\log\left|v_{\Delta_{n},T_{n}}(n)\right|\asymp_{C}2^{n}v_{F,U\cup V}\left(\eta^{n}r\right)+\log v_{\mathfrak{G,}S}(r).
\]
\item [{(iii)}] In the case that $F$ is a non-trivial finite group, for
$r\ge\left(2/\eta\right)^{n}{\rm Diam}(L,U\cup V)$,
\[
\log\left|v_{\Delta_{n},T_{n}}(n)\right|\asymp_{C}2^{n}\log|F|+\log v_{\mathfrak{G,}S}(r).
\]
\end{description}
\end{proposition}

\begin{proof}

The statement (i) follows from the first items in Lemma \ref{V-A}
and Lemma \ref{w-lower}. Statement (ii) follows from Lemma \ref{A-contract-1},
the second items in Lemma \ref{V-A} and Lemma \ref{w-lower}, and
the assumption on $v_{F,U\cup V}$. Statement (iii) follows from Remark
\ref{V-A-remark} and the second item of Lemma \ref{w-lower}.

\end{proof}

\begin{remark}

Recall that the balls of radius $2^{n-1}-1$ around identities in
$\left(\Delta_{n},T_{n}\right)$ and $(W,T)$ are identical. Proposition
\ref{growth-factor} (i) shows that the growth functions of $\left(\Delta_{n},T_{n}\right)$
and $(W,T)$ remain equivalent up to the radius $\left(2/\eta\right)^{n}$,
although balls are no longer identical. The estimates show that the
growth in the lamp groups becomes more visible only after the radius
$\left(2/\eta\right)^{n}$, in the way specified in (ii). 

\end{remark}

\subsection{Growth estimates for $\Delta$ \label{subsec:main-estimate}}

We have proved in the previous subsection that if the growth functions
of the lamp groups in $\mathcal{L}$ satisfies the assumption of Proposition
\ref{growth-factor}, then the growth function of $\left(\Delta_{n},T_{n}\right)$
can be estimated with good precision. We now proceed to estimate the
growth function of $\left(\Delta,T\right)$. To this end we impose
the following uniform assumption on the sequence $\mathcal{L}$.

\begin{assumption}\label{L-growth-assum}

Suppose $\mathcal{L}=\left(F_{n}\right)_{n=1}^{\infty}$ is a sequence
of marked quotients of $U\ast V$ which satisfies the following conditions.
There exists a non-decreasing concave function $\Phi:\mathbb{R}_{\ge0}\to\mathbb{R}_{\ge0}$,
$\Phi(1)=1$, and a constant $C_{0}>1$, such that for any $n\in\mathbb{N}$
and $0\le\ell\le{\rm Diam}(F_{n},U\cup V)$, we have $\Phi(\ell)\ge\frac{1}{C_{0}}\log\left(1+\ell\right)$,
and
\[
\log v_{F_{n},U\cup V}(\ell)\asymp_{C_{0}}\Phi(\ell).
\]

\end{assumption}

Note that in Assumption \ref{L-growth-assum}, the groups $F_{n}$
can be finite or infinite. Also the sequence $\left(F_{n}\right)_{n=1}^{\infty}$
does not need to have monotone properties, for example, it could be
that on an infinite set of indices, $F_{n}=\{id\}$. 

\begin{example}[Linear $\Phi$]\label{linear-ex}

Suppose $\mathcal{L}=\left(F_{n}\right)_{n=1}^{\infty}$ is a sequence
of expanders, that is, there exists a constant $\lambda_{0}>0$, such
that for any $n\in\mathbb{N}$, and $A\subseteq F_{n}$ with $|A|\le\frac{|F_{n}|}{2}$,
we have 
\begin{equation}
\frac{\left|\partial_{U\cup V}A\right|}{\left|A\right|}\ge\lambda_{0},\label{eq:expansion}
\end{equation}
where $\partial_{U\cup V}A=\{x\in A:\exists s\in U\cup V,\ xs\notin A\}$.
Then expansion in (\ref{eq:expansion}) implies that 
\[
v_{F_{n},U\cup V}(r)\ge\left(1+\lambda_{0}\right)^{r}\ \mbox{for all }r\mbox{ s.t.\ }\left|B_{F_{n}}(id,r-1)\right|\le\frac{|F_{n}|}{2}.
\]
It follows that if $\mathcal{L}=\left(F_{n}\right)_{n=1}^{\infty}$
is a sequence of expanders, where each $F_{n}$ is a finite marked
quotient of $U\ast V$, then $\mathcal{L}$ satisfies Assumption \ref{L-growth-assum}
with $\Phi(x)=x$ and a constant $C_{0}$ that only depends on $\lambda_{0}$
and $|U|+|V|$. 

\end{example}

We now put combine the bounds in the individual factors to give estimates
on the growth function of the diagonal product $\Delta=\Delta(\mathcal{L},(012)^{\infty})$.
Given $r\in\mathbb{N}$, denote by $J_{r}$ the index set
\[
J_{r}(\mathcal{L}):=\left\{ j\in\mathbb{N}:\ r\ge\left(\frac{2}{\eta}\right)^{j}\mbox{ and }F_{j}\neq\{id\}\right\} .
\]
Write $k_{r}=\left\lceil \log_{2/\eta}r\right\rceil $. The key observation
in the upper bound direction is to regard $\Delta$ as the diagonal
product of $\left(\Delta_{n},T_{n}\right)_{n=1}^{k_{r}}$ and $\left(\Delta_{>k_{n}},T\right)$,
and then apply Proposition \ref{growth-factor} to each of these factors.
The lower bound simply comes from taking the maximum of the growth
functions of the factors. 

\begin{theorem}[Volume growth of $\Delta$]\label{thm:main}

Suppose $\mathcal{L}=\left(F_{n}\right)_{n=1}^{\infty}$ is a sequence
of marked quotients of $U\ast V$ that satisfies Assumption \ref{L-growth-assum}
with function $\Phi$ and constant $C_{0}>0$. Let $\Delta=\Delta\left(\mathcal{L},(012)^{\infty}\right)$
be the diagonal product defined in the Section \ref{sec:construction}.
Then there is a constant $C>0$ which only depends on $C_{0}$ such
that 

\[
\log v_{\Delta,T}(r)\le C\sum_{j\in J_{r}(\mathcal{L})}2^{j}\Phi\left(\min\left\{ \left(\frac{\eta}{2}\right)^{j}r,{\rm Diam}\left(F_{j},U\cup V\right)\right\} \right)+Cr^{\alpha_{0}}.
\]
and

\[
\log v_{\Delta,T}(r)\ge\frac{1}{C}\max_{j\in J_{r}(\mathcal{L})}2^{j}\Phi\left(\min\left\{ \left(\frac{\eta}{2}\right)^{j}r,{\rm Diam}\left(F_{j},U\cup V\right)\right\} \right)+\frac{1}{C}r^{\alpha_{0}}.
\]

\end{theorem}

\begin{proof}

We first prove the volume upper bound. Denote by $N=\ker\left(\Delta\to W\right)$.
Since $v_{\Delta,T}(r)\le\left|B_{\Delta}(id,r)\cap N\right|\cdot v_{W,T_{0}}(r)$,
and $\log v_{W,T_{0}}(r)\asymp n^{\alpha_{0}}$ by \cite{BE1}, it
suffices to show the upper bound for $\left|B_{\Delta}(id,r)\cap N\right|$.
Regard $\Delta$ as the diagonal product of $\left(\Delta_{n},T_{n}\right)_{n=1}^{k_{r}}$
and $\left(\Delta_{>k_{n}},T\right)$. Note that if $\gamma\in N$
and $F_{j}=\{id\}$, then the projection of $\gamma$ to $\Delta_{j}=\mathfrak{G}$
is trivial. Thus we only need to count in factors indexed by $J_{r}(\mathcal{L})$
and $\Delta_{>k_{r}}$. For $j\in J_{r}(\mathcal{L})$, we invoke
Assumption \ref{L-growth-assum} and apply the second item in Lemma
\ref{V-A} and Lemma \ref{A-contract-1}, then
\[
\log\left|B_{\Delta_{j}}(id,r)\cap\ker(\Delta_{j}\to\mathfrak{G})\right|\le C'2^{j}\Phi\left(\min\left\{ \left(\frac{\eta}{2}\right)^{j}r,{\rm Diam}\left(F_{j},U\cup V\right)\right\} \right),
\]
where $C'$ only depends on $C_{0}$. For the factor $\Delta_{>k_{r}}$,
recall that it is isomorphic to a subgroup of the permutation wreath
product $\Delta(\mathfrak{s}^{k_{r}}\mathcal{L},\mathfrak{s}^{k_{r}}(012)^{\infty})\wr_{\mathsf{L}_{k_{r}}}\pi_{k_{r}}(\mathfrak{G})$.
Therefore by Proposition \ref{growth-factor} (i), we have that 
\[
v_{\Delta_{>k_{r}},T}(r)\le C''r^{\alpha_{0}},
\]
where $C''$ is an absolute constant. Therefore 
\begin{align*}
\log\left|B_{\Delta}(id,r)\cap N\right| & \le\sum_{j\in J_{r}}\log\left|B_{\Delta_{j}}(id,r)\cap\ker(\Delta_{j}\to\mathfrak{G})\right|+v_{\Delta_{>k_{r}},T}(r)\\
 & \le C'\sum_{j\in J_{r}}2^{j}\Phi\left(\min\left\{ \left(\frac{\eta}{2}\right)^{j}r,{\rm Diam}\left(F_{j},U\cup V\right)\right\} \right)+C''r^{\alpha_{0}}.
\end{align*}
We have proved the upper bound.

To show the lower bound, first note that since $W$ is a marked quotient
of $\Delta$, we have $v_{\Delta,T}(n)\ge v_{W,T_{0}}(n)$. By the
second item in Lemma \ref{w-lower}, we have that for each $j\in J_{r}$,
\begin{align*}
\log v_{\Delta_{j},T_{j}}(r) & \ge2^{j}\log v_{F_{j},U\cup V}\left(\left(\frac{\eta}{2}\right)^{j}\frac{r}{C}\right)\\
 & \ge\frac{1}{C_{0}}2^{j}\Phi\left(\min\left\{ \left(\frac{\eta}{2}\right)^{j}\frac{r}{C},{\rm Diam}\left(F_{j},U\cup V\right)\right\} \right)\\
 & \ge\frac{1}{C_{0}C}2^{j}\Phi\left(\min\left\{ \left(\frac{\eta}{2}\right)^{j}r,{\rm Diam}\left(F_{j},U\cup V\right)\right\} \right).
\end{align*}
Since $v_{\Delta,T}\ge\max\left\{ v_{\Delta_{j,}T_{j}},j\in J_{r},v_{W,T_{0}}\right\} $,
the statement follows. 

\end{proof}

\subsection{Approximations of prescribed functions\label{subsec:approximations}}

In this subsection, let $f:\mathbb{R}_{+}\to\mathbb{R}_{+}$ be a
continuous function that is non-decreasing, subadditive and $f(x)/x\to0$
when $x\to\infty$. Suppose in addition that there is a constant $\lambda>2$
such that $f(x)\ge x^{\alpha}$ for all $x>0$, where $\alpha=\frac{1}{\log_{2}\lambda}$.
We describe a procedure to approximate $f$ by quantities that appear
in the bounds of Theorem \ref{thm:main}. 

Given such a function $f$, define recursively the following two sequences.
Let $\theta_{0}=0$, $m_{0}=0$. For $j\in\mathbb{N}$, define 
\begin{align*}
m_{j}' & =\min\left\{ m>m_{j-1}:\frac{f\left(\lambda^{m}\right)}{\lambda^{m}}\le\left(2/\lambda\right)^{\theta_{j-1}+1}\right\} \\
m_{j}'' & =\min\left\{ m>m_{j-1}:f\left(\lambda^{m}\right)\ge2f\left(\lambda^{m_{j-1}}\right)\right\} ,\\
m_{j} & =\max\left\{ m_{j}',m_{j}''\right\} ;
\end{align*}
and 
\[
\theta_{j}:=\log_{2/\lambda}\left(\frac{f\left(\lambda^{m_{j}}\right)}{\lambda^{m_{j}}}\right).
\]
Although suppressed in the notation, the sequences $(m_{j})$ and
$\left(\theta_{j}\right)$ are determined by $f$ and $\lambda$.
Define the function $\phi:\{\theta_{1},\theta_{2}\ldots\}\to\mathbb{R}_{+}$
to be 
\[
\phi(\theta_{j})=\lambda^{m_{j}-\theta_{j}}.
\]
The choice of these sequences is to guarantee the following approximation. 

\begin{lemma}\label{approximation}

There is a constant $c_{\lambda}>0$ which only depends on $\lambda$,
such that for $j\in\mathbb{N}$, and any $x\in\left[\lambda^{m_{j-1}},\lambda^{m_{j}}\right)$,
\[
c_{\lambda}\sum_{i:\ \lambda^{\theta_{i}}\le x}2^{j}\min\left\{ \lambda^{-\theta_{i}}x,\phi(\theta_{i})\right\} \le f(x)\le\lambda\max_{i:\ \lambda^{\theta_{i}}\le x}2^{\theta_{i}}\min\left\{ \lambda^{-\theta_{i}}x,\phi(\theta_{i})\right\} .
\]

\end{lemma}

\begin{proof}

For $i,m$ such that $\theta_{i}\le m$, denote by $A_{i,m}$ the
quantity
\[
A_{i,m}=2^{\theta_{i}}\min\left\{ \lambda^{m-\theta_{i}},\phi(\theta_{i})\right\} ,
\]
for $i$ such that $\theta_{i}>m$, set $A_{i.m}=0$. Let $m\in\left[m_{j-1},m_{j}\right)$.
Note that by definitions, we have that $\theta_{j-1}<m$; moreover,
for any $i\le j-1$, $\phi(\theta_{i})=\lambda^{m_{i}-\theta_{i}}\le\lambda^{m-\theta_{i}}$;
for $i\ge j$, $\phi(\theta_{i})=\lambda^{m_{i}-\theta_{i}}>\lambda^{m-\theta_{i}}$.
Therefore for $m\in\left[m_{j-1},m_{j}\right)$,
\[
A_{i,m}=\begin{cases}
2^{\theta_{i}}\phi\left(\theta_{i}\right)=2^{\theta_{i}}\lambda^{m_{i}-\theta_{i}} & \mbox{if }i\le j-1\\
\left(2/\lambda\right)^{\theta_{i}}\lambda^{m} & \mbox{if }i\ge j,\theta_{i}\le m.
\end{cases}
\]

\begin{claim}

The sequence $\left(\phi\left(\theta_{i}\right)\right)_{i=1}^{\infty}$
is non-decreasing. 

\end{claim}

\begin{proof}[Proof of the Claim]

Note that by the definitions, we have 
\[
\phi\left(\theta_{i}\right)=\lambda^{m_{i}-\theta_{i}}=f\left(\lambda^{m_{i}}\right)/2^{\theta_{i}}.
\]
Consider the following two cases.

Case 1: $m_{i}'>m_{i}''$. In this case $m_{i}=m_{i}'$ and $f\left(\lambda^{m_{i}}\right)/\lambda^{m_{i}}=(2/\lambda)^{\theta_{i-1}+1}$,
$\theta_{i}=\theta_{i-1}+1$. Therefore 
\[
\phi\left(\theta_{i}\right)=f\left(\lambda^{m_{i}}\right)/2^{\theta_{i}}=f\left(\lambda^{m_{i}}\right)/2^{\theta_{i-1}+1}\ge f\left(\lambda^{m_{i-1}}\right)/2^{\theta_{i-1}}=\phi(\theta_{i-1}).
\]

Case 2: $m_{i}'\le m_{i}''$. In this case $m_{i}=m_{i}''$ and $f\left(\lambda^{m_{i}}\right)=2f\left(\lambda^{m_{i}-1}\right)$,
$\theta_{i}\ge\theta_{i-1}+1$. Therefore 
\[
\phi\left(\theta_{i}\right)=f\left(\lambda^{m_{i}}\right)/2^{\theta_{i}}=2f\left(\lambda^{m_{i-1}}\right)/2^{\theta_{i}}\ge2f\left(\lambda^{m_{i-1}}\right)/2^{\theta_{i-1}+1}=\phi(\theta_{i-1}).
\]

\end{proof}

We now return to the proof of the lemma. Let $m=\log_{\lambda}x$.
The summation on the left-hand side, over $i$ such that $\theta_{i}\le m$,
splits into two parts, $i\ge j-1$ and $i\ge j$ (the second part
might be empty). Since by the definitions $\theta_{i}\ge\theta_{i-1}+1$
and by the Claim $\phi\left(\theta_{i}\right)$ is non-decreasing,
we have that $\max_{i\le j-1}A_{i,m}=A_{j-1,m},$ 
\[
\sum_{i\le j-1}A_{i,m}=\sum_{i\le j-1}2^{\theta_{i}}\phi\left(\theta_{i}\right)\le2^{\theta_{j-1}+1}\phi\left(\theta_{j-1}\right)=2(2/\lambda)^{\theta_{j-1}}\lambda^{m_{j-1}}=2A_{j-1,m}.
\]
If the set $\{i:\theta_{i}\le m,i\ge j\}$ is non-empty (equivalently,
$\theta_{j}\le m$), then we have a decreasing geometric sum 
\[
\sum_{i\ge j,\ \theta_{i}\le m}A_{i,m}=\sum_{i\ge j,\ \theta_{i}\le m}\left(2/\lambda\right)^{\theta_{i}}\lambda^{m}\le\frac{1}{1-2/\lambda}\left(2/\lambda\right)^{\theta_{j}}\lambda^{m}=\frac{1}{1-2/\lambda}A_{j,m}.
\]

Similar to the discussion in the proof of the Claim, we consider two
cases. Case 1 $m_{i}'>m_{i}''$. In this case, $\theta_{j}=\theta_{j-1}+1$
and 
\[
(2/\lambda)^{\theta_{j}-1}\ge f\left(\lambda^{m}\right)/\lambda^{m}>(2/\lambda)^{\theta_{j}}.
\]
Then we have the upper bound
\[
f(\lambda^{m})\le\lambda^{m}\left(2/\lambda\right)^{\theta_{j-1}}\le\lambda\max_{i:\theta_{i}\le m}A_{i,m};
\]
and the lower bound
\[
f(\lambda^{m})\ge\lambda^{m}\left(2/\lambda\right)^{\theta_{j}}\ge\min\left\{ \left(\frac{1}{2}-\frac{1}{\lambda},\frac{1}{2\lambda}\right)\right\} \sum_{i:\theta_{i}\le m}A_{i,m}.
\]

Case 2: $m_{i}'>m_{i}''$. In this case 
\[
f\left(\lambda^{m_{j-1}}\right)\le f\left(\lambda^{m}\right)\le2f\left(\lambda^{m_{j-1}}\right),
\]
and 
\[
\frac{A_{j,m}}{A_{j-1.m}}\le\frac{\lambda^{m_{j}-\theta_{j}}2^{\theta_{j}}}{\lambda^{m_{j-1}-\theta_{j-1}}2^{\theta_{j-1}}}=\frac{\lambda^{-\theta_{j}}f(\lambda^{m_{j}})}{\lambda^{-\theta_{j-1}}f(\lambda^{m_{j-1}})}=2\lambda^{-\theta_{j}+\theta_{j-1}}\le\frac{2}{\lambda}.
\]
Then we have the upper bound 
\[
f(\lambda^{m})\le2f\left(\lambda^{m_{j-1}}\right)=2(2/\lambda)^{\theta_{j-1}}\lambda^{m_{j-1}}\le2\max_{i:\theta_{i}\le m}A_{i,m};
\]
and the lower bound 
\[
f(\lambda^{m})\ge f\left(\lambda^{m_{j-1}}\right)=A_{j-1,m}\ge\frac{1}{2}\min\left\{ \frac{1}{2},\frac{\lambda}{2}-1\right\} \sum_{i:\theta_{i}\le m}A_{i,m}.
\]
The statement of the lemma follows by combing these two cases, where
$c_{\lambda}=\min\left\{ \frac{1}{2}-\frac{1}{\lambda},\frac{1}{2\lambda},\frac{\lambda}{2}-1\right\} $. 

\end{proof}

\subsection{Proof of Theorem \ref{growth-pres}}

The goal of this subsection is to apply Theorem \ref{thm:main} to
prove the prescribed growth function theorem \ref{growth-pres} stated
in the Introduction. 

First we fix a choice of an expander sequence $\left(\Gamma_{i}\right)_{i=1}^{\infty}$
such that $\Gamma_{n}$ is a marked quotient of $U\ast V$, both $U$
and $V$ are abelian and there is a constant $K_{0}$ such that 
\[
{\rm Diam}(\Gamma_{n},U\cup V)\asymp_{K_{0}}n.
\]
Denote by $\delta_{0}>0$ the lower bound for Cheeger constants of
$\left(\Gamma_{i}\right)_{i=1}^{\infty}$. Such expander sequences
exist: for example, by the Margulis construction, one can start with
a group with Kazhdan's property (T) which is generated by two finite
abelian subgroups and take a sequence of its finite quotients. For
an explicit sequence of such expanders, see e.g., \cite[Example 2.3]{BZ1}.
Alternatively, one can take $\Gamma_{i}={\rm PSL}_{2}(\mathbb{Z}/5^{i}\mathbb{Z})$.
There exists an explicit marking of $\Gamma_{i}$ with $U=\mathbb{Z}/2\mathbb{Z}$
and $V=\mathbb{Z}/2\mathbb{Z}\times\mathbb{Z}/2\mathbb{Z}$; and $\left(\left(\Gamma_{i},U\cup V\right)\right)_{i=1}^{\infty}$
forms a sequence of expanders, see \cite[Lemma 6.1]{KassabovPak}. 

Given a function $f$ satisfying the assumptions of Theorem \ref{growth-pres},
the strategy is to find a sequence $\mathcal{L}=\left(F_{n}\right)$
such that the log-volume bounds for $\Delta(\mathcal{L})$ in Theorem
\ref{thm:main} approximate the function $f$ up to a fixed constant.
The choice of parameters is provided by the approximation lemma \ref{approximation}.
Namely, we choose $F_{n}$ to be nontrivial groups only along the
sequence $(\theta_{j})$, and at the level $\left\lfloor \theta_{j}\right\rfloor $,
we take the corresponding lamp group to have diameter comparable to
$\phi\left(\theta_{j}\right)$.

\begin{proof}[Proof of Theorem \ref{growth-pres}]

Take $\Gamma_{i}={\rm PSL}_{2}(\mathbb{Z}/5^{i}\mathbb{Z})$ and choose
the marking on $\Gamma_{i}$ by $U=\mathbb{Z}/2\mathbb{Z}$ and $V=\mathbb{Z}/2\mathbb{Z}\times\mathbb{Z}/2\mathbb{Z}$
as in \cite[Lemma 6.1]{KassabovPak}. Then there is a constant $K_{0}$
such that ${\rm Diam}(\Gamma_{i},U\cup V)\asymp_{K_{0}}i$ for all
$i$ and $\left(\left(\Gamma_{i},U\cup V\right)\right)_{i=1}^{\infty}$
forms a sequence of expanders. Denote by $\delta_{0}$ the infimum
of the Cheeger constants of $\left(\Gamma_{i},U\cup V\right)$. 

Given a function $f:\mathbb{N\to N}$ satisfying the assumptions of
the statement, we may continue it to a sub-additive continuous function
$f:\mathbb{R}_{\ge0}\to\mathbb{R}_{\ge0}$ such that $f(x)\ge x^{\alpha_{0}}$
for all $x\ge0$. Recall the contraction ratio $\eta$ as in Theorem
\ref{thm:main} and the fact that $\alpha_{0}=1/\log_{2}(2/\eta)$.
Let $\lambda=2/\eta$. Recall the sequences $(m_{j})$, $(\theta_{j})$
and $\left(\phi(\theta_{j})\right)$ associated with $f$ and $\lambda$
defined in Subsection \ref{subsec:approximations}. Note that the
assumption $f(x)\ge x^{\alpha_{0}}$ implies that $\phi(\theta_{j})\ge1$
for all $j\in\mathbb{N}$.

Now we choose the sequence $\mathcal{L}=(F_{n})_{n=1}^{\infty}$ by
\begin{align*}
F_{\left\lfloor \theta_{j}\right\rfloor } & =\Gamma_{\left\lfloor \phi\left(\theta_{j}\right)\right\rfloor }\mbox{ for }j\in\mathbb{N};\\
F_{n} & =\{id\}\mbox{ for }n\notin\left\{ \left\lfloor \theta_{1}\right\rfloor ,\left\lfloor \theta_{2}\right\rfloor ,\ldots\right\} .
\end{align*}
Take the diagonal product $\Delta=\Delta\left(\mathcal{L},(012)^{\infty}\right)$
as defined in Section \ref{subsec:construction}. 

As explained in Example \ref{linear-ex}, the sequence $\mathcal{L}$
defined above satisfies Assumption \ref{L-growth-assum} with $\Phi(x)=x$
and constant $C_{0}$ which only depends on the constant $K_{0}$
and $\delta_{0}$ associated with the expander sequence $\left(\Gamma_{i}\right)$.
Therefore by Theorem \ref{thm:main}, there is a constant $C$, which
only depends on $K_{0}$ and $\delta_{0}$, such that 

\[
\log v_{\Delta,T}(r)\le C\sum_{j:\left\lfloor \theta_{j}\right\rfloor \le\log_{\lambda}r}2^{j}\min\left\{ \left(\frac{\eta}{2}\right)^{j}r,\phi(\theta_{j})\right\} +Cr^{\alpha_{0}}.
\]
and

\[
\log v_{\Delta,T}(r)\ge\frac{1}{C}\max_{j:\left\lfloor \theta_{j}\right\rfloor \le\log_{\lambda}r}2^{j}\min\left\{ \left(\frac{\eta}{2}\right)^{j}r,\phi(\theta_{j})\right\} +\frac{1}{C}r^{\alpha_{0}}.
\]
Then by the approximation lemma \ref{approximation}, we have that
for all $r\ge1$,
\begin{align*}
\log v_{\Delta,T}(r) & \le\frac{C}{c_{\lambda}}f(r)+Cr^{\alpha_{0}}\le\left(\frac{C}{c_{\lambda}}+C\right)f(r);\\
\log v_{\Delta,T}(r) & \ge\frac{1}{C\lambda}f(r)+\frac{1}{C}r^{\alpha_{0}}\ge\frac{1}{C\lambda}f(r).
\end{align*}
By Fact \ref{FC-delta2} and Corollary \ref{FC-abelianUV}, we have
that $\Delta$ is an FC-central extension of $W=(U\times V)\wr_{\mathcal{S}}\mathfrak{G}=(\mathbb{Z}/2\mathbb{Z})^{3}\wr_{\mathcal{S}}\mathfrak{G}$.
We have finished the proof of the statement. 

\end{proof}

\bibliographystyle{alpha}
\bibliography{FCgr}

\end{document}